\documentclass[english,12pt]{article}
\usepackage[textwidth=1.75in, textsize=footnotesize]{todonotes}
\presetkeys{todonotes}{fancyline, color=white}{}

\newif\ifofficial
\officialfalse
\ifofficial
  \usepackage[nodisplayskipstretch]{setspace}
  \usepackage[
    paperheight=11in,
    paperwidth=8.5in,
    top=1in,
    bottom=1.5in,
    right=1in,
    left=1.5in
  ]{geometry}
\else
  \usepackage[
    paperheight=11in,
    paperwidth=8.5in,
    top=0.65in,
    bottom=1in,
    right=1.5in,
    left=1.5in
  ]{geometry}
\fi

\usepackage{stmaryrd}
\usepackage{lmodern}
\usepackage{diagbox}
\usepackage{mathtools}
\usepackage{amssymb,amsmath}
\usepackage{graphicx}
\DeclareGraphicsExtensions{.png, .jpg}

\usepackage{tikz}
\usetikzlibrary{arrows.meta,arrows,positioning,shapes,fit,calc}
\pgfdeclarelayer{background}
\pgfsetlayers{background,main}

\usepackage{forest}

\usepackage{amsthm}
\begingroup
    \makeatletter
    \@for\theoremstyle:=definition,remark,plain\do{%
        \expandafter\g@addto@macro\csname th@\theoremstyle\endcsname{%
            \addtolength\thm@preskip\parskip
            }%
        }
\endgroup
  \newtheorem{definition}{Definition}[section]
  \newtheorem{lemma}{Lemma}[section]
  \newtheorem{notation}{Notation}[section]
  \newtheorem{theorem}{Theorem}[section]
  \newtheorem{proposition}{Proposition}[section]
  \newtheorem{corollary}{Corollary}[section]
\usepackage{thmtools}
\declaretheorem[style=definition,qed=$\triangle$,numberwithin=section]{example}
\declaretheorem[style=definition,qed=$\triangle$,numberwithin=section]{remark}
\numberwithin{equation}{section}

\usepackage{ifxetex,ifluatex}
\ifnum 0\ifxetex 1\fi\ifluatex 1\fi=0 % if pdftex
  \usepackage[T1]{fontenc}
  \usepackage[utf8]{inputenc}
\else % if luatex or xelatex
  \ifxetex
    \usepackage{mathspec}
  \else
    \usepackage{fontspec}
  \fi
  \defaultfontfeatures{Ligatures=TeX,Scale=MatchLowercase}
\fi
\IfFileExists{upquote.sty}{\usepackage{upquote}}{}
\IfFileExists{microtype.sty}{%
\usepackage{microtype}
\UseMicrotypeSet[protrusion]{basicmath} % disable protrusion for tt fonts
}{}
\usepackage{hyperref}
\hypersetup{unicode=true,
            pdftitle={Enumerative properties of restricted words and
              compositions},
            pdfborder={0 0 0},
            breaklinks=true}
\ifnum 0\ifxetex 1\fi\ifluatex 1\fi=0 % if pdftex
  \usepackage[shorthands=off,main=english]{babel}
\else
  \usepackage{polyglossia}
  \setmainlanguage[]{english}
\fi

\usepackage{mathrsfs}

\usepackage[parfill]{parskip}
\usepackage{titlesec}
\titleformat{\paragraph}
  {\normalfont\normalsize\bfseries}{\theparagraph}{1em}{}
  \titlespacing*{\paragraph}
  {0pt}{3.25ex plus 1ex minus .2ex}{1.5ex plus .2ex}

\usepackage{csquotes}
\usepackage{epigraph, varwidth}
\usepackage{datetime}
  \newdateformat{amdate}{\THEYEAR.\shortmonthname[\THEMONTH].\twodigit{\THEDAY}}
\usepackage[affil-it]{authblk}
\usepackage{bm}
\usepackage{xparse}
\DeclarePairedDelimiterX{\rvect}[1]{[}{]}{\,\makervect{#1}\,}
\ExplSyntaxOn
\NewDocumentCommand{\makervect}{m}
 {
  \seq_set_split:Nnn \l_tmpa_seq { , } { #1 }
  \begin{matrix}
  \seq_use:Nn \l_tmpa_seq { & }
  \end{matrix}
 }
\ExplSyntaxOff
\usepackage{titlesec}

\newif\iflocal
\IfFileExists{private/notes.md}{
  \localtrue
}{
  \localfalse
}

\renewcommand{\epigraphsize}{\small}
\renewcommand{\textflush}{flushright}
\renewcommand{\sourceflush}{flushright}
\newcommand{\epitextfont}{\itshape}
\newcommand{\episourcefont}{\scshape}

\makeatletter
\newsavebox{\epi@textbox}
\newsavebox{\epi@sourcebox}
\newlength\epi@finalwidth
\renewcommand{\epigraph}[2]{%
  \vspace{\beforeepigraphskip}
  {\epigraphsize\begin{\epigraphflush}
   \epi@finalwidth=\z@
   \sbox\epi@textbox{%
     \varwidth{\epigraphwidth}
     \begin{\textflush}\epitextfont#1\end{\textflush}
     \endvarwidth
   }%
   \epi@finalwidth=\wd\epi@textbox
   \sbox\epi@sourcebox{%
     \varwidth{\epigraphwidth}
     \begin{\sourceflush}\episourcefont#2\end{\sourceflush}%
     \endvarwidth
   }%
   \ifdim\wd\epi@sourcebox>\epi@finalwidth 
     \epi@finalwidth=\wd\epi@sourcebox
   \fi
   \leavevmode\vbox{
     \hb@xt@\epi@finalwidth{\hfil\box\epi@textbox}
     \vskip1.75ex
     \hrule height \epigraphrule
     \vskip.75ex
     \hb@xt@\epi@finalwidth{\hfil\box\epi@sourcebox}
   }%
   \end{\epigraphflush}
   \vspace{\afterepigraphskip}}}
\makeatother

\newcommand{\graphf}[1]{#1}

\newcommand{\cev}[1]{\reflectbox{\ensuremath{\vec{\reflectbox{\ensuremath{#1}}}}}}

\newcommand{\hype}{{\hbox{-}}}

\DeclareMathOperator{\folds}{folds}
\DeclareMathOperator{\merges}{merges}
\DeclareMathOperator{\gap}{gap}
\DeclareMathOperator{\red}{red}
\DeclareMathOperator{\Var}{Var}
\DeclareMathOperator{\spr}{spr}

\DeclareMathOperator{\fix}{fix}

\DeclareMathOperator{\acyc}{acyc}
\DeclareMathOperator{\lcm}{lcm}
\newcommand{\bigzero}{\mbox{\normalfont\Large\bfseries 0}}

\newcommand{\Zk}{\mathbb{Z}_k}

\newcommand{\SEQ}{\textsc{Seq}}

\newcommand{\disjun}{\mathbin{\dot{\cup}}}
\newcommand{\caret}{{}^{\wedge}}
\newcommand{\too}{\longrightarrow}

\begin{document}
\renewcommand\thmcontinues[1]{Continued}

\begin{titlepage}
\pagenumbering{gobble}
\begin{center}
{\LARGE
Enumerative properties of restricted words and compositions}

\vspace{1cm}
{\large Andrew MacFie}

\vspace{1cm}
\ifofficial
{\small
A thesis submitted to the Faculty of Graduate and Postdoctoral Affairs
in partial fulfillment of the requirements for the degree of}

{\large Doctor of Philosophy}\\
{\small in}\\
{\large Mathematics}

\vspace{2cm}
{\small
School of Mathematics and Statistics\\
Ottawa-Carleton Institute for Mathematics and Statistics\\
Carleton University\\
Ottawa, Ontario}

{©2018 Andrew MacFie}
\fi

\end{center}
\end{titlepage}

\pagenumbering{roman}
\setcounter{page}{2}

\ifofficial
\newpage
\emph{It is not knowledge, but the act of learning, not possession but the act of
getting there, which grants the greatest enjoyment.}\\
\begin{flushright}
-Carl Friedrich Gauss, 1808
\end{flushright}
\fi

\section*{Abstract}

Words and integer compositions are fundamental combinatorial objects.
In each case, the object is a finite sequence of terms over a particular set.
Relevant properties, sometimes called \enquote{parameters}, are the length of
the sequence and, for integer compositions, the sum of the sequence.

There has been interest within enumerative combinatorics in counting words
and compositions, especially restricted variations where the objects satisfy
extra conditions.
\enquote{Local} restrictions are related to contiguous subsequences, for
example Smirnov words where adjacent letters must be different.
For integer compositions or words over an ordered alphabet, a
\enquote{subword pattern avoidance} restriction requires all contiguous
subsequences of a fixed length to not satisfy a certain relative ordering.
For example, we may count compositions not containing a strictly increasing
contiguous subsequence of length three.
\enquote{Global} restrictions, on the other hand, are related to arbitrary
subsequences.
A \enquote{subsequence pattern avoidance} restriction requires all subsequences
of a fixed length to not satisfy a certain relative ordering.

Beyond sequences we may also consider objects with different structures, and
interpret local and global restrictions appropriately.
We say \enquote{cyclically restricted} finite sequences are those where the
last and first terms are considered adjacent for the purposes of the
restriction, i.e.\ the restriction wraps around from the end to the start.
\enquote{Circular} objects are the orbits of finite sequences under circular
shifts, so all circular shifts of a finite sequence are considered the same
object.

We can generalize integer compositions by replacing the semigroup of
positive integers with a different additive semigroup, giving the broader
concept of a \enquote{composition over a semigroup}, i.e.\ a finite sequence
with a certain sum over the semigroup.
Beyond the positive integers,
we focus on semigroups which are finite groups -- where such
compositions are in fact also \enquote{words} in the group theory sense.
Compositions over a finite group are relatively little-studied in combinatorics
but turn out to be amenable to combinatorial analysis in analogy to both words
and integer compositions.

In this document we achieve exact and asymptotic enumeration of
words,
compositions over a finite group,
and/or
integer compositions
characterized by local
restrictions and, separately, subsequence pattern avoidance.
We also count cyclically restricted and circular objects.
This either fills gaps in the current literature by e.g.\ considering
particular new patterns, or involves general progress, notably with locally
restricted compositions over a finite group.
We associate these compositions to walks on a
covering graph whose structure is exploited to simplify asymptotic expressions.
Specifically, we show that under certain conditions the number of locally
restricted compositions of a group element is asymptotically independent of
the group element.
For some problems our results extend to the case of a positive number of
subword pattern occurrences (instead of zero for pattern avoidance) or
convergence in distribution of the normalized number of occurrences.
We typically apply the more general propositions to concrete examples such as
the familiar Carlitz compositions or simple subword patterns.
%For subsequence pattern avoidance, we focus on pairs of generalized patterns,
%and partially-ordered patterns with two distinct letters.

% not sure what the ideal length of this abstract is. i could include more
% details summarizing the results. what i've done is include basic background
% info so anyone could understand roughly what's going on. if i include
% the results it feels like i should include all the background to understand
% the statement of results, which would be too much. so it's either A) state
% the results without background (more typical?) or B) state background for the
% general concepts but not the specific results.

\ifofficial
\section*{Acknowledgements}

I thank Zhicheng Gao, not just for supervision on this thesis but as well for
collaboration and opportunities over the years.
I am also grateful for the suggestions and corrections of the thesis committee
which includes Toufik Mansour, Mike Newman, Daniel Panario, and Michiel Smid.
Professor Panario also played a big role in letting me get into combinatorics
research.

This work was done in part on the premises of the Fields Institute for Research
in Mathematical Sciences and the Massachussetts Institute of Technology.
\fi

\setcounter{tocdepth}{4}
\tableofcontents

\ifofficial
\cleardoublepage
\pagenumbering{arabic}
\addcontentsline{toc}{section}{List of tables}
\listoftables
\cleardoublepage
\addcontentsline{toc}{section}{List of figures}
\listoffigures
\else
\pagenumbering{arabic}
\fi

% offdef
\section*{Notation}
\addcontentsline{toc}{section}{Notation}

All $n$-tuples over set $\Xi$ \cite{ac}: $\SEQ_n(\Xi)$
% \pageref{not:seq}

All finite sequences over set $\Xi$: $\SEQ(\Xi) = \cup_n \SEQ_n(\Xi)$

Asymptotic equivalence; $f(n)$ asymptotic to $g(n)$:
$\lim_{n \to \infty} f(n)/g(n) = 1 \iff f(n) \sim g(n)$
% \pageref{not:asympt}

Big-Oh;
there is $c>0$ such that for all sufficiently large $n$, $|f(n)| \leq c g(n)$:
$f(n) = O(g(n))$
% \pageref{not:bigoh}

% offdef
Big-Theta;
there are $c,d>0$ such that for all sufficiently large $n$,
$|f(n)| \leq c g(n)$ and $|f(n)| \geq d g(n)$:
$f(n) = \Theta(g(n))$

Cardinality: $|\Xi|$

% https://math.stackexchange.com/a/233117/3456
Closed neighborhood: $N[v]$,
open neighborhood (excludes $v$ unless there is a loop) \cite{bondy1976graph}:
$N(v)$
% \pageref{not:nei}

Concatenation of finite sequences: $(a, b)^\frown (c, d) = (a,b,c,d)$
% https://math.stackexchange.com/a/423602/3456
% \pageref{not:concat}

Convergence in distribution, weak convergence: $\Rightarrow$
% \pageref{not:dist}

Disjoint union: $\disjun$
% \pageref{not:disjun}

Falling factorial: $n^{\underline{k}} = n(n-1)\cdots (n-k+1)$
% \pageref{not:fact}

Finite cyclic groups: $\mathbb{Z}_k = \{0,\ldots,k-1\}$

Finite sequence short form: $1^4 2 3^2 = (1,1,1,1,2,3,3)$

First $k$ positive integers: $[k] = \{ n : 1 \leq n \leq k, n \in \mathbb{Z} \}$

In-neighborhoods (open and closed) \cite{bondy1976graph}: $N^-(v), N^-[v]$,
out-neighborhoods: $N^+(v), N^+[v]$

Iverson bracket; $1$ if the
statement $\phi$ is true and $0$ otherwise \cite{knuth1992two}:
$[\phi]$
%\pageref{not:iverson}

Matrix entry: $[M]_{i,j}$

Matrix/vector literal:
$\left[
\begin{array}{c}
 1 \\
 2 \\
\end{array}
\right] = \rvect{1, 2}^\top$

Non-negative integers: $\mathbb{Z}_{\geq 0} = \{0, 1, 2, \ldots\}$

Normal distribution function with mean $\mu$, variance $\sigma^2$:
$N(\mu, \sigma^2)$

Partial derivative of power series with respect to indeterminate $u$: $D_u f$
% \pageref{not:deriv}

Positive integers: $\mathbb{Z}_{>0} = \{1, 2, 3, \ldots\}$

Reversal of finite sequence: if $a = (a(1), \ldots, a(m))$, then
$\cev{a} = (a(m), \ldots, a(1))$
% \pageref{not:rev}

Stirling subset numbers (Stirling numbers of the second kind)
  \cite[p.\ 258]{concrete}: $\left\{ {m \atop k} \right\}$

Subset: $\subseteq$, strict subset: $\subset$

Sum of finite sequence $a$: $\Sigma a$ (capital sigma)

Transitive closure of arc relation: $\longrightarrow$ (long arrow)
% \pageref{not:arrow}

% Typical notation but not definitions
%
%For power series we use capital letters, such as $A(z)$.
%For power series coefficients we use lower case letters e.g.\ $a_n$.
%For combinatorial classes we use calligraphic capital Roman letters, such as
%$\mathcal{A}$.
%These symbols take a tilde to refer to unlabeled graphs.
%
%Directed paths: $p, P, \mathcal{P}$,
%cycles: $c, C, \mathcal{C}$,
%trees: $t, T, \mathcal{T}$
%
%general unstructured sets $\Xi, \Psi, \Phi$
%weight func: $W$
%
%number of occurrences: $r$
%
%span: $\sigma$
%
%digraph: $D$
%
%edge relation: $E(D)$. old: $\de$
%
%matrix: $M$
%
%derived digraph: $D_\times$
%
%arc (directed edge): $e$
%
%real numbers: $x,y$
%
%pattern: $\tau$
%
%complex number/indeterminate: $z, u,v,w$
%
%vertices: $u, v, w$
%
%eigenvectors: $u_\lambda, v_\lambda$
%
%total: $n$, length: $m$
%
%don't use: $I, l, \ell, \iota$
%
%unused so far: bold caligraphic, bold normal, hat above, arrow above
%
%group: $G$
%
%group elements: $a,b,c$
%
%semigroup: $S$
%
%semigroup elements: $s,t$
%
%random variables: $X, Y, Z$
%
%integers: $n, j,k$
%
%named $m$-tuples/vectors/sequences/words/compositions: lowercase Roman or Greek
%e.g.\ $d$ with domain $[n]$, $d = (d(1), \ldots, d(n))$
%
%sequence terms/set elements: $d_i, d_j, \ldots, d_k$ (unnamed)
%
%secondary sequence terms/set elements:
%  $d^{\langle i \rangle}, d^{\langle j \rangle}, \ldots, d^{\langle k \rangle}$
%  (unnamed)

\section{Introduction} \label{sec:intro}

If $\Xi$ is a finite set (sometimes called an \emph{alphabet}), a \emph{word}
$w$ over $\Xi$ is a sequence $w = (w(1), \ldots, w(m))$ where $w(i) \in \Xi$
for each $i$.
In particular, if $|\Xi| = k$ we call $w$ a $k$-ary $m$-word.
Without loss of generality, if $|\Xi| = k$ we assume
$\Xi = [k] = \{1, \ldots, k\}$, which
is an alphabet with a total order.
The terms that make up a word are called letters.
Of course the number of all $k$-ary $m$-words is $k^m$.

If $(S, +)$ is a semigroup, an \emph{$m$-composition of $s \in S$ over $S$} is
a sequence $x = (x(1), \ldots, x(m))$ where $x(i) \in S$ for each $i$,
and $\Sigma x = x(1) + \cdots + x(m) = s$.
If $S$ is finite, a composition over $S$ and a word over $S$ mean the same
thing; the difference is that we use the word \emph{composition} in the context
where we pay attention to the sum of the sequence.
The prototypical compositions are integer compositions, where
$(S, +) = (\mathbb{Z}_{> 0}, +)$.
The terms that make up a composition are called parts.
A simple exercise gives that the number of $m$-compositions of $n$ over
the positive integers is $\binom{n-1}{m-1}$.

A \emph{subword} of a finite sequence is a contiguous subsequence, so
$(a,a,c,b)$ is a subword of $(a,a,a,a,c,b,b)$.
For any kind of finite sequences we may sometimes use the shorthand word
notation $(a,a,a,a,c,b,b) = a^4 c b^2 = aaaacbb$.

A directed graph (\emph{digraph}) is a pair $(V, E)$, where $V$ is a
finite set (the \enquote{vertices}), and $E \subseteq V \times V$ is a binary
relation (the \enquote{directed edges} or \enquote{arcs}).
If a digraph is specified only by its arcs, the vertices are taken to be
all those which appear in an arc.
A \emph{weighted digraph} is a digraph $(V,E)$ together with a vertex weight
function $W: V \to S$, where $S$ is a fixed semigroup.
Words and compositions are both finite sequences over a set.
Equivalently, we may regard them as directed paths (digraphs with an ordered
set of vertices and arcs from predecessors to successors) where vertices take
weights from the set (which is taken without loss of generality to be a
semigroup).
The benefit of this view comes when generalizing beyond directed paths to
different types of weighted digraphs.

Our goal, ultimately, is to count weighted digraphs.
Specifically, we are interested in counting how many of these objects
satisfy a certain restriction.
Below, we describe a general concept of pattern occurrence and avoidance
in weighted digraphs which we can use to express restricted weighted digraph
families.
The familiar definitions of patterns in words and compositions (e.g.\
\cite{cofc}) are available as special cases.

If $\Gamma$ is a digraph, we write $V(\Gamma)$ and $E(\Gamma)$ for the sets
of vertices and arcs of $\Gamma$.
Given digraphs $\Gamma_1, \Gamma_2$, a digraph \emph{homomorphism} is a function
$h: V(\Gamma_1) \to V(\Gamma_2)$ such that for any two vertices
$u, v \in V(\Gamma_1)$, we have
\[ (u,v) \in E(\Gamma_1) \implies (h(u), h(v)) \in E(\Gamma_2). \]
If $\Gamma_1, \Gamma_2$ have weight functions $W_1, W_2$, a \emph{weighted
homomorphism} is a homomorphism $h$ such that $W_1(v) = W_2(h(v))$ for all
$v \in V(\Gamma_1)$.
A (weighted) \emph{isomorphism} is a bijective (weighted) homomorphism.

% offdef
\begin{notation}
\label{not:disjun}
We use the notation $A \disjun B$ to denote the union of the disjoint
sets $A$ and $B$.
\end{notation}

A one-vertex subdivision of a digraph $(V, E)$ is a new digraph $(V', E')$,
where $V' = V \disjun \{v\}$ and for some $(v_1, v_2) \in E$, we have
\[ E' = \{(v_1, v), (v, v_2)\} \cup E \setminus \{(v_1, v_2)\}. \]
A weighted one-vertex subdivision is one where the weight function is not
modified except for the new vertex $v$, which may take any weight.
In general, a \emph{subdivision} of a digraph is a digraph obtained
by $0$ or more one-vertex subdivisions.
In the context of weighted graphs, subdivisions are always weighted.

\begin{example}
Figure~\ref{fig:undivided} shows digraphs $\Gamma, \Gamma'$.
The digraph $\Gamma'$ is a subdivision of $\Gamma$ obtained by adding
the vertices $u_5, u_6$ which are shown in bold.\qedhere

\begin{figure}
%\centering  \includegraphics[width=25em]{undivided}
\centering
\begin{tikzpicture}[scale=0.9]
\begin{scope}[every node/.style={circle,thick,draw}]
    \node[label={below:$1$}] (A) at (0,0) {$u_1$};
    \node[label={below:$2$}] (B) at (2,0) {$u_2$};
    \node[label={-45:$3$}] (C) at (4,1) {$u_3$};
    \node[label={below:$1$}] (D) at (4,-1) {$u_4$};
\end{scope}

\begin{scope}[>={Stealth[black]},
              every node/.style={fill=white,circle},
              every edge/.style={draw=gray,very thick}]
    \path [->] (A) edge (B);
    \path [->] (B) edge (C);
    \path [->] (C) edge (D);
    \path [->] (D) edge (B);

    %\path [->] (B) edge[bend right=60] (E);
\end{scope}
\end{tikzpicture}

\vspace{0.5cm}
%\centering  \includegraphics[width=32em]{divided}
\centering
\begin{tikzpicture}[scale=0.9]
\begin{scope}[every node/.style={circle,thick,draw}]
    \node[label={below:$1$}] (A) at (-2,0) {$u_1$};
    \node[label={below:$2$}] (B) at (2,0) {$u_2$};
    \node[label={below:$3$}] (C) at (6,0) {$u_3$};
    \node[label={below:$1$}] (D) at (4,-1) {$u_4$};
    \node[label={below:$3$}] (E) at (0,0) {$\bm{u_5}$};
    \node[label={below:$3$}] (F) at (4,1) {$\bm{u_6}$};
\end{scope}

\begin{scope}[>={Stealth[black]},
              every node/.style={fill=white,circle},
              every edge/.style={draw=gray,very thick}]
    \path [->] (A) edge (E);
    \path [->] (E) edge (B);
    \path [->] (B) edge (F);
    \path [->] (F) edge (C);
    \path [->] (C) edge (D);
    \path [->] (D) edge (B);

    %\path [->] (B) edge[bend right=60] (E);
\end{scope}
\end{tikzpicture}

\caption{Weighted digraphs $\Gamma$ (above) and $\Gamma'$ (below).
  Integer vertex weights are shown below the corresponding vertices.
\label{fig:undivided}}
\end{figure}
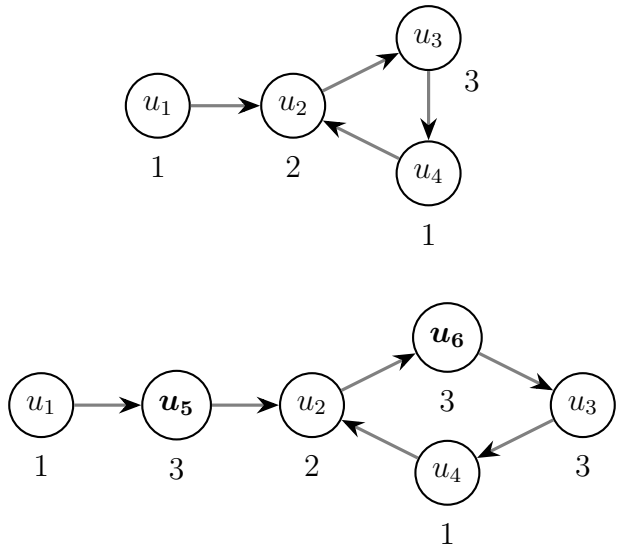
\end{example}

Given weighted digraphs $\Gamma$, $P$, a subdivision $P'$ of $P$, and a
subgraph $\Gamma_1$ of $\Gamma$, if we have a weighted isomorphism $f: V(P')
\to V(\Gamma_1)$, we say that $f_{|V(P)}$ is the \emph{match of $f$ with
respect to $P$}.
%yes this is correct; we count 11 as having 2 cyclic occurrences of 11
%  because we count normal occurrences of 11 in 111
A \emph{local occurrence of $P$ in $\Gamma$} is the match with respect to $P$
of some weighted $f$ from $P$ to a subgraph of $\Gamma$.
A \emph{global occurrence of $P$} is the match with respect to $P$ of some
weighted $f$ from any subdivision $P'$ of $P$ to a subgraph of $\Gamma$.

That is, global occurrences may map adjacent vertices in $P$ to non-adjacent
vertices in $\Gamma$ while local occurrences cannot.
The semigroup $S$ of weights is always the same for $\Gamma$ and $P$.

\begin{example}
Figure~\ref{fig:c6_123} shows weighted digraphs $\Gamma$ and $P$.
There exist no local occurrences of $P$ in $\Gamma$ but many global occurrences.
One global occurrence is given by matching the
vertices
\[ u_1 \mapsto v_1, u_2 \mapsto v_2, u_3 \mapsto v_3. \]
Another is given by
\[ u_1 \mapsto v_1, u_2 \mapsto v_5, u_3 \mapsto v_6. \qedhere\]
\end{example}

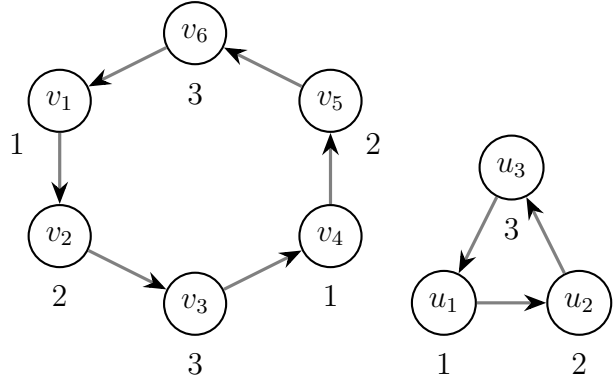
\begin{figure}
%\centering  \includegraphics[width=25em]{c6_123}
\centering
\begin{tikzpicture}[scale=0.9]
\begin{scope}[every node/.style={circle,thick,draw}]
    \node[label={-135:$1$}] (A) at (-2,-1) {$v_1$};
    \node[label={below:$2$}] (B) at (-2,-3) {$v_2$};
    \node[label={below:$3$}] (C) at (0,-4) {$v_3$};
    \node[label={below:$1$}] (D) at (2,-3) {$v_4$};
    \node[label={-45:$2$}] (E) at (2,-1) {$v_5$};
    \node[label={below:$3$}] (F) at (0,0) {$v_6$};
\end{scope}

\begin{scope}[>={Stealth[black]},
              every node/.style={fill=white,circle},
              every edge/.style={draw=gray,very thick}]
    \path [->] (A) edge (B);
    \path [->] (B) edge (C);
    \path [->] (C) edge (D);
    \path [->] (D) edge (E);
    \path [->] (E) edge (F);
    \path [->] (F) edge (A);

    %\path [->] (B) edge[bend right=60] (E);
\end{scope}
\end{tikzpicture}
%\centering  \includegraphics[width=25em]{c3_123}
\centering
\begin{tikzpicture}[scale=0.9]
\begin{scope}[every node/.style={circle,thick,draw}]
    \node[label={below:$1$}] (A) at (0,0) {$u_1$};
    \node[label={below:$2$}] (B) at (2,0) {$u_2$};
    \node[label={below:$3$}] (C) at (1,2) {$u_3$};
\end{scope}

\begin{scope}[>={Stealth[black]},
              every node/.style={fill=white,circle},
              every edge/.style={draw=gray,very thick}]
    \path [->] (A) edge (B);
    \path [->] (B) edge (C);
    \path [->] (C) edge (A);

    %\path [->] (B) edge[bend right=60] (E);
\end{scope}
\end{tikzpicture}
\caption{Weighted digraphs $\Gamma$ (left) and $P$ (right).
\label{fig:c6_123}}
\end{figure}

The size of a digraph is the number of arcs it contains.
A \emph{digraph pattern} is a set $\mathcal{P}$ of weighted digraphs such
that the sizes of digraphs in $\mathcal{P}$ form a bounded set.
The elements $P \in \mathcal{P}$ are \emph{pattern instances} and for
our purposes, the instances of $\mathcal{P}$ are always one or more
different weightings of a single digraph.
%For a given definition of matching (isomorphic/homomorphic and local/global),
A digraph $\Gamma$ has $r$ occurrences of $\mathcal{P}$ if the total, over all
instances $P \in \mathcal{P}$, of the number of occurrences of $P$ in $\Gamma$
is $r$.
Avoiding a digraph pattern means having $0$ occurrences, and avoidance
of a set of patterns means avoiding each of the patterns.

Given an arbitrary finite vertex set, say $V = [n]$,
an \emph{unlabeled} weighted digraph $\tilde{\Gamma}$ on $V$ is an equivalence
class of weighted digraphs on $V$ where $\Gamma_1$ and $\Gamma_2$ are
equivalent if there is a weighted isomorphism from $\Gamma_1$ to $\Gamma_2$.
Relabeling vertices has no effect on digraph pattern matching
because it does not affect the structure of the digraph or its weights,
so we may speak of the number of occurrences of a pattern in an unlabeled
weighted digraph.

The sum, a.k.a.\ total, of a weighted digraph $\Gamma$ is
$\sum_{v \in V(\Gamma)} W(v)$.
This expression is always well defined for abelian semigroups $S$.
For non-abelian $S$ we must have labeled digraphs, and the vertices
must have a fixed total order which determines the order of summation.

\begin{example} \label{exa:pathcycle}
Define the set of directed paths
$$\big\{ \{(j, j+1) : 1 \leq j < n \}: n \geq 1 \big \},$$
and define the set of directed cycles
$$\big\{ \{(j, j+1) : 1 \leq j < n \} \cup \{(n,1)\}: n \geq 1
  \big \}.$$
With the terminology of \cite{cofc} we may make the following identifications.
Weighted paths are words or compositions, and
weighted cycles are cyclic words or compositions (where the last term is
considered to precede the first for pattern occurrence purposes).
Unlabeled weighted cycles over the vertex sets $[n]$
correspond to circular words or compositions.
\end{example}

In the remainder, \enquote{path} means directed path and \enquote{cycle}
means directed cycle.

With the above concepts laid down we are able to describe a wide taxonomy of
counting problems which all ask, how many weighted digraphs are there with $r$
occurrences of a digraph pattern $\mathcal{P}$?
The primary dimensions of this taxonomy follow.
\begin{itemize}
\item \emph{Class of digraphs.}
  There are many options for the kind of digraph that we are counting.
  Paths and cycles are the most basic, but others could be used:
  regular, planar, bipartite, et cetera.
  The elements of $\mathcal{P}$ are also digraphs and can come in any
  form.
  We largely focus on digraph patterns $\mathcal{P}$ that consist of various
  weightings of a path.
\item \emph{Labeled vs.\ unlabeled.}
  The digraphs we count may be either labeled or unlabeled.
  Usually counting the labeled case is a prerequisite for the unlabeled case.
\item \emph{Local vs.\ global occurrences.}
  If we ask for $r$ occurrences, we are either talking about local or global
  occurrences.
\item \emph{Track size, total, or both.}
  When counting words, any algebraic structure of the alphabet is ignored,
  unlike with compositions.
  Similarly for digraphs, we may or may not keep track of the total.
\item \emph{Choice of semigroup.}
  Any semigroup could be used as long as it gives finite counts, e.g.\ the
  number of $6$-compositions of $17$ over $\mathbb{Z}$ is infinite.
\end{itemize}

The remainder of this document contains solutions to a selection of problems
from the taxonomy just described.
We largely defer discussions of the relevant prior literature to the
sections that follow due to their heterogeneity.
However, we mention here the 2010 book \cite{cofc} by Heubach and Mansour which
is a useful reference for many of the topics of this document, especially for
exact (as opposed to asymptotic) counting.
The remaining sections of this document are organized as follows.

% 2.1
We begin in \S \ref{sec:localpaths} with local occurrences in weighted
paths, where the semigroup $S$ is a finite group.
A weighted path with no local occurrences of some pattern is known as a
locally restricted composition, assuming we track the total.
(In this and subsequent sections we generally use the familiar concepts, such
as \enquote{compositions}, although we refer to weighted digraphs when useful.)
We find, under conditions, that the number of locally restricted
compositions of a group element is asymptotically independent of the group
element.
We reach the same conclusion for compositions containing $r>0$ local
occurrences of a pattern.
After verifying these conditions for a variety of examples, we
show that under similar conditions the number of local occurrences in a random
composition is asymptotically normal.
%2.2
In \S \ref{sec:transfer} we make a note on when the matrices used in the
transfer matrix method can be reduced in size for computational and
practical benefits.
This section and others with heading \enquote{Note on\ldots} are extended
remarks which briefly introduce relevant lines of research.
%2.3
The problem of counting directed rooted trees is noted in \S \ref{sec:trees},
also in the context of local pattern occurrence.

%3.1
Next, \S \ref{sec:cycliccomps} is similar to \S \ref{sec:localpaths} but counts
digraphs which are cycles rather than paths, which correspond to
objects known as locally cyclically restricted compositions.
Again we find that under conditions the asymptotic number of
such compositions of a finite group element does not depend
on the group element, and we show asymptotic normality of the number of
local pattern occurrences.
%3.2
In \S \ref{sec:intcomps} we note how to count locally cyclically restricted
integer compositions, i.e.\ cycles weighted by $\mathbb{Z}_{> 0}$,
in the framework of locally restricted integer compositions of \cite{infinite}.

%4.1
The results of \S \ref{sec:circcomps} together with Moebius inversion
allow us to count circular locally restricted
compositions over a finite group which is done in \S \ref{sec:circcomps}.
As in Example \ref{exa:pathcycle},
circular objects correspond to unlabeled weighted cycles.
%4.2
In \S \ref{sec:palcomps} we note how to count \enquote{undirected}
locally restricted compositions, i.e.\ unlabeled weighted undirected paths.

%5
%Global pattern avoidance is the subject of \S \ref{sec:subseq}.
%5.1.1
\enquote{Subsequence patterns} and \enquote{generalized patterns}
are types of patterns that are used in the context of ordered semigroups $S$.
In the language of this section they are digraph patterns made up of all paths
that have a certain size and a certain relative ordering among the vertex
weights.
\enquote{Partially ordered patterns} can be used to represent a set of
subsequence patterns.
Subsequence and partially ordered patterns are used in the context of
global occurrences, while generalized patterns actually specify which
arcs may be subdivided and which may not.
In \S \ref{sec:pairs} we count weighted paths, specifically words or integer
compositions, that avoid different pairs of generalized patterns.
%5.1.2
The counting results in \S \ref{sec:pops} are concerned with words or integer
compositions that avoid a family of partially ordered patterns
(roughly, patterns where the maximal weights must be at the beginning or end).
%5.1.3
In \S \ref{sec:subseqsym} we note how to adapt results for subsequence
pattern avoidance in words to circular words (unlabeled weighted cycles) or
\enquote{undirected} words (unlabeled weighted undirected paths).
%5.2
We make a note on subsequence pattern avoidance in objects other than words and
integer compositions in \S \ref{sec:zk}, namely in compositions over
$\mathbb{Z}_k$.
Our technique involves using the multisection formula together with results
for integer compositions, and we apply it to an example partially
ordered pattern.

Finally, we list open problems in \S \ref{sec:conc}.

\section{Locally restricted compositions} \label{sec:labeledlocal}

% implicit
% main paper
A locally restricted composition is one that avoids a certain set of length-$l$
sequences as subwords.
Over the integers, these objects have been studied successfully in a number
of papers by Bender et al.\ \cite{infinite, compositionsiv, bender2014part}.
In fact, those works include somewhat more general restrictions, where a
subword may or may not be allowed based on the residue of its position in a
composition, and special rules can apply to parts near the beginning or end.
Under some conditions, asymptotics for the number of locally restricted integer
compositions were given in \cite{infinite}.
That paper also established a normal limiting distribution for the number of
occurrences of a subword in a uniform random integer composition.
The later papers \cite{compositionsiv, bender2014part} focused on the
probability distributions of part sizes and other parameters.

Given two sequences $x,y$ of the same length over ordered sets, we say
$x$ and $y$ are order isomorphic if
$x(i) < x(j) \iff y(i) < y(j)$ for all $i,j$.
A \emph{subword pattern} $\tau$ is a word over $[k]$.
Assume the length of $\tau$ is $l$.
An \emph{occurrence} of $\tau$, as a subword pattern, in $x$ is a sequence of
indices $i, \ldots, i + l - 1$ such that
$(x(i), \ldots, x(i + l - 1))$ is order isomorphic to $\tau$.

% explicit
Mansour and others in \cite{mansour2006counting, subwords, cofc} count
integer compositions by number of occurrences of specific subword patterns such
as $123$ and $112$.
These results are less general than those obtained by Bender and
collaborators in the case of avoidance but give simpler expressions.
The umbral technique in \cite{zeilberger2000umbral} is also used to explicitly
count locally restricted objects.

\begin{remark} \label{rem:undirectedpaths}
In the language of \S \ref{sec:intro}, compositions are weighted directed
paths where we keep track of the total weight.
Weighted undirected paths may be counted in a similar manner.
% pattern in undirected <=> set of patterns in directed, but
% you have to multiply occurrences in directed by number of ways for that
% pattern to be undirected occurrence of original pattern.
% for iso matching it's a bit simpler
\end{remark}

\subsection{Compositions over a finite group} \label{sec:localpaths}
% code

Locally restricted compositions over finite fields and even finite
abelian groups were counted in \cite{abelian} under some conditions, and in
less generality in the preceding papers mentioned therein.
Over $\mathbb{Z}_k$, the method used in that paper involves obtaining the
relevant generating function $F(z)$ for integer compositions over $[k]$, and
working with $\sum_{j \equiv s \pmod{k}}[z^j]F(z)$.
For other finite abelian groups, the method is extended to multivariate
generating functions.
Below we give an alternative counting method that expands the range of
applicability beyond abelian groups, addressing a problem posed in
\cite{abelian}.
We begin this section considering compositions over a finite semigroup
$(S,+)$ and eventually specialize to finite groups.

% *D and D_x*

% offdef
\begin{notation} \label{not:seq}
For a finite set $\Xi$, we denote all $n$-tuples over $\Xi$ by
$\SEQ_n(\Xi)$.
We define $\SEQ(\Xi) = \cup_{n \geq 0} \SEQ_n(\Xi)$.
\end{notation}

\begin{definition}
Let $\Xi$ be a finite set, and let $n$ be a positive integer.
The $n$-dimensional \emph{de Bruijn graph} (actually a digraph) on $\Xi$ has
vertex set $V = \SEQ_n(\Xi)$ and includes the arc from
$(\graphf{u}_{1}, \ldots, \graphf{u}_{n})$ to
$(\graphf{v}_{1}, \ldots, \graphf{v}_{n})$
if and only if
\[
(\graphf{u}_{2}, \ldots, \graphf{u}_{n}) =
(\graphf{v}_{1}, \ldots, \graphf{v}_{n - 1}).
\]
\end{definition}

% offdef
\begin{notation}
\label{not:concat}
We use the symbol ${}^\frown$ to denote concatenation of finite sequences,
e.g.\ $(a,b)^\frown(c,d) = (a,b,c,d)$.
\end{notation}

Let $\sigma$ be a positive integer which we call the \emph{span}, and let
$\graphf{D}$ be a subgraph of the $\sigma$-dimensional de Bruijn graph on $S$.
Then $D$ is a \emph{de Bruijn subgraph}.
This digraph $\graphf{D}$ is associated with a set of locally restricted
compositions as follows.
A walk in a digraph is a sequence of vertices, not necessarily distinct,
where for any subword $(u, v)$ there is an arc from $u$ to $v$.
An $m$-composition over $S$ is legal according to $\graphf{D}$ if it takes
the form
\[ {\graphf{w}_1}^\frown (\graphf{w}_{2}(\sigma), \ldots,
  \graphf{w}_{m-\sigma+1}(\sigma)) =
(\graphf{w}_{1}(1), \ldots, \graphf{w}_{1}(\sigma),
\graphf{w}_{2}(\sigma), \ldots,
  \graphf{w}_{m-\sigma+1}(\sigma)), \]
where $\graphf{w}_{1}, \ldots, \graphf{w}_{m-\sigma+1}$ is a walk in
$\graphf{D}$.
In other words, we build compositions from $\graphf{D}$ by starting at
any vertex, and taking a walk in which we append the last element of each
vertex we visit after the first vertex.
Additionally, we may designate sets of start and end vertices which are the
allowed vertices for walks to start and end at.

We write the set of all $m$-compositions of $s$ that are legal according to
$D$ with start set $\Psi$ and finish set $\Phi$ as
$\mathcal{P}_s(m; D, \Psi, \Phi)$.
We may also write this with $s, m, \Psi$, or $\Phi$ omitted to remove those
conditions,
e.g.\ $\mathcal{P}_s(D, \Psi, \Phi) = \cup_m \mathcal{P}_s(m; D, \Psi, \Phi)$.
Also, $p_s(m; D, \Psi, \Phi) = |\mathcal{P}_s(m; D, \Psi, \Phi)|$,
$P_s(z; D, \Psi, \Phi) = \sum_{m \geq 0} p_s(m; D, \Psi, \Phi) z^m$.

Define a new digraph $\graphf{D}_\times$ with vertex set $V(\graphf{D}) \times
S$ such that $((\graphf{u}, s), (\graphf{v}, t)) \in E(\graphf{D}_\times)$ if and only if
$(\graphf{u}, \graphf{v}) \in E(\graphf{D})$ and $s + \graphf{v}(\sigma) = t$.
We call $\graphf{D}_\times$ the \emph{derived digraph} (of the \emph{base
digraph} $\graphf{D}$).
We define the start set $\Psi_\times \subseteq V(\graphf{D}_\times)$ to contain
all $(\graphf{v}, s)$ such that $\sum \graphf{v} = s$ and $\graphf{v} \in
\Psi$.
For each $s \in S$ the finish set $\Phi_s \subseteq V(\graphf{D})$ for $s$
contains all vertices $(\graphf{v}, s)$ where $\graphf{v} \in \Phi$.

Fix an ordering on $V(\graphf{D}_\times)$ so we can define an adjacency matrix
$M_\times$ of $\graphf{D}_\times$.
Let $\psi_{\times} \in \mathbb{R}^{|V(D_\times)|}$ be the indicator vector for
$\Psi_\times$, and let $\phi_s \in \mathbb{R}^{|V(D_\times)|}$ be the indicator
vector for $\Phi_s$.

\begin{proposition} \label{prop:dtimescount}
For $m \geq \sigma$, we have
\[
p_s(m; D, \Psi, \Phi) = \psi_\times^\top M_\times^{m-\sigma} \phi_s.
\]
The generating function $P_s(z; D, \Psi, \Phi)$ is rational.
\end{proposition}
\begin{proof}
Let $W_s$ be a walk in $\graphf{D}_\times$ starting in $\Psi_\times$ and ending
in $\Phi_s$ in the form
\[W_s =
((\graphf{w}_1, t), \ldots, (\graphf{w}_{m-\sigma+1}, s)),
\]
so the $D$-vertices corresponding to $W_s$ are
$\graphf{w}_1, \ldots, \graphf{w}_{m-\sigma+1}$.
We say the $m$-composition of $s$ defined by $W_s$ is
\[ {\graphf{w}_1}^\frown (\graphf{w}_{2}(\sigma), \ldots,
  \graphf{w}_{m-\sigma+1}(\sigma)) . \]
By the definition of $D_\times$, the $m$-compositions corresponding to a walk
$W_s$ in $D_\times$ are exactly those $m$-compositions allowed by $D$ with
total $s$.
That is, the compositions defined by $\graphf{D}_\times$ and $\graphf{D}$
are the same, but $\graphf{D}_\times$ also directly keeps track of the total.

Counting walks in a digraph via the adjacency matrix is a well-known procedure.
The result follows from the relation
\[ [M_\times^q]_{i,j}
  = \sum_{k=1}^{|V(D_\times)|} [M_\times]_{i,k} [M_\times^{q-1}]_{k,j} \]
which means walks of length $q+1$ from vertex $i$ to $j$ are walks of length
$2$ from $i$ to $k$, merged with a walk of length $q$ from $k$ to $j$.
This is known as the transfer matrix method;
background may be found in \cite{stanley1}.

We have
\begin{align*}
P_s(z; D, \Psi, \Phi) &= \sum_{m \geq \sigma }
  \psi_\times^\top M_\times^{m-\sigma} \phi_s z^m + P(z) \\
&= z^\sigma \psi_\times^\top \left(\sum_{m \geq 0 } M_\times^{m} z^m \right)
  \phi_s + P(z) \\
&= z^\sigma \psi_\times^\top \left(I - z M_\times \right)^{-1}
  \phi_s + P(z),
\end{align*}
where $P(z)$ is a polynomial which counts the appropriate $m$-compositions with
$m < \sigma$.
The entries of $(I - z M_\times)^{-1}$ lie in the field of fractions of
$\mathbb{Q}[z]$, i.e.\ the rational functions $\mathbb{Q}(z)$.
% the inverse representation for the infinite matrix series is stated in
% AC Prop V.9
\end{proof}

\begin{example} \label{exa:exactcount}
Carlitz compositions are those where adjacent parts must be different.
Figure \ref{fig:carl_count} shows an example of $\graphf{D}_\times$ for Carlitz
compositions over $\mathbb{Z}_3$.

\begin{figure}
\centering

\begin{tikzpicture}[scale=0.9]
\begin{scope}[every node/.style={circle,thick,draw}]
    \node (A) at (0,0) {(0)};
    \node (B) at (2,-1) {(1)};
    \node (C) at (-2,-1) {(2)};
\end{scope}

\begin{scope}[>={Stealth[black]},
              every node/.style={fill=white,circle},
              every edge/.style={draw=gray,very thick}]
    \path [<->] (A) edge (B);
    \path [<->] (B) edge (C);
    \path [<->] (C) edge (A);

    %\path [->] (B) edge[bend right=60] (E);
\end{scope}
\end{tikzpicture}
\begin{tikzpicture}[scale=0.8]
\begin{scope}[every node/.style={circle,thick,draw}]
    \node (A) [double] at (0,0) {(0),0};
    \node (B) [double] at (2,-1) {(1),1};
    \node (C) at (-2,-1) {(2),0};
    \node (D) at (0,-3) {(0),2};
    \node (E) at (2,-4) {(1),0};
    \node (F) [double] at (-2,-4) {(2),2} ;
    \node (G) at (0,-6) {(0),1};
    \node (H) at (2,-7) {(1),2};
    \node (I) at (-2,-7) {(2),1} ;
\end{scope}

\begin{scope}[>={Stealth[black]},
              every node/.style={fill=white,circle},
              every edge/.style={draw=gray,very thick}]
    \path [->] (A) edge (B);
    \path [<->] (B) edge (C);
    \path [->] (C) edge (A);
    \path [->] (D) edge (E);
    \path [<->] (E) edge (F);
    \path [->] (F) edge (D);
    \path [->] (G) edge (H);
    \path [<->] (H) edge (I);
    \path [->] (I) edge (G);

    \path [->] (A) edge (F);
    \path [->] (D) edge (I);
    \path [->] (G) edge[bend left=70] (C);

    \path [->] (B) edge[bend left=70] (G);
    \path [->] (E) edge (A);
    \path [->] (H) edge (D);

    %\path [->] (B) edge[bend right=60] (E);
\end{scope}
\end{tikzpicture}

\caption{A base digraph $D$ (left) and derived digraph $D_\times$ (right)
  representing Carlitz compositions over $\mathbb{Z}_3$.
  Here all vertices of $D$ are allowed start and finish vertices.
  Vertices in $D_\times$ that are allowed start vertices are shown with a
  double circle.
\label{fig:carl_count}}
\end{figure}
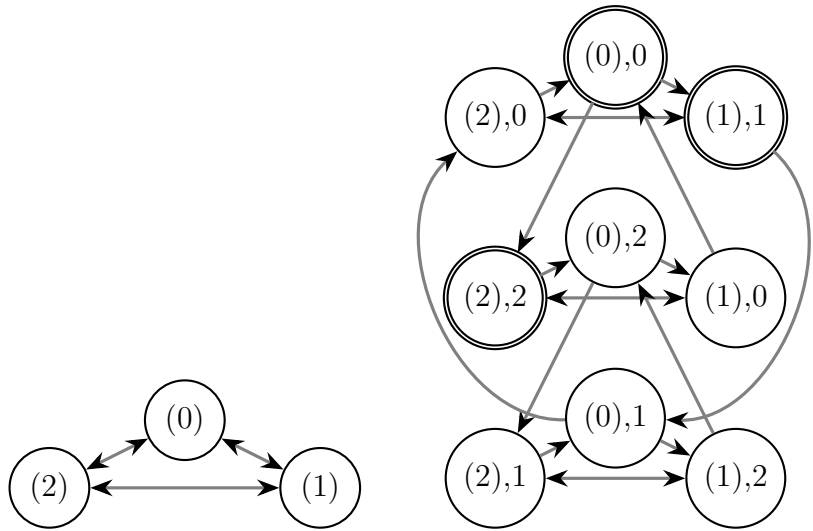

Let us order the vertices of $D_\times$ as
\[ ((0), 0), ((1), 1), ((2), 0), ((0), 2), ((1), 0), ((2), 2),
((0), 1), ((1), 2), ((2), 1). \]
Then we get
\[
M_\times = \left[
\begin{array}{ccccccccc}
 0 & 1 & 0 & 0 & 0 & 1 & 0 & 0 & 0 \\
 0 & 0 & 1 & 0 & 0 & 0 & 1 & 0 & 0 \\
 1 & 1 & 0 & 0 & 0 & 0 & 0 & 0 & 0 \\
 0 & 0 & 0 & 0 & 1 & 0 & 0 & 0 & 1 \\
 1 & 0 & 0 & 0 & 0 & 1 & 0 & 0 & 0 \\
 0 & 0 & 0 & 1 & 1 & 0 & 0 & 0 & 0 \\
 0 & 0 & 1 & 0 & 0 & 0 & 0 & 1 & 0 \\
 0 & 0 & 0 & 1 & 0 & 0 & 0 & 0 & 1 \\
 0 & 0 & 0 & 0 & 0 & 0 & 1 & 1 & 0 \\
\end{array}
\right],
\psi_\times = \left[
\begin{array}{c}
 1 \\
 1 \\
 0 \\
 0 \\
 0 \\
 1 \\
 0 \\
 0 \\
 0 \\
\end{array}
\right],
\phi_0 = \left[
\begin{array}{c}
 1 \\
 0 \\
 1 \\
 0 \\
 1 \\
 0 \\
 0 \\
 0 \\
 0 \\
\end{array}
\right],
\]
\[ \psi_\times^\top M_\times^{3-1} \phi_0 = 6. \]
So the number of Carlitz $3$-compositions of $0$ in $\mathbb{Z}_3$ is $6$.
\end{example}

\begin{remark} \label{rem:randpaths}
The following procedure generates a walk in $D_\times$ of length $m-\sigma+1$,
where all such walks are equally probable:
\begin{enumerate}
\item Pick a start vertex $v_1$ weighted by the number of $(m-\sigma+1)$-walks
from that vertex to a finish vertex.
\item
Given the current vertex $v_i$, select an out-neighbor where such neighbors are
weighted by the number of $(m-\sigma+1-i)$-walks from the neighbor vertex to a
finish vertex.
\end{enumerate}

Naturally, using the correspondence between walks and compositions, this gives
a method of random generation for locally restricted compositions.
Figure \ref{fig:carlitz100compZ3} shows an example with Carlitz
compositions.
Other examples of the method are found throughout this section.
\qedhere

\begin{figure}
\centering
\includegraphics[width=5.5in]{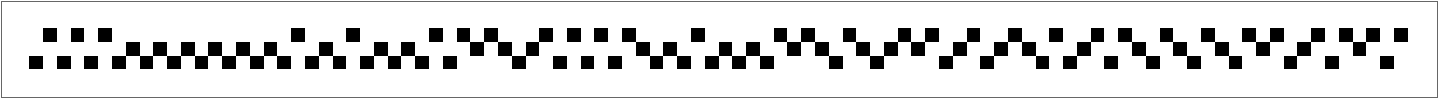}
\includegraphics[width=5.5in]{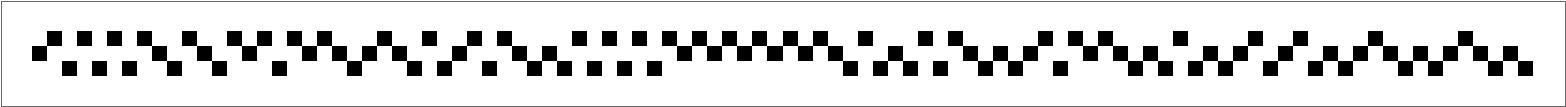}

\caption{Uniform-randomly generated Carlitz $100$-compositions of $0$ (above)
  and $1$ (below) over $\mathbb{Z}_3$.
  % offdef
  (The vertical axis represents the value of a part, i.e.\ the below
  composition starts $(1,2,0,2,\ldots)$.)
  \label{fig:carlitz100compZ3}}
\end{figure}
\end{remark}

If $\graphf{D}_\times$ is strongly connected and aperiodic, then we can obtain
a highly-precise asymptotic expression for
$p_s(m; D, \Psi, \Phi)$, $m \to \infty$,
via Proposition \ref{prop:dtimescount} and the Perron-Frobenius theorem.
(A digraph is \emph{aperiodic} if the set of all cycle lengths has no common
divisor besides $1$.)
We now give some general facts about the strong connectedness of
$\graphf{D}_\times$.

% *characterization of strong connectedness and aperiodicity (of components)*

If $\graphf{D}$ is not strongly connected then certainly $\graphf{D}_\times$
is not strongly connected either.
However, if $\graphf{D}$ decomposes into disconnected strong components, then
naturally we are able to simply count with each component separately and add.
In the following, \emph{we assume $\graphf{D}$ is strongly connected} (and
nonempty).

Unfortunately, if $\graphf{D}$ \emph{is} strongly connected,
$\graphf{D}_\times$ is not necessarily strongly connected.
Say $\graphf{D}$ is the digraph given in Figure~\ref{fig:not_strong}, over
$\mathbb{Z}_4$ with span $\sigma=2$.
In $\graphf{D}_\times$, there is a path from $((3,0),3)$ to $((1,3),3)$, but
there is no path from $((1,3),0)$ to $((1,3),3)$.

\begin{figure}
\vspace{0.5cm}
%\centering \includegraphics[width=25em]{not_strong}
\centering
\begin{tikzpicture}[scale=0.9]
\begin{scope}[every node/.style={circle,thick,draw}]
    \node (A) at (0,0) {$(1,3)$};
    \node (B) at (-2,1) {$(0,1)$};
    \node (C) at (-2,-1) {$(3,0)$};
    \node (D) at (2,1) {$(3,2)$};
    \node (E) at (4,0) {$(2,2)$};
    \node (F) at (2,-1) {$(2,1)$};
\end{scope}

\begin{scope}[>={Stealth[black]},
              every node/.style={fill=white,circle},
              every edge/.style={draw=gray,very thick}]
    \path [->] (A) edge (C);
    \path [->] (C) edge (B);
    \path [->] (B) edge (A);
    \path [->] (A) edge (D);
    \path [->] (D) edge (E);
    \path [->] (E) edge (F);
    \path [->] (F) edge (A);

    %\path [->] (B) edge[bend right=60] (E);
\end{scope}
\end{tikzpicture}

\caption{A digraph $\graphf{D}$ with vertices in $\mathbb{Z}_4^2$.
\label{fig:not_strong}}
\end{figure}
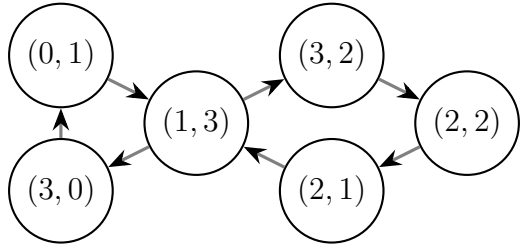

If the entirety of $\graphf{D}_\times$ is not strongly connected then we would
hope it is simply a disjoint union of strong components.
This is not true for general finite semigroups $S$.
For example, if there is $s^* \in S$ satisfying
\[\forall s \in S: s^* + s = s + s^* = s^*,\]
then walks in the digraph $\graphf{D}_\times$ will get \enquote{stuck} at $s^*$
and some weakly connected vertices will not be strongly connected.
We do obtain this desideratum, however, if $S$ is a group, as we show
eventually below.
\emph{In the following we assume that $S=G$ is a group.}

\begin{definition}
Let $D_{B}$ be an arbitrary digraph, referred to as the \emph{base digraph}.
Let $G$ be a finite group, and let $\alpha: E(D_B) \to G$ map arcs of $D_B$ to
group elements.
Together, $D_B$ and $\alpha$ are known as a \emph{voltage graph}.
We define the \emph{derived digraph} $D_\alpha$ such that
$V(D_\alpha) = V(D_B) \times G$ and $((u, a), (v, b))$ is an arc if and only if
$(u,v) \in E(D_B)$ and $a + \alpha(u,v) = b$.
\end{definition}

The digraphs $\graphf{D}_\times$
directly give derived digraphs in the sense of voltage graphs, specifically
\enquote{right derived ordinary voltage graphs},
if we associate the group element $\graphf{u}(\sigma)$ to all incoming arcs
to $\graphf{u}$ in the base digraph $\graphf{D}$.

% offdef
\begin{notation}
\label{not:nei}
If $v$ is a vertex in a digraph with arc relation $E$, we use the following
notations:
\begin{align*}
  N^-(v) &= \{ u: (u,v) \in E \}\\
  N^-[v] &= N^-(v) \cup \{v\}\\
  N^+(v) &= \{ u: (v,u) \in E \}\\
  N^+[v] &= N^+(v) \cup \{ v \}\\
  N(v) &= N^-(v) \cup N^+(v)\\
  N[v] &= N^-[v] \cup N^+[v].
\end{align*}
\end{notation}

\begin{remark}
Let $(V_1, E_1)$ and $(V_2, E_2)$ be graphs.
Then $(V_2, E_2)$ is a \emph{covering graph} of $(V_1, E_1)$ if there is a
surjection $f: V_2 \to V_1$ such that for each $v \in V_2$, the restriction
$f_{|N[v]}$ is a bijection.
In that case, $f$ is called a \emph{covering map}.
We note that derived graphs can be seen as a covering graphs of the base graph,
but directed.
The book \cite{gross1987topological} provides a basic introduction to covering
graphs in Chapter 2.
Covering graphs are more generally known as covering spaces in topology.
\end{remark}

\begin{notation}
\label{not:arrow}
For vertices $u$ and $v$, we use the notation $u \too v$ to denote that there
exists a directed walk from $u$ to $v$.
\end{notation}

\begin{lemma} \label{lem:strong}
The derived digraph $\graphf{D}_\times$ is a disjoint union of strong
components.
\end{lemma}
\begin{proof}
Select a vertex $(\graphf{u}, a)$, and take another
vertex $(\graphf{v},b)$ such that there is a path $(\graphf{u}, a) \too
(\graphf{v}, b)$ in $\graphf{D}_\times$.
Since $\graphf{D}$ is strongly
connected, there is a path $(\graphf{v}, b) \too
(\graphf{u},c)$ in $D_\times$ for some $c \in G$.
This implies that $(\graphf{u}, a) \too (\graphf{u},c)$.
We are done if we can show that $(\graphf{u}, c) \too (\graphf{u},a)$.

Since there is a path $(\graphf{u}, a) \too (\graphf{u}, c)$,
we know that for
any positive integer $j$, there is a path $(\graphf{u}, a) \too
(\graphf{u}, a + j(- a + c))$, which is found by
repeating the path in $\graphf{D}$.
In a finite digraph we will eventually get $g>j>0$ with $a + j(-
a + c) = a + g(- a + c)$, thus
$j(- a + c) = g(- a + c)$ and
$(g-j)(- a + c) = 0$.
We conclude that
\begin{align*}
(\graphf{u}, a) &\too (\graphf{u}, a + (- a + c)) \\
  &\too \cdots \\
  &\too (\graphf{u}, a + (g-j)(- a + c)) =
    (\graphf{u}, a).
  \qedhere
\end{align*}
\end{proof}

\begin{lemma} \label{lem:auto}
For each $\graphf{v} \in V(\graphf{D})$ and $a,b \in G$, there is a digraph
automorphism $f$ on $\graphf{D}_\times$ with $f(\graphf{v}, a) = (\graphf{v},
b)$.
In particular, the strong components of $\graphf{D}_\times$ are isomorphic.
\end{lemma}
\begin{proof}
Let $f:V(\graphf{D})\times G \to V(\graphf{D}) \times G$ be defined
$f(\graphf{v}, c) = (\graphf{v}, b-a+c)$.
We have $f(\graphf{v},a) = (\graphf{v}, b)$, and clearly $f$ is a bijection.
Take an arc from $(\graphf{u}, c)$ to $(\graphf{w}, d)$.
Then $c + \graphf{w}_\sigma = d$, so $b-a+c + \graphf{w}_\sigma = b-a+d$,
so there is also an arc from
$f(\graphf{u}, c)$ to $f(\graphf{w}, d)$.
This automorphism is mentioned in \cite[\S 2.2.1]{gross1987topological}.

The second claim follows since every strong component contains a vertex
$(\graphf{v}, c)$ for some $c \in G$, which follows from the strong
connectedness of $\graphf{D}$.
\end{proof}

Aperiodicity of $D_\times$ does not follow from aperiodicity of $D$, as shown
in Example \ref{exa:notaperiodic}, so it must be verified separately.
% given a digraph D with period p, we can get p different aperiodic digraphs
%   with vertex sets that partition the vertex set of D
%   Section 3.1 in https://web.archive.org/web/20170813010718/http://compphysiol.mi.fu-berlin.de/teaching/downloads/SS05_MP/MarkovChainScript2005-01-26.pdf
% a general theorem guaranteeing aperiodicity would be nice and I don't think
%   I'm going to find one. so it could be stated as an open problem. but it's
%   not needed for, say, subword patterns, so it's probably not valuable
%   enough for that distinction.

\begin{example} \label{exa:notaperiodic}
The condition of aperiodicity of $\graphf{D}_\times$ cannot be transfered from
$\graphf{D}$.
For example, if $a \in G$ has order at least $3$ and if $E(D) = \{(a,a)\}$ then
$D_\times$ is periodic.
Figure \ref{fig:apD} shows a less trivial counterexample digraph $\graphf{D}$.
\qedhere

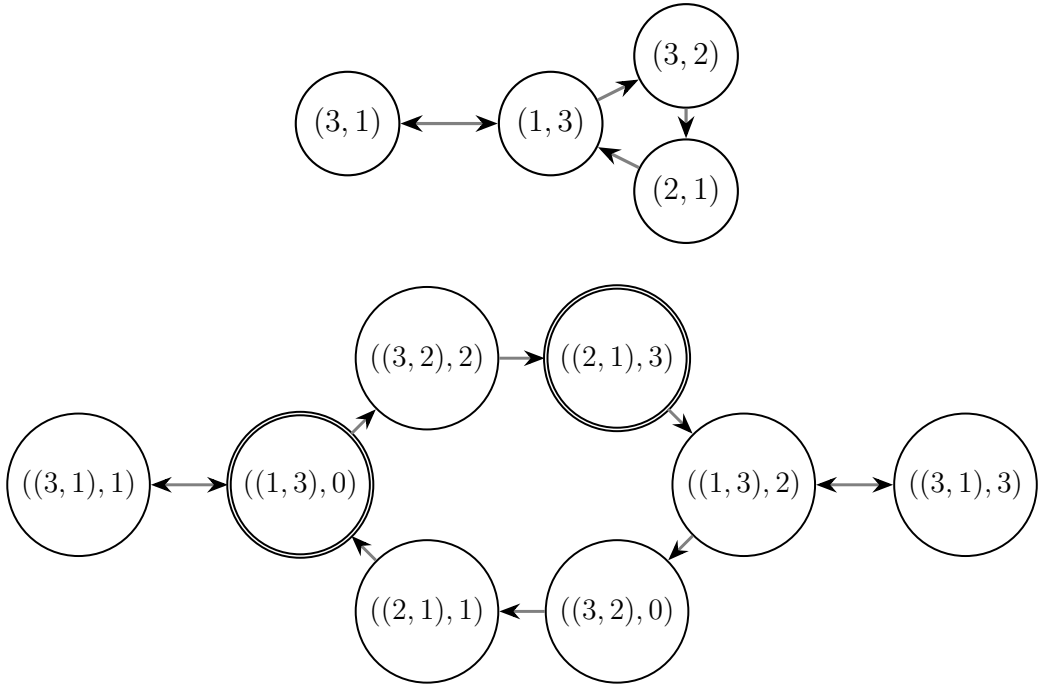
\begin{figure}
%\centering \includegraphics[width=25em]{apD}
\centering
\begin{tikzpicture}[scale=0.9]
\begin{scope}[every node/.style={circle,thick,draw}]
    \node (A) at (-1,0) {$(3,1)$};
    \node (B) at (2,0) {$(1,3)$};
    \node (C) at (4,1) {$(3,2)$};
    \node (D) at (4,-1) {$(2,1)$};
\end{scope}

\begin{scope}[>={Stealth[black]},
              every node/.style={fill=white,circle},
              every edge/.style={draw=gray,very thick}]
    \path [<->] (A) edge (B);
    \path [->] (B) edge (C);
    \path [->] (C) edge (D);
    \path [->] (D) edge (B);

    %\path [->] (B) edge[bend right=60] (E);
\end{scope}
\end{tikzpicture}
\vspace{0.5cm}

%\centering \includegraphics[width=40em]{apDx}
\centering
\resizebox{5.5in}{!}{
\begin{tikzpicture}[scale=0.9]
\begin{scope}[every node/.style={circle,thick,draw}]
    \node (A) at (-1.5,0) {$((3,1),1)$};
    \node[double] (B) at (2,0) {$((1,3),0)$};
    \node (C) at (4,2) {$((3,2),2)$};
    \node (D) at (4,-2) {$((2,1),1)$};

    \node (E) at (12.5,0) {$((3,1),3)$};
    \node (F) at (9,0) {$((1,3),2)$};
    \node[double] (G) at (7,2) {$((2,1),3)$};
    \node (H) at (7,-2) {$((3,2),0)$};
\end{scope}

\begin{scope}[>={Stealth[black]},
              every node/.style={fill=white,circle},
              every edge/.style={draw=gray,very thick}]
    \path [<->] (A) edge (B);
    \path [->] (B) edge (C);
    \path [->] (C) edge (G);
    \path [->] (D) edge (B);
    \path [->] (G) edge (F);
    \path [->] (F) edge (H);
    \path [->] (H) edge (D);
    \path [<->] (F) edge (E);

    %\path [->] (B) edge[bend right=60] (E);
\end{scope}
\end{tikzpicture}
}

\caption{An aperiodic strongly connected digraph $\graphf{D}$ (above) with
  vertices in $\mathbb{Z}_4^2$ such that $\graphf{D}_\times$ (one component
  shown below) has period $2$.
  Examples with connected $\graphf{D}_\times$ exist as well, such as
  the above $\graphf{D}$ over $\mathbb{Z}_8$ with $7$ replacing $3$ and $3$
  replacing $2$.
\label{fig:apD}}
\end{figure}
\end{example}

% offdef
\begin{notation}
\label{not:bigoh}
For real sequences $f(n), g(n)$, the notation $f(n) = O(g(n))$ means
there is $c>0$ such that for all sufficiently large $n$ we have
$|f(n)| \leq c g(n)$.
The notation $f(n) = \Theta(g(n))$ means there are $c,d > 0$ such that for
all sufficiently large $n$ we have
$|f(n)| \leq c g(n)$ and $|f(n)| \geq d g(n)$.
\end{notation}

The following basic result applies the Perron-Frobenius theorem to asymptotic
counting.

\begin{proposition} \label{pro:tmm}
Let $M$ be a nonzero $n \times n$ adjacency matrix of a strongly connected and
aperiodic digraph.
Then if $\alpha, \beta \in \mathbb{R}^n$, we have
\[ \alpha^\top M^m \beta =
  (\alpha \cdot v_\lambda) (u_\lambda \cdot \beta) \lambda^m
  (1 + O(\theta^m)), \qquad m \to \infty, \]
where $\lambda \geq 1$ is the largest-magnitude eigenvalue of $M$,
$v_\lambda$ is a positive $\lambda$-eigenvector of $M$, $u_\lambda$ is a
positive $\lambda$-eigenvector of $M^\top$ such that
$v_\lambda \cdot u_\lambda = 1$, and $0 \leq \theta < 1$.
\end{proposition}
\begin{proof}
By \cite[Proposition 2.4]{MR0423039}, any largest-magnitude eigenvalue
$\lambda$ of $M$ satisfies $|\lambda| \geq 1$.
% https://math.stackexchange.com/q/907633/3456
We use a few other facts from linear algebra covered in e.g.\
\cite{lanotes, meyer2000matrix}.
By the Perron-Frobenius theorem, we conclude there is a unique
largest-magnitude eigenvalue $\lambda > 0$ and $\lambda$ has
multiplicity $1$.
Furthermore, $M$ has Jordan decomposition
\[
M =
\begin{bmatrix}
    \vert &  & \\
    v_{\lambda}   & * &  \\
    \vert & &
\end{bmatrix}
\left[\begin{array}{@{}c|c@{}}
  \lambda
  & \bigzero \\
\hline
  \bigzero &
  \begin{matrix}
   &  & \\
   & B & \\
   &  & 
  \end{matrix}
\end{array}\right]
\begin{bmatrix}
    \text{---} \hspace{-0.2cm} & u_\lambda & \hspace{-0.2cm} \text{---} \\
     & * & \\
     &   &
\end{bmatrix},
\]
where $v_\lambda$ is a positive $\lambda$-eigenvector of $M$,
$u_\lambda$ is a positive $\lambda$-eigenvector of $M^\top$,
and $B$ is a block-diagonal matrix with spectral radius $0 \leq r < \lambda$.
The fact that $v_\lambda \cdot u_\lambda = 1$ follows once we note that the
first and last matrices in a Jordan decomposition are inverses.
Taking powers, we have
\[
M^m =
\begin{bmatrix}
    \vert &  & \\
    v_{\lambda}   & * &  \\
    \vert & &
\end{bmatrix}
\left[\begin{array}{@{}c|c@{}}
  \lambda^m
  & \bigzero \\
\hline
  \bigzero &
  \begin{matrix}
   &  & \\
   & B^m & \\
   &  & 
  \end{matrix}
\end{array}\right]
\begin{bmatrix}
    \text{---} \hspace{-0.2cm} & u_\lambda & \hspace{-0.2cm} \text{---} \\
     & * & \\
     &   &
\end{bmatrix},
\]
% offdef
where entries of $B^m$ are $O(r^m)$.
% https://math.stackexchange.com/questions/910635/power-of-a-matrix-given-its-jordan-form
The result is now immediate with $\theta = r/\lambda$.
\end{proof}

The essential idea of
Proposition \ref{pro:tmm} is quite classical, see e.g.\ \cite[Corollary
V.1]{ac}.

It is sometimes useful to know more about the growth rate of the number of
walks.

\begin{proposition} \label{pro:2vert}
Let $M$ be an  $n \times n$ adjacency matrix of a strongly connected and
aperiodic digraph.
% offdef
If $n \geq 2$ then all entries of $M^m$ are $\Theta(B^m)$, where $B>1$.
\end{proposition}
\begin{proof}
Let $v_1$ be a vertex in the digraph.
Since the digraph is aperiodic, there are two distinct cycles $C_1$ and $C_2$
starting from $v_1$; let their lengths be $c_1, c_2$.
Let $\ell = \lcm(c_1, c_2)$.
Construct a walk of length $\alpha \ell$ by choosing $\alpha$ segments of
length $\ell$ where each segment is either $C_1$ repeated or $C_2$ repeated.
Then the number of walks of length $\alpha \ell$ is at least
$(2^{1/\ell})^{\alpha \ell}$.
By Proposition \ref{pro:tmm} the number of walks of length $m$ is $\Theta(B^m)$
so we must have $B > 1$.
\end{proof}

\begin{definition}
A de Bruijn subgraph $D$ is \emph{regular} if $D$ is strongly connected,
contains at least $2$ vertices, and its derived digraph $D_\times$ is
aperiodic.
\end{definition}

\begin{notation}
\label{not:asympt}
If two real sequences satisfy $\lim_{n \to \infty} f(n)/g(n) = 1$,
we write $f(n) \sim g(n)$ and say $f(n)$ is asymptotic to $g(n)$.
\end{notation}

\begin{proposition} \label{prop:count}
Suppose $D$ is regular and $p(m; D, \Psi, \Phi) \sim A \cdot B^m$.
Then either $p_s(m; D, \Psi, \Phi) = 0$ or
\[p_s(m; D, \Psi, \Phi) = C_s \cdot B^m(1 + O(\theta^m)), \qquad m \to \infty\]
where $C_s > 0$ can be computed from $\graphf{D}_\times$ and
$0 \leq \theta < 1$.
\end{proposition}
\begin{proof}
Since the strong components of $D_\times$ are isomorphic by
Lemma \ref{lem:strong} and Lemma \ref{lem:auto}, they each have
the same adjacency matrix and the same eigenvalues.
The only difference between a composition of $s$ and an arbitrary
composition is the allowed finish vertices.
Thus by Proposition \ref{prop:dtimescount} and Proposition \ref{pro:tmm} we
conclude
$p_s(m; D, \Psi, \Phi) = C_s \cdot B^m(1 + O(\theta^m))$
where $C_s = 0$ only if $p_s(m; D, \Psi, \Phi) = 0$.
The latter case occurs if there is no strong component of $D_\times$ containing
vertices from both $\Psi_\times$ and $\Phi_s$.
\end{proof}

% *dependence on s*
The asymptotics of $p_s(m; D, \Psi, \Phi)$ are now established.
However, in some cases we can usefully simplify the constants involved.

\begin{definition}
If $A$ is an $m \times n$ matrix and $B$ is a $p \times q$ matrix, then
the \emph{Kronecker product} $A \otimes B$ is the $mp \times nq$ matrix
$C$ such that $[C]_{p(r-1)+v, q(s-1) + w} = [A]_{r,s} [B]_{v,w}$.
Visually,
\[
A \otimes B =
  \begin{bmatrix}
  [A]_{1,1} {B} & \cdots &
  [A]_{1,n}{B} \\ \vdots & \ddots & \vdots \\ [A]_{m,1} {B} & \cdots &
  [A]_{m,n} {B}
  \end{bmatrix}.
\]
\end{definition}

Basic properties of the Kronecker product are discussed in \cite[Chapter
4]{hornjohnson}.
We quote a couple of relevant facts.

\begin{proposition}[Lemma 4.2.10 in \cite{hornjohnson}] \label{pro:horn1}
Let $F$ be a field.
Let $A \in F^{m \times n}, B \in F^{p \times q}, C \in F^{n \times k},
D \in F^{q \times r}$.
Then $(A \otimes B)(C \otimes D) = AC \otimes BD$.
\end{proposition}

\begin{proposition}[Equation 4.2.8 in \cite{hornjohnson}] \label{pro:horn2}
Let $F$ be a field.
We have
\[ A \otimes (B + C) = A \otimes B + A \otimes C \]
for all $A \in F^{m \times n}$ and $B,C \in F^{p \times q}$.
\end{proposition}

% offdef
\begin{notation}
\label{not:iverson}
For a logical statement $\phi$, the notation $[\phi]$ stands for $1$ if $\phi$
is true and $0$ otherwise.
\end{notation}

We next establish the structure of derived digraphs $D_\times$ in terms of
the Kronecker product.

\begin{lemma} \label{lem:permmatrix}
Say $|V(D)| = \alpha$ and fix a vertex ordering $v_1, \ldots, v_\alpha$.
Let $M$ be the adjacency matrix of $\graphf{D}$ with respect to this ordering.
Also fix an ordering on $G = \{a_1, a_2, \ldots, a_\beta\}$ where $a_1=0$.
Finally, define a vertex ordering on $\graphf{D}_\times$ as
\[
(v_1, a_1), \ldots, (v_\alpha, a_1), (v_1, a_2), \ldots, (v_\alpha, a_2),
\ldots, (v_1, a_\beta), \ldots, (v_\alpha, a_\beta).
\]
Let $M_\times$ be the adjacency matrix of $\graphf{D}_\times$ with respect to
this ordering.
For each $a \in G$, define the $\alpha \times \alpha$ matrix $M_a$ and
$\beta \times \beta$ matrix $P_a$ such that
\begin{align*}
[M_a]_{i,j} &= [v_j(\sigma) = a, (v_i,v_j) \in E(D)], \\
[P_a]_{i,j} &= [a_i + a = a_j].
\end{align*}

Then
$M_\times = \sum_{a \in G} P_a \otimes M_a$ and
$M = \sum_{a \in G} M_a$.
\end{lemma}
\begin{proof}
We have
\[
[M_\times]_{i + \alpha(j-1),  k + \alpha(l-1)} =
  [(v_i, v_k) \in E(D), a_j + v_k(\sigma) = a_l]
\]
and
\begin{align*}
\left[ \sum_{a \in G} P_a \otimes M_a \right]_{i + \alpha(j-1),  k + \alpha(l-1)}
&= \sum_{a \in G} [P_a]_{j,l} [M_a]_{i,k} \\
&= \sum_{a \in G} [a_j + a = a_l][v_k(\sigma) = a, (v_i, v_k) \in E(D)] \\
&= [(v_i, v_k) \in E(D), a_j + v_k(\sigma) = a_l].
\end{align*}

Clearly
\[ [M]_{i,j} = [(v_i, v_j) \in E(D)] = \sum_{a \in G} [v_j(\sigma) =
  a, (v_i, v_j) \in E(D)] = \sum_{a \in G} [M_a]_{i,j}. \qedhere \]
\end{proof}

\begin{theorem} \label{thm:equal}
Assume
\begin{itemize}
\item for some $\graphf{v} \in V(\graphf{D})$ we have that for all $a \in G$
  there is a legal composition starting and ending with $\graphf{v}$ with total
  $a$, and
\item for some $\graphf{u} \in V(\graphf{D})$ the set
  $$\left\{ m : \exists \text{ a walk } x =
  (\graphf{u}, \graphf{v}, \ldots, \graphf{w},\graphf{u})
  \text{ of length } m+1, \sum x = \sum \graphf{u} \right\},$$
  has a GCD of $1$, where $\sum x$ is the total of the composition
  corresponding to $x$.
\end{itemize}
Assume $p(m; D, \Psi, \Phi) \sim A \cdot B^m$.
Then
\[
p_a(m; D, \Psi, \Phi) =
  \frac{A}{|G|} \cdot B^m(1 + O(\theta^m)), \qquad m \to \infty,\, 0 \leq \theta
  < 1.
\]
\end{theorem}
\begin{proof}
%offdef
From the first condition we know there is a single strong component by
Lemma \ref{lem:strong},
i.e.\
$\graphf{D}_\times$ is strongly connected.
The second condition ensures that this component is aperiodic.
This allows us to conclude that the Perron-Frobenius theorem applies directly
to $D_\times$.

Let $M, M_\times, M_a, P_a$ be as in
Lemma \ref{lem:permmatrix}.

Let $\lambda>0$ be the dominant eigenvalue of $M$,
let $v_\lambda$ be an associated positive eigenvector, and let $u_\lambda$ be
an associated positive left eigenvector (eigenvector of $M^\top$).
Let $\xi \in \mathbb{R}^\beta$ be the all-$1$ vector $\rvect{1, 1, \cdots, 1}$.

We claim that $\xi \otimes v_\lambda$ is an eigenvector of $M_\times$ with
eigenvalue $\lambda$.
First,
by Proposition \ref{pro:horn1},
$(P_a \otimes M_a)(\xi \otimes v_\lambda) = P_a \xi \otimes M_a v_\lambda$.
Since $P_a$ is a permutation matrix, we have $P_a \xi = \xi$.
Thus
\begin{align*}
M_\times (\xi \otimes v_\lambda)
  &= \left(\sum_{a \in G} P_a \otimes M_a\right) (\xi \otimes v_\lambda) \\
  &= \sum_{a \in G} (P_a \otimes M_a) (\xi \otimes v_\lambda) \\
  &= \sum_{a \in G} P_a \xi \otimes M_a v_\lambda \\
  &= \sum_{a \in G} \xi \otimes M_a v_\lambda.
\end{align*}
By Proposition \ref{pro:horn2},
$\sum_{a \in G} \xi \otimes M_a v_\lambda =
\xi \otimes \sum_{a \in G}M_a v_\lambda$.
We conclude
\begin{align*}
M_\times (\xi \otimes v_\lambda) &= \xi \otimes \sum_{a \in G} M_a v_\lambda \\
  &= \xi \otimes \left(\sum_{a \in G} M_a\right) v_\lambda \\
  &= \xi \otimes M v_\lambda \\
  &= \xi \otimes \lambda v_\lambda \\
  &= \lambda \xi \otimes v_\lambda.
\end{align*}
Similarly we have that $\xi \otimes u_\lambda$ is a left eigenvector for
$M_\times$ with eigenvalue $\lambda$.

By Proposition \ref{pro:tmm}, we know that
\[
p_a(m; D, \Psi, \Phi) = \psi_\times^\top M_\times^{m-\sigma} \phi_a
  = C_a \lambda^m(1 + O(\theta^m)),
\]
where
$C_a = c (\psi_\times \cdot (\xi \otimes v_\lambda)) ((\xi \otimes
u_\lambda) \cdot \phi_a)$ for
some fixed scaling factor $c>0$.

Suppose $a = a_i$;
then $\phi_a = e_i \otimes \phi$.
Thus
\[
(\xi \otimes u_\lambda) \cdot \phi_a
= (\xi \otimes u_\lambda) \cdot (e_i \otimes \phi)
= (\xi \cdot e_i) \otimes (u_\lambda \cdot \phi) = u_\lambda \cdot \phi.\]
Since $u_\lambda \cdot \phi$ does not depend on $a$, the proof is now complete.
\end{proof}

An alternative proof approach involves using Lemma \ref{lem:auto} and the fact
that automorphisms of $M_\times$ correspond to permutation matrices $P$ such
that $P M_\times P^{-1} = M_\times$.
% \cite[Lemma 8.1.1]{godsil2013algebraic}
% this immediately shows that eigenvectors of M_\times are of the form
% v \otimes [1,1,...1] for some v but doesn't say what v is.
% i think we could show that v must be an eigenvector of M by looking at
% relative asymptotic counts for compositions with any total finishing on
% each vertex of D which are counted by both M and M_\times.
% the approach used at least tells us the fact about M_x that
% $M_\times = \sum_{a \in G} P_a \otimes M_a$

\begin{corollary}
Assume the conditions of Theorem \ref{thm:equal}.
Construct a probability space from $\mathcal{P}(m; D, \Psi, \Phi)$ and
the uniform probability measure.
Then for $a \in G$, let $\mathbb{P}_m(a)$ be the probability that an element
drawn randomly from
$\mathcal{P}(m; D, \Psi, \Phi)$ has total $a$.
We have for any $a \in G$
\[ \mathbb{P}_m(a) \to \frac{1}{|G|}, \qquad m \to \infty, \]
or in other words, $\mathbb{P}_m$ converges strongly to the uniform measure
on $G$.
\end{corollary}
\begin{proof}
Direct from Theorem \ref{thm:equal}.
\end{proof}

\begin{example}
We show a case where the strong connectedness condition in Theorem
\ref{thm:equal} is required.
Let $D$ be the digraph given in Figure \ref{fig:equalfail}, where
$G = \mathbb{Z}_2$.
Assume $\Psi = \Phi = V(D)$.

\begin{figure}
\centering

\begin{tikzpicture}[scale=3.5]
{
\small
\begin{scope}[every node/.style={circle,thick,draw}]
    \node (A) at (-.8,0) {(0,0,0,0)};
    \node (B) at (0,0) {(0,0,0,1)};
    \node (C) at (0.5,-0.5) {(0,0,1,1)};
    \node (D) at (0,-1) {(0,1,1,0)};
    \node (E) at (-.8,-1) {(1,1,0,0)};
    \node (F) at (-1.3,-0.5) {(1,0,0,0)};
\end{scope}
}

\begin{scope}[>={Stealth[black]},
              every node/.style={fill=white,circle},
              every edge/.style={draw=gray,very thick}]
    \path [->] (A) edge (B);
    \path [->] (B) edge (C);
    \path [->] (C) edge (D);
    \path [->] (D) edge (E);
    \path [->] (E) edge (F);
    \path [->] (F) edge (A);
    \path (A) edge [loop left] (A);

    %\path [->] (B) edge[bend right=60] (E);
\end{scope}
\end{tikzpicture}

\vspace{1.2cm}

\begin{tikzpicture}[scale=4]
{
\small
\begin{scope}[every node/.style={circle,thick,draw}]
    \node (A) [double] at (-.8,0) {((0,0,0,0),0)};
    \node (B) [double] at (0,0) {((0,0,0,1),1)};
    \node (C) [double] at (0.5,-0.5) {((0,0,1,1),0)};
    \node (D) [double] at (0,-1) {((0,1,1,0),0)};
    \node (E) [double] at (-.8,-1) {((1,1,0,0),0)};
    \node (F) at (-1.3,-0.5) {((1,0,0,0),0)};
\end{scope}
}

\begin{scope}[>={Stealth[black]},
              every node/.style={fill=white,circle},
              every edge/.style={draw=gray,very thick}]
    \path [->] (A) edge (B);
    \path [->] (B) edge (C);
    \path [->] (C) edge (D);
    \path [->] (D) edge (E);
    \path [->] (E) edge (F);
    \path [->] (F) edge (A);
    \path (A) edge [loop left] (A);

    %\path [->] (B) edge[bend right=60] (E);
\end{scope}
\end{tikzpicture}

\vspace{.3cm}

\begin{tikzpicture}[scale=4]
{
\small
\begin{scope}[every node/.style={circle,thick,draw}]
    \node (A) at (-.8,0) {((0,0,0,0),1)};
    \node (B) at (0,0) {((0,0,0,1),0)};
    \node (C) at (0.5,-0.5) {((0,0,1,1),1)};
    \node (D) at (0,-1) {((0,1,1,0),1)};
    \node (E) at (-.8,-1) {((1,1,0,0),1)};
    \node (F) [double] at (-1.3,-0.5) {((1,0,0,0),1)};
\end{scope}
}

\begin{scope}[>={Stealth[black]},
              every node/.style={fill=white,circle},
              every edge/.style={draw=gray,very thick}]
    \path [->] (A) edge (B);
    \path [->] (B) edge (C);
    \path [->] (C) edge (D);
    \path [->] (D) edge (E);
    \path [->] (E) edge (F);
    \path [->] (F) edge (A);
    \path (A) edge [loop left] (A);

    %\path [->] (B) edge[bend right=60] (E);
\end{scope}
\end{tikzpicture}

\caption{A base digraph $D$ (above) and derived digraph $D_\times$ with $2$
  strong components (below).
  The vertices of $D$ are $4$-tuples over $\mathbb{Z}_2$.
\label{fig:equalfail}}
\end{figure}
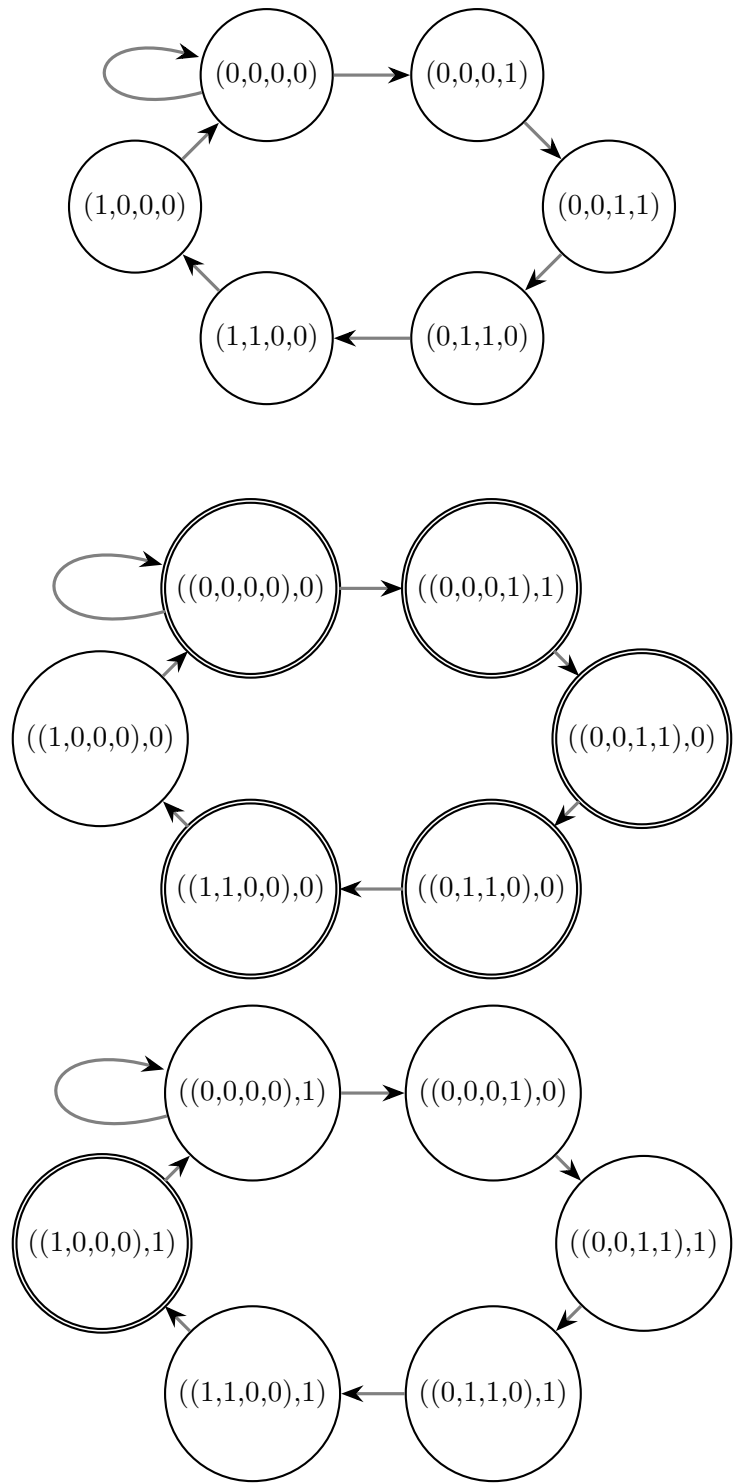

Let $M^{\langle 1 \rangle}, M^{\langle 2 \rangle}$ be adjacency matrices of the
two strong components of $D_\times$, under a particular vertex ordering.
We have
\[
M^{\langle 1 \rangle} = M^{\langle 2 \rangle} = \left[
\begin{array}{cccccc}
 1 & 1 & 0 & 0 & 0 & 0 \\
 0 & 0 & 1 & 0 & 0 & 0 \\
 0 & 0 & 0 & 1 & 0 & 0 \\
 0 & 0 & 0 & 0 & 1 & 0 \\
 0 & 0 & 0 & 0 & 0 & 1 \\
 1 & 0 & 0 & 0 & 0 & 0 \\
\end{array}
\right].
\]
From a Jordan decomposition, we get a left eigenvalue
\[u_\lambda \doteq
  \rvect{0.368841, 0.286991, 0.223305, 0.173751, 0.135194, 0.105193},\]
and a right eigenvector
\[v_\lambda \doteq \rvect{1.2852, 0.366538, 0.471074, 0.605423, 0.77809, 1.0},\]
% offdef
corresponding to the dominant eigenvalue $\lambda \doteq 1.2852$.
Also
\[ \psi^{\langle 1 \rangle}_\times = \rvect{1,1,1,1,1,0}^\top,\,\,
 \phi^{\langle 1 \rangle}_0 = \rvect{1,0,1,1,1,1}^\top,\]
\[\phi^{\langle 1 \rangle}_1 = \rvect{0,1,0,0,0,0}^\top, \]
and
\[ \psi^{\langle 2 \rangle}_\times = \rvect{0,0,0,0,0,1}^\top,\,\,
\phi^{\langle 2 \rangle}_0 = \rvect{0,1,0,0,0,0}^\top, \]
\[ \phi^{\langle 2 \rangle}_1 = \rvect{1,0,1,1,1,1}^\top. \]
We compute
\begin{align*}
p_0(m; D) &= {\psi^{\langle 1 \rangle}_\times}^\top
  (M^{\langle 1 \rangle})^{m-4}
  \phi^{\langle 1 \rangle}_0 +
{\psi^{\langle 2 \rangle}_\times}^\top
  (M^{\langle 2 \rangle})^{m-4}
  \phi^{\langle 2 \rangle}_0 \\
&\sim \lambda^{m-4} ( ({\psi^{\langle 1 \rangle}_\times} \cdot v_\lambda)
  (u_\lambda \cdot \phi^{\langle 1 \rangle}_0) +
  ({\psi^{\langle 2 \rangle}_\times} \cdot
  v_\lambda) (u_\lambda \cdot \phi^{\langle 2 \rangle}_0)) \\
&\doteq \lambda^{m-4} (3.50632 \cdot 1.00628 + 1.0 \cdot 0.286991) \\
&= 3.81533 \cdot \lambda^{m-4}
\end{align*}
and
\begin{align*}
p_1(m; D) &= {\psi^{\langle 1 \rangle}_\times}^\top
  (M^{\langle 1 \rangle})^{m-4}
  \phi^{\langle 1 \rangle}_1 +
{\psi^{\langle 2 \rangle}_\times}^\top
  (M^{\langle 2 \rangle})^{m-4}
  \phi^{\langle 2 \rangle}_1 \\
&\sim \lambda^{m-4} ( (\psi^{\langle 1 \rangle}_\times \cdot v_\lambda)
  (u_\lambda \cdot \phi^{\langle 1 \rangle}_1) +
  (\psi^{\langle 2 \rangle}_\times \cdot
  v_\lambda) (u_\lambda \cdot \phi^{\langle 2 \rangle}_1)) \\
&\doteq \lambda^{m-4} (3.50632 \cdot 0.286991 + 1.0 \cdot 1.00628) \\
&= 2.01256 \cdot \lambda^{m-4}.
\end{align*}
So indeed the conclusion of Theorem \ref{thm:equal} does not hold.
\end{example}

% *helpers for thm \ref{thm:equal}*

\begin{lemma} \label{lem:abelian}
Assume $G$ is abelian.
Define a $\graphf{D}_\times$-automorphism
$f$ by $f(\graphf{v}, b) = (\graphf{v}, a+b)$ for some $a \in G$.
If $f$ maps any vertex to its own strong component, then $f$ maps
all vertices to their own strong component.
\end{lemma}
\begin{proof}
Let $(\graphf{u}, c), (\graphf{v}, b)$ be arbitrary vertices and
say $(\graphf{v}, b) \too (\graphf{v}, a + b)$ in $\graphf{D}_\times$.
We seek to show that
$(\graphf{u}, c) \too (\graphf{u}, a + c)$.

There is some $(\graphf{u}, d)$ in the same strong component as
$(\graphf{v}, b)$ and so $(\graphf{u}, d) \too (\graphf{u}, a + d)$.

The automorphism $g(\graphf{w}, r) = (\graphf{w}, c-d+r)$ maps
$(\graphf{u}, d)$ to $(\graphf{u}, c)$ and
$(\graphf{u}, a+d)$ to $(\graphf{u}, c-d+a+d) = (\graphf{u}, a+c)$.

Thus $(\graphf{u}, d) \too (\graphf{u}, a+d)$ implies
$g(\graphf{u}, d) = (\graphf{u}, c)
  \too g(\graphf{u}, a+d) = (\graphf{u}, a+c)$.
This is illustrated in Figure \ref{fig:abeliandiag}. \qedhere

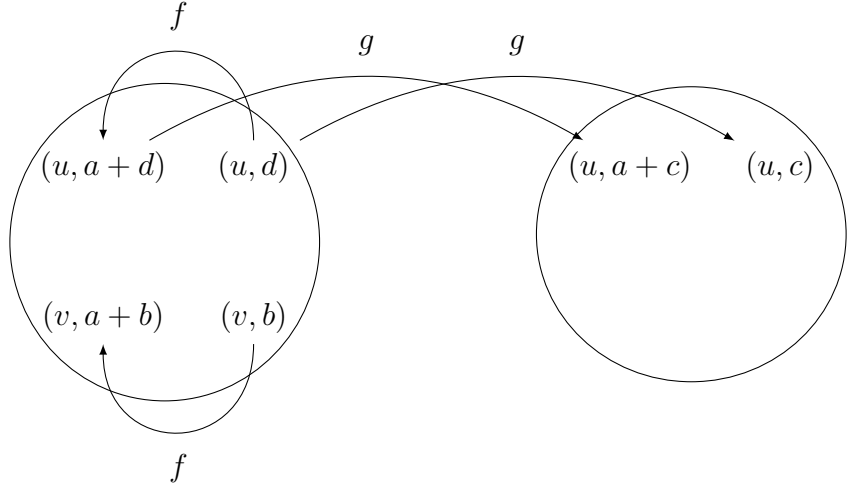
\begin{figure}
\centering
\begin{tikzpicture}[
  every node/.style={on grid},
  every fit/.style={draw,ellipse,text width=75pt},
  >=latex
]

% set A
\node [] (a) {$(u, a + d)$};
\node [below =2cm of a] (c) {$(v, a + b)$};
\node [right =2cm of a] (d) {$(u, d)$};
\node [right =2cm of c] (e) {$(v, b)$};

% set B
\node[right=7cm of a] (x) {$(u, a + c)$};
\node[right=2cm of x] (y) {$(u,c)$};
\node[below =2cm of x] (z) {};
\node[below =2cm of y] (zz) {};

% the arrows
\draw[->, looseness=2] (e) to [out=-90,in=-90] node[label=below:$f$] {} (c);
\draw[->, looseness=2] (d) to [out=90,in=90] node[label=above:$f$] {} (a);
\draw[->] (a) to [out=30,in=150] node[label=above:$g$] {} (x);
\draw[->] (d) to [out=30,in=150] node[label=above:$g$] {} (y);

% the boxes around the sets
\begin{pgfonlayer}{background}
\node[fit= (a) (d) (c) ] {};
\node[fit= (x) (y) (z) ] {};
\end{pgfonlayer}
\end{tikzpicture}
\caption{Strong components in the proof of Lemma \ref{lem:abelian}.
\label{fig:abeliandiag}}
\end{figure}
\end{proof}

The following is a useful characterization of strong connectedness of
$\graphf{D}_\times$ for abelian $G$.

\begin{lemma} \label{lem:addany}
Assume $G$ is abelian.
Let $A$ be a generating set for $G$, i.e.\ $\langle A \rangle = G$.
If for all $a_i \in A$ there is a vertex $(\graphf{v}, 0) \in
V(\graphf{D}_\times)$ such that $(\graphf{v}, 0) \too (\graphf{v}, a_i)$ in
$\graphf{D}_\times$, then $\graphf{D}_\times$ is strongly connected.
\end{lemma}
\begin{proof}
We show that for any $\graphf{u} \in V(\graphf{D})$ and $r \in G$, the vertices
$(\graphf{u}, 0)$ and $(\graphf{u}, r)$ are in the same strong component
of $\graphf{D}_\times$.

Say $r = {j(1)}a_1 + \cdots + {j(p)}a_p$ for $a_i \in A, j(i) \in
\mathbb{Z}_{\geq 0}$.
We know from Lemma \ref{lem:abelian} that the $\graphf{D}_\times$-automorphism
$f_j(\graphf{v},s) = (\graphf{v}, a_j + s)$
maps strong components to themselves.
Thus the composition $f_r = f_1^{j(1)} \circ \cdots \circ f_p^{j(p)}$ also
maps strong components to themselves.
We conclude that $(\graphf{u}, 0)$ and
$f_r(\graphf{u}, 0) = (\graphf{u}, j(1)a_1 + \cdots + j(p)a_p) =
(\graphf{u}, r)$ belong to the same strong component.
\end{proof}

% *examples*
We now consider some examples of $\graphf{D}$.

We generalize Carlitz compositions as follows.
A sequence $x \in \SEQ_m(G)$ is a \emph{$d$-Carlitz composition} if
every subword $x(i), \ldots, x(i+d)$ contains no repeated part.
Thus Carlitz compositions are 1-Carlitz.
We note that unlike for integer compositions, we generally allow the identity
element as a part.
We also note that this definition is consistent with \cite{abelian} but
different from Definition 4.33 in \cite[p.\ 115]{cofc}.
Words with no equal adjacent letters are also called Smirnov words as in
\cite[Example III.24]{ac}.

\begin{lemma} \label{lem:carlitzD}
There is a de Bruijn subgraph $D$ with span $\sigma = d+1$ representing
$d$-Carlitz compositions such that $D_\times$ is strongly connected and
aperiodic, provided $|G| \geq d+2$.
\end{lemma}
\begin{proof}
% strong connectedness
Take as vertex set for $\graphf{D}$ all $(d+1)$-tuples of distinct elements
of $G$.
The allowed start and finish vertices are all of $V(D)$.
The strong connectedness of $\graphf{D}$ is established in (the proof of)
\cite[Corollary 2]{abelian};
% offdef
we include the argument here for completeness.
Let $u, v \in V(D)$ be distinct.
Let $w$ be a vertex such that
$w = (w(1), \ldots, w(j), v(j+1), \ldots, v(d+1))$, and assume
$u \too w$.
Clearly this is possible if $j=d+1$.
If this is true for some $j \leq d+1$
we seek to show that there is a vertex
$y = (y(1), \ldots, y(j-1), v(j), \ldots, v(d+1))$ such that
$w \too y$ and thus $u \too y$.
If $v(j) \not\in \{w(1), \ldots, w(j)\}$ then
$w \too y = (w(1), \ldots, w(j-1), v(j), \ldots, v(d+1))$.
If $v(j) \in \{w(1), \ldots, w(j)\}$, assume $w(r) = v(j)$.
Let $a\in G$ be an element not found
in $\{w(1), \ldots, w(j), v(j), \ldots, v(d+1)\}$.
Then
$w \too (w(1), \ldots, w(r-1), a, w(r+1), \ldots, w(j), v(j+1), \ldots, v(d+1))
\too y = (w(1), \ldots, w(r-1), a, w(r+1), \ldots, w(j-1), v(j), \ldots, v(d+1))$.
By induction, we conclude that $u \too w$ in the case $j=0$, i.e.\ $u \too v$.

To show strong connectedness of $\graphf{D}_\times$, we fix a vertex $(a_1,
\ldots, a_{d+1}) \in V(\graphf{D})$ and for any $s \in G$ exhibit a walk
from $((a_1, \ldots, a_{d+1}), 0)$ to $((a_1, \ldots, a_{d+1}), s)$.
Let $n$ be the order of $\Sigma a = a_1 +  \cdots +  a_{d+1}$.
We consider two cases.

Case 1: $s \not\in \{a_1, \ldots, a_{d+1}\}$.
The first step is to $(a_2, \ldots, a_{d+1}, s)$.
Follow this by the $(d+1)$-step path to $(a_1, \ldots, a_{d+1})$
Take the $(d+1)$-step path back to $(a_1, \ldots, a_{d+1})$ exactly $n-1$ times.
The total of this walk is
\[ s + \Sigma a + (n-1)\Sigma a = s, \]
thus
$((a_1, \ldots, a_{d+1}), 0) \too ((a_1, \ldots, a_{d+1}), s)$
in $\graphf{D}_\times$.

Case 2: $s = a_j$ for some $1 \leq j \leq d+1$.
Let $b$ represent some element of $G$ not in $\{a_1, \ldots, a_{d+1}\}$.
Begin with the $(d+1)$-step path to
$(b, a_1, \ldots, a_{j-1}, a_{j+1}, a_{d+1})$.
Let $n_1$ be the order of $b + a_1 + \cdots + a_{j-1} + a_{j+1} + a_{d+1}$.
Follow the $(d+1)$-step path back to
$(b, a_1, \ldots, a_{j-1}, a_{j+1}, a_{d+1})$ exactly $n_1 - 1$ times.
Traverse one arc to
$(a_1, \ldots, a_{j-1}, a_{j+1}, a_{d+1}, s)$,
then follow the $(d+1)$-step path to
$(a_1, \ldots, a_{j-1}, a_{j+1}, a_{d+1}, b)$.
Let $n_2$ be the order of $a_1 + \ldots + a_{j-1} + a_{j+1} + a_{d+1} + b$.
Take the $(d+1)$-step path back to $(a_1, \ldots, a_{j-1}, a_{j+1}, a_{d+1}, b)$
exactly $n_2-1$ times.
Finally take the $(d+1)$-step path to $(a_1, \ldots, a_{d+1})$
and cycle $(a_1, \ldots, a_{d+1})$ the suitable number of times.
The total of this walk is $0 + s + 0 + 0 = s$ so
$((a_1, \ldots, a_{d+1}), 0) \too ((a_1, \ldots, a_{d+1}), s)$
in $\graphf{D}_\times$.

% aperiodicity
We now turn to aperiodicity.
Take a $\graphf{D}$-vertex $\graphf{u} = (a_1, \ldots, a_{d+1})$ that does not
contain the part $0$.
We give two closed walks starting from $\graphf{u}$ with total $0$ and lengths
that differ by $1$.

Let $n$ be the order of $a_1 + \cdots + a_{d+1}$.
The first walk repeats the $(d+1)$-step cycle back to $u$ exactly $n$ times.
The second walk first takes a step to
$(\graphf{u}_2, \ldots, \graphf{u}_{d+1}, 0)$ followed by the
$(d+1)$-step path to $\graphf{u}$.
Then we cycle back to $\graphf{u}$ exactly $n-1$ times.
\end{proof}

% offdef
\begin{notation}
\label{not:fact}
We denote the falling factorial $n(n-1)\cdots (n-k+1)$ by
$n^{\underline{k}}$.
\end{notation}

\begin{proposition} \label{prop:carlitz}
The number of $d$-Carlitz $m$-compositions of $s \in G$ over a finite group $G$
is
\[
  \frac{1}{|G|} |G|^{\underline{d}} (|G| - d)^{m-d}(1 + O(\theta^m)),
  \qquad m \to \infty, 0 \leq \theta < 1,
\]
provided $|G| \geq d + 2$.
\end{proposition}
\begin{proof}
With Lemma \ref{lem:carlitzD}
we conclude that the conditions of Theorem \ref{thm:equal} are satisfied.

% explicit count
In $\graphf{D}$ each vertex has an out-degree of $|G| - d$.
This allows us to count walks in $\graphf{D}$ directly.
We have $V(\graphf{D}) = |G|^{\underline{d + 1}}$.
Thus the number of $m$-compositions represented by $\graphf{D}$ is
$|G|^{\underline{d+1}}(|G| -d)^{m-d-1} =
|G|^{\underline{d}}(|G| -d)^{m-d}$.
We conclude by applying Theorem \ref{thm:equal}.
\end{proof}

Figure \ref{fig:2carlitz0} shows randomly generated 2-Carlitz 100-compositions
over $\mathbb{Z}_5$.
Table \ref{tab:carS3} gives counts for Carlitz $m$-compositions of $a$ over
$S_3$.

\begin{figure}
\centering
\includegraphics[width=5.5in]{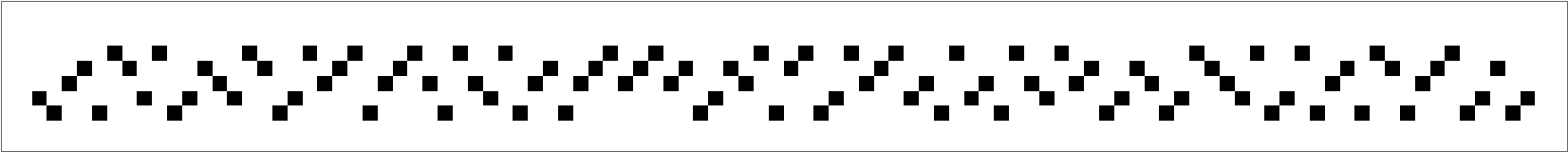}
\includegraphics[width=5.5in]{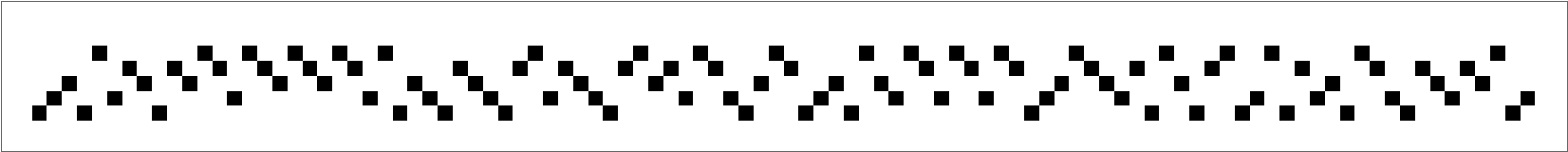}

\caption{Uniform-randomly generated 2-Carlitz $100$-compositions of $0$ (above)
  and $1$ (below) over $\mathbb{Z}_5$.
  \label{fig:2carlitz0}}
\end{figure}

\begin{table}
\centering
\begin{tabular}[c]{|c|rrr|}
\hline
\diagbox{$m$}{$a$} & $id$ & $(123)$ & $(12)$\\ \hline
3 &27 & 24 & 25 \\
4 &134 & 128 & 120 \\
5 &613 & 631 & 625 \\
6 &3096 & 3102 & 3150 \\
7 &15667 & 15604 & 15625 \\
8 &78224 & 78263 & 78000 \\
9 &390513 & 390681 & 390625 \\
10&1952696 & 1952402 & 1953750 \\
11&9765817 & 9765529 & 9765625 \\
12&48830424 & 48831663 & 48825000 \\
13&244140763 & 244140556 & 244140625 \\
14&1220690096 & 1220686202 & 1220718750 \\
15&6103512717 & 6103517079 & 6103515625 \\
16&30517650374 & 30517659188 & 30517500000 \\\hline
\end{tabular}
\caption{Counts of Carlitz $m$-compositions of some $a$ over $S_3$.}
\label{tab:carS3}
\end{table}

As in \cite{abelian} we say an $m$-composition $x$ is \emph{locally $d$-Mullen}
if no nonempty subword of $x$ of length at most $d$ has total $0$.

\begin{proposition}
The number of locally $d$-Mullen $m$-compositions of $a \in G$ over a finite
group $G$ is
\[
  \frac{1}{|G|} (|G|-1)^{\underline{d-1}}(|G|-d)^{m-d+1} (1 + O(\theta^m)),
  \qquad m \to \infty, 0 \leq \theta < 1,
\]
provided $|G| \geq d + 2$.

\end{proposition}
\begin{proof}
Let $D^{\langle 1 \rangle}$ be the digraph of span $d$ representing
locally $(d+1)$-Mullen compositions, and let $D^{\langle 1 \rangle}_\times$ be
the derived digraph.
Let $D^{\langle 2 \rangle}$ be the digraph of span $d+1$ representing
$d$-Carlitz compositions.
We note that
$|V(D^{\langle 1 \rangle}_\times)| = |V(D^{\langle 2 \rangle})|
= |G|^{\underline{d+1}}$,
recalling that the part $0$ is never allowed in a locally Mullen composition.
Define a function
$f: V(D^{\langle 1 \rangle}_\times) \to V(D^{\langle 2 \rangle})$
as follows:
\[
f((v_1, \ldots, v_d), a) =
\left(a- \sum_{j=1}^d v_j, a-\sum_{j=2}^d v_j, \ldots, a-v_d, a\right).
\]
A computation gives us that $f$ is a graph isomorphism from
$D^{\langle 1 \rangle}_\times$ to $D^{\langle 2 \rangle}$.
Thus the strong connectedness and aperiodicity of $D^{\langle 2 \rangle}$
established in Lemma \ref{lem:carlitzD} hold for $D^{\langle 1 \rangle}_\times$
as well and Theorem \ref{thm:equal} applies.

A part $x(i)$ in a locally $d$-Mullen compositions must not take the value
$0, -x(i-1), -x(i-1) -x(i-2)$, etc.\ and these values are distinct since
\[
n>n',\ \sum_{j=1}^{n} (-x(i -j)) = \sum_{j=1}^{n'} (-x(i -j))
\implies \sum_{j= n'+1}^n(-x(i-j)) = 0.
\]
The number of locally $d$-Mullen $m$-compositions with any total is then
$(|G|-1)^{\underline{d-1}}(|G|-d)^{m-d+1}$ and Theorem $\ref{thm:equal}$
gives the result.
\end{proof}

\begin{proposition}
The number of $m$-compositions of $s \in G$ over a finite group $G$
such that the sum of any $d+1$ consecutive parts is not $0$ is
\[ |G|^{d-1} (|G|-1)^{m-d} (1 + O(\theta^m)), \qquad m \to \infty,
0 \leq \theta < 1, \]
provided $d \leq |G| -2$.
\end{proposition}
\begin{proof}
% strong connectedness
Define the appropriate $D$ so that $V(D)$ contains all $(d+1)$-tuples of
vertices that do not sum to $0$.
The strong connectedness of $\graphf{D}$ is established in (the proof of)
\cite[Corollary 2]{abelian};
% offdef
we include the argument here for completeness.
Let $u, v \in V(D)$ be distinct.
Let $w$ be a vertex such that
$w = (w(1), \ldots, w(j), v(j+1), \ldots, v(d+1))$, and assume
$u \too w$.
Clearly this is possible if $j=d+1$.
If this is true for some $j \leq d+1$
we seek to show that there is a vertex
$y = (y(1), \ldots, y(j-1), v(j), \ldots, v(d+1))$ such that
$w \too y$ and thus $u \too y$.
Let $a \in G$ satisfy
\begin{align*}
  w(1) + \cdots + w(j-2) + a + w(j) + v(j+1) + \cdots + v(d+1) &\neq 0,\\
  w(1) + \cdots + w(j-2) + a + v(j) + \cdots + v(d+1) &\neq 0.
\end{align*}
Then
$w \too (w(1), \ldots, w(j-2), a, w(j), v(j+1), \ldots, v(d+1))
\too y = (w(1), \ldots, w(j-2), a, v(j), \ldots, v(d+1))$.
By induction, we conclude that $u \too w$ in the case $j=0$, i.e.\ $u \too v$.

We turn to strong connectedness of $D_\times$.
Let $\graphf{u} = (a_1, \ldots, a_{d+1})$ be an arbitrary vertex in
$V(\graphf{D})$, and let $s$ be an element of $G$.
We seek a path (or a walk) from
$(\graphf{u}, 0)$ to $(\graphf{u}, s)$ in $\graphf{D}_\times$.

Let $b \in G$ satisfy the system
\begin{align*}
a_1 + \cdots + a_d + b &\neq 0 \\
a_2 + \cdots + a_d + b + s &\neq 0.
\end{align*}
This gives at least $|G| - 2$ possible values for $b$.

Let $b' \in G$ satisfy the system
\begin{align*}
a_j + \cdots + a_d + b + s + b' + a_2 + \cdots + a_{j-2} &\neq 0,
  \qquad 3 \leq j \leq d + 1 \\
s + b' + a_2 + \cdots + a_d &\neq 0 \\
b' + a_2 + \cdots + a_{d+1} &\neq 0.
\end{align*}
This gives at least $|G| -d -1$ possible values for $b'$.

%Let $\graphf{u}$ be a vertex in $V(\graphf{D})$.
%Let $l(\graphf{u})$ be the cyclic left shift of $\graphf{u}$, and let
%$c$ be the order of $\sum \graphf{u}$.
%Then we say the \emph{reduplication} of $\graphf{u}$ is
%$c-1$ concatenated copies of the sequence
%$l(\graphf{u}), l^2(\graphf{u}), \ldots, l^{d+1}(\graphf{u})$.

Starting from $u$,
we take a $(d+1)$-step walk to $(a_1, \ldots, a_d, b)$.
Let $n_1$ be the order of $a_1 + \cdots + a_d + b$.
We cycle back to $(a_1, \ldots, a_d, b)$ exactly $n_1-1$ times.
Now we take one step by appending $s$.
Then we take a $(d+1)$-step walk to $b', a_2, \ldots, a_{d+1}$ and cycle that
vertex the appropriate number of times.
Finally walk to and cycle $a_1, \ldots, a_{d+1}$.
The total of this walk is $0 + s + 0 + 0 = s$.
We conclude that $\graphf{D}_\times$ is strongly connected.

% aperiodicity
To establish aperiodicity,
let $\graphf{u} = (a_1, \ldots, a_{d+1})$ be a $\graphf{D}$-vertex satisfying
the following.
Set $a_{d}$ so that $a_1 + \cdots + a_{d} \neq 0$.
Set $a_{d+1}$ so that for $i=1,\ldots,d$ we have $\Sigma a - a_i \neq 0$.
Thus for $i=1,\ldots,d+1$ we have $\Sigma a - a_i \neq 0$.
There are at least $|G| - d$ possible values for $a_{d+1}$.
Then we may take the same approach as in the proof of Proposition
\ref{prop:carlitz} where we consider two cycles from $\graphf{u}$, one with
an extra $0$ inserted.

% explicit count
We have $|V(\graphf{D})| = |G|^d (|G| - 1)$ and each vertex has out-degree
$|G| - 1$.
Thus there are
$|G|^d (|G| - 1) (|G| - 1)^{m-d-1} = |G|^d (|G|-1)^{m-d}$
walks in $\graphf{D}$ defining an $m$-composition.
Applying Theorem \ref{thm:equal} gives the result.
\end{proof}

Figure \ref{fig:noadj0} shows uniformly-randomly generated
100-compositions over $\mathbb{Z}_5$ such that no part may be followed by
its (additive) inverse.

\begin{figure}
\centering
\includegraphics[width=5.5in]{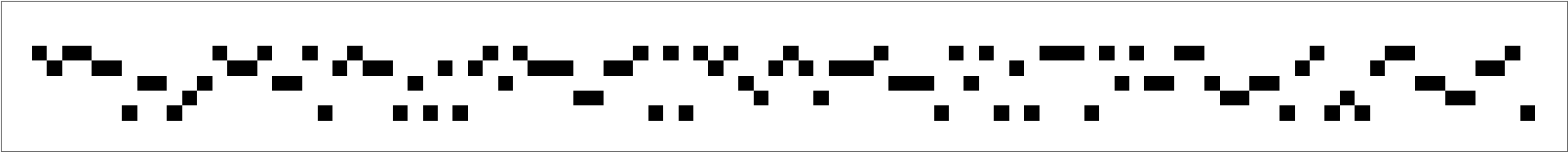}
\includegraphics[width=5.5in]{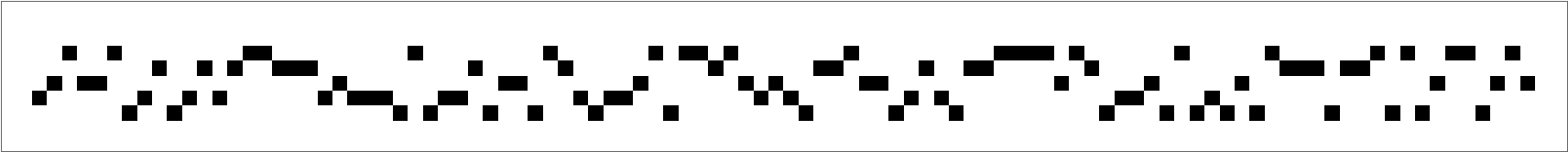}

\caption{Uniform-randomly generated $100$-compositions of $0$ (above)
  and $1$ (below) over $\mathbb{Z}_5$ with $x(i) \neq -x(i+1)$.
  \label{fig:noadj0}}
\end{figure}

\begin{proposition}
Let $p_a(m)$ be the number of $m$-compositions of $a \in G$ such that
the sum of any $d+1$ consecutive parts is not $0$.
Then for $a \neq 0, b\neq 0$, we have $p_a(m) = p_b(m)$.
If $m$ is not a multiple of $d+1$, then $p_0(m) = p_a(m)$.
\end{proposition}
\begin{proof}
Let $x = (x(1), \ldots, x(m))$ be an $m$-composition.
Let $y(i) = \sum_{n=1}^i x(n)$.
Clearly $x$ uniquely determines $y$ and vice versa.
Also, $x$ has total $a$ if and only if $y(m) = a$.

Let $y^{\langle j \rangle}(i) = y((i-1)(d+1) + j-1)$ for $j \in [d+1]$.
Then $x$ satisfies the condition if and only if each $y^{\langle j \rangle}$ is
Carlitz and $y^{\langle d+1 \rangle}(1) \neq 0$.

First assume $m$ is not a multiple of $d+1$, so $y(m)$ is the last part of
some $y^{\langle j \rangle}, j \neq d+1$.
Let $\pi: G \to G$ be defined $\pi(b) = b$ for all $b \not\in \{0, a\}$,
and $\pi(a) = 0, \pi(0) = a$.
Then if we apply $\pi$ to $y^{\langle j \rangle}$ within $y$ and take
differences, we get a new $x'$ which satisfies the condition and has total $0$.
Thus $p_0(m) = p_b(m)$.

Second, if $m$ is a multiple of $d+1$, the previous $\pi$ does not work
since it may change whether $y^{\langle d+1 \rangle}(1) \neq 0$.
However, if we take some bijective $\pi':G \to G$ which fixes $0$ and swaps two
nonzero elements $a$ and $b$, and apply it to $y^{\langle d+1 \rangle}$
in $y$ we conclude $p_a(m) = p_b(m)$.
\end{proof}

Table \ref{tab:noadj0} gives counts for $m$-compositions of some $a$ over
the quaternion group $Q_3$ such that no part may be followed by its inverse.

\begin{table}
\centering
\begin{tabular}[c]{|c|rrr|}
\hline
\diagbox{$m$}{$a$} & $id$ & $(1854)(2763)$ & (1256)(3478)\\ \hline
2  & 0             & 8             & 8             \\
3  & 49            & 49            & 49            \\
4  & 392           & 336           & 336           \\
5  & 2401          & 2401          & 2401          \\
6  & 16464         & 16856         & 16856         \\
7  & 117649        & 117649        & 117649        \\
8  & 825944        & 823200        & 823200        \\
9  & 5764801       & 5764801       & 5764801       \\
10 & 40336800      & 40356008      & 40356008      \\
11 & 282475249     & 282475249     & 282475249     \\
12 & 1977444392    & 1977309936    & 1977309936    \\
13 & 13841287201   & 13841287201   & 13841287201   \\
14 & 96888186864   & 96889128056   & 96889128056   \\
15 & 678223072849  & 678223072849  & 678223072849  \\
16 & 4747567274744 & 4747560686400 & 4747560686400 \\ \hline
\end{tabular}
\caption{Counts of $m$-compositions of $a$ with no part followed by
  its inverse, over $Q_8$ (written as a subgroup of $S_8$).}
\label{tab:noadj0}
\end{table}

\begin{example}
We examine restrictions where all parts are simply required to lie in a
fixed set $\Xi$.
We assume without loss of generality that the subset $\Xi$ generates $G$.
When working with permutation groups, the meaning of \enquote{composition} as
in \enquote{integer composition} is actually the same as in \enquote{functional
composition}.

If $\Xi = G$ then the number of compositions of $a$ is always $|G|^{m-1}$
since the first $m-1$ parts are arbitrary and the last part is uniquely
determined.
However if $\Xi \subset G$ this is no longer the case.

The digraph $\graphf{D}$ with vertex set $\Xi$ is clearly strongly connected,
and it is straightforward to see that $D_\times$ is strongly connected as
well.

For any cycle with final arc labeled $a$ in the Cayley graph constructed from
$\Xi$, there is a cycle of equal length at $((a), 0)$ in $D_\times$. This
implies that $D_\times$ is aperiodic if and only if the Cayley graph is aperiodic.

One way to ensure an aperiodic Cayley graph is to include $0 \in \Xi$.
In general Cayley graphs are not aperiodic e.g.\
only an even number of transposition permutations can equal the identity
since the identity is an even permutation.
\end{example}

% number of occurrences

We turn to the problem of counting compositions with $r > 0$ occurrences of
a pattern.
In preparation, we quote the fundamental fact of rational generating function
asymptotics which is applied a few times in the remainder.

\begin{theorem}[Theorem IV.9 in \cite{ac}] \label{thm:ratgf}
If $f(z)$ is a rational function that is analytic at $0$ and has poles
$\alpha_1, \alpha_2, \ldots, \alpha_m$, then there exist $m$ polynomials
$\Pi_j(x)$ such that for sufficiently large $n$ we have
$[z^n]f(z) = \sum_{j=1}^m \Pi_j(n) \alpha_j^{-n}$ where the degree of $\Pi_j$
is the order of the pole of $f$ at $\alpha_j$, minus one.
\end{theorem}

\begin{theorem} \label{thm:pathsr}
Let $\bar{D}$ be a de Bruijn graph.
Let $U \subset V(\bar{D})$ and $\Psi, \Phi \subseteq V(\bar{D})$ all be
nonempty and suppose $\graphf{D} = \bar{D} - U$ is regular with strongly
connected derived digraph $D_\times$.

Let $\mu$ be the minimum number of occurrences of $U$ (as subwords) in a
composition in $\mathcal{P}(\bar{D}, \Psi, \Phi)$ that has at least $1$
occurrence of $V(D)$.
Assume that for all sufficiently large values of $m$ there exist compositions
in $\mathcal{P}(m; \bar{D}, V(D), V(D))$ with exactly $1$ occurrence of $U$,
and that $p(m; D, V(D), V(D)) \sim A \cdot B^m$.

If $r \geq \max(\mu, 1), \mu \geq 0$ then the number of $m$-compositions of $a
\in G$ starting in $\Psi$ and finishing in $\Phi$ with exactly $r$ occurrences
of $U$ is
\[
p_a(m, r; D, \Psi, \Phi)
= m^{r-\mu} A_{r,\mu} \cdot B^{m}(1 + O(m^{-1})), A_{r,\mu}>0,B>1, \qquad m \to
\infty.
\]
\end{theorem}
\begin{proof}
Define an occurrence segment as a composition $w$ of length at least $\sigma$
where every part in $w$ is involved in an occurrence of $U$.
A detour in $\bar{D}$ is an occurrence segment $w=(w(1), \ldots, w(\ell))$
where there is an arc from $V(D)$ to $(w(1), \ldots, w(\sigma))$ and an
arc from $(w(\ell - \sigma + 1), \ldots, w(\ell))$ to $V(D)$.
The occurrence segment $w$ gives a left semi-detour if
there is an arc from $V({D})$ to $(w(1), \ldots, w(\sigma))$ and $w$ gives a
right semi-detour if there is an arc from $(w(\ell-\sigma + 1), \ldots,
w(\ell))$ to $V({D})$.

Fix elements of $v_{\Psi} \in \Psi$ and $v_{\Phi} \in \Phi$ as start and finish
segments.
If $v_{\Psi} \not\in U$, set $\bar{v_{\Psi}} = v_{\Psi}$, and
if $v_{\Psi} \in U$, set $\bar{v_{\Psi}}$ to some right semi-detour with
$v_{\Psi}$ at the beginning.
If $v_{\Phi} \not\in U$, set $\bar{v_{\Phi}} = v_{\Phi}$ and
if $v_{\Phi} \in U$, set $\bar{v_{\Phi}}$ to a left semi-detour with $v_{\Phi}$
at the end.
Fix a further sequence of detours $d_1, \ldots, d_n$ so that the total number
of occurrences of $U$ in all (semi-)detours is $r$.
For $m$ sufficiently large, an $m$-composition with $r$ occurrences of $U$
has the form
\[ x = \bar{v_{\Psi}} {y_1} {d_1} {y_2} d_2
  \cdots y_n {d_n} {y_{n+1}} \bar{v_{\Phi}},
\]
where each $y_i$ is a non-empty composition such that no parts of $y_i$ are
involved in an occurrence of $U$ in $x$.
We further fix $a_1, \ldots, a_{n+1} \in G$ such that
$\Sigma y_i = a_i$ implies $\Sigma x = a$.
Let the total length of the $y_i$ be $m - \delta$.

Given all of the fixed objects, the $y_i$ are subject to
start and finish constraints, totals, and a total length $m - \delta$.
Let $c^{\langle i \rangle}(z)$ be the generating function counting possible
$y_i$ where $z$ marks length.
Then
\[c^{\langle i \rangle}(z) =
  z^\sigma \psi_i^\top \left(\sum_{j \geq 0} M_\times^j z^j \right)
\phi_i + P_i(z), \]
where $M_\times$ is the adjacency matrix of $D_\times$ and $\psi_i$ and
$\phi_i$ are the appropriate start and finish vectors.
The term $P_i(z)$ is a polynomial which counts the appropriate $m$-compositions
with $m < \sigma$.
The number of sequences $y_1, \ldots, y_{n+1}$ is then
\[ [z^{m-\delta}] \prod_{i=1}^{n+1} c^{\langle i \rangle}(z). \]
By Theorem \ref{thm:equal} and Theorem \ref{thm:ratgf},
\[ \prod_{i=1}^{n+1} c^{\langle i \rangle}(z) =
A' \frac{1}{(1-Bz)^{n+1}} + O((1-Bz)^{-n}), \qquad z \to 1, \]
and
\[
[z^{m-\delta}] \prod_{i=1}^{n+1} c^{\langle i \rangle}(z) =
m^{n} A'' \cdot B^m(1 + O(m^{-1})), \qquad m \to \infty.
\]
There is a finite set of possible values for the objects we fixed and from the
assumptions we know $n$ attains the value $r - \mu$, so we conclude the result.
\end{proof}

% offdef
\begin{notation}
\label{not:dist}
For a sequence of random variables $X_0, X_1, \ldots, $ we write $X_n
\Rightarrow X_0$ to denote that the sequence converges in distribution to $X_0$.
\end{notation}

\begin{proposition}\label{prop:markovvisits}
Let $D$ be a strongly connected, aperiodic digraph with at least $2$ vertices.
Let $A \subseteq V(D) \times V(D)$ be a nonempty set of allowed start-finish
vertex pairs.
Let $\Xi \subset V(D)$ be a nonempty set of designated vertices such that
there are arbitrarily long walks in $D- \Xi$ and/or $D[\Xi]$,
e.g.\ $D- \Xi$ is strongly connected.
Let $X_m$ be the number of vertices of $\Xi$ in a uniform random walk of length
$m$ in $D$ where the initial and final vertices are found as a pair in $A$.
Then $E(X_m) \sim c_1 m$, $\Var(X_m) \sim c_2 m$ where $c_1, c_2 > 0$ do not
depend on $A$, and
\[ \frac{X_m - E(X_m)}{\sqrt{\Var(X_m)}} \Rightarrow N(0,1). \]
\end{proposition}
\begin{proof}
Say $V(D) = n$ and fix an ordering $v_1, \ldots, v_n$ on $V(D)$.
We use $u$ as an indeterminate to mark occurrences of $\Xi$.
Let $C$ be an $n \times n$ matrix where
\[[C]_{i,j} = [(v_i, v_j) \in A](u[v_i \in \Xi] + [v_i \not\in \Xi]),\]
and let $T$ be an $n \times n$ matrix where
\[ [T]_{i,j} = [(v_i, v_j) \in E(D)](u[v_j \in \Xi] + [v_j \not\in \Xi]). \]
The matrix $T$ is known as the transfer matrix.
Then $[u^r] \sum_{i,j} [C]_{i,j} [T^{m-1}]_{i,j}$ is the number of walks of
length $m$ in $D$ with $r$ occurrences of $\Xi$ with start and finish vertices
allowed by $A$.

Theorem 1 in \cite{matrix} establishes limiting distributions for secondary
parameters in the context of the transfer matrix method.
It can be applied to obtain the result if we verify that there is a
vertex $v \in V(D)$ and positive integer $k$ such that there are
walks from $v$ to $v$ of length $s$ with differing numbers of terms in $\Xi$.
Assume WLOG that there are arbitrarily long walks in $D - \Xi$ and
suppose $v \not\in \Xi$.
Let $W$ be a sufficiently long walk from $v$ to $v$ in $D - \Xi$.
There are walks from $v$ to $v$ which visit $\Xi$, and by aperiodicity of $D$
such a walk exists of the exact length of $W$, so we are finished.
% Schur's theorem, nonnegative linear combinations of terms with gcd=1
% https://en.wikipedia.org/wiki/Schur%27s_theorem#Combinatorics
\end{proof}

Clearly if $\Xi = \emptyset$ or $\Xi = V(D)$ the number of occurrences of $\Xi$
is trivial.
We note that de Bruijn graphs and their derived digraphs are always aperiodic
and strongly connected and there are always arbitrarily long walks at the
vertices with loops.
% easy to see

% offdef
We quote a helpful theorem on convergence in distribution.
\begin{theorem}[Slutsky]
\label{thm:slut}
Assume $X_n$ and $Y_n$ are random variables for $n \geq 1$.
Also assume that $X_n$ converges in distribution to a random variable $X$ and
$Y_n \Rightarrow c$ where $c \in \mathbb{R}$.
Then $X_n + Y_n \Rightarrow X + c$ and $X_n Y_n \Rightarrow X \cdot c$.
\end{theorem}
\begin{proof}
See \cite[Theorem 11.4]{gut2013probability}.
\end{proof}

% offdef
\begin{notation}
\label{not:deriv}
If $f(z,\ldots)$ is a power series, we use $D_z$ to denote the derivative
$(d/dz)f(z, \ldots)$.
\end{notation}

\begin{example}
We look at compositions over $\mathbb{Z}_2$ and keep track of occurrences of $U
= \{00, 11\}$.
The relevant de Bruijn graph $D$ has span $\sigma = 2$.
For the following we fix a particular ordering $v_i$ on the vertices of
$D_\times$.
% 00,0 11,0 01,0 10,0 00,1 11,1 01,1 10,1
We use the indeterminates $u$ and $z$ to mark length and total.
We define start vector
\[ \psi = z^2 \rvect{u, u, 0, 0, 0, 0, 1, 1}^\top, \]
where $\psi(i)$ is $0$ for non-start vertices, $z^\sigma u$ for start vertices
corresponding to $U$, and $z^\sigma$ otherwise.
The finish vector for compositions of $0$ is
\[ \phi_0 = \rvect{1,1,1,1,0,0,0,0}^\top. \]
The matrix $C$ from Proposition \ref{prop:markovvisits} is then
$\psi \phi_0^\top$.
The transfer matrix is
\[
T = \left[
\begin{array}{cccccccc}
 u & 0 & 0 & 0 & 0 & 0 & 1 & 0 \\
 0 & 0 & 0 & 1 & 0 & u & 0 & 0 \\
 0 & 0 & 0 & 1 & 0 & u & 0 & 0 \\
 u & 0 & 0 & 0 & 0 & 0 & 1 & 0 \\
 0 & 0 & 1 & 0 & u & 0 & 0 & 0 \\
 0 & u & 0 & 0 & 0 & 0 & 0 & 1 \\
 0 & u & 0 & 0 & 0 & 0 & 0 & 1 \\
 0 & 0 & 1 & 0 & u & 0 & 0 & 0 \\
\end{array}
\right],
\]
where $T_{i,j} = 0$ if $(v_i, v_j) \not\in E(D_\times)$, $T_{i,j} = u$ if
$(v_i, v_j) \in E(D_\times)$ and $v_j$ corresponds to $U$, and $T_{i,j} = 1$
otherwise.
We define
\[ P(z,u) = z^{\sigma - 1} \psi^\top (I - z T)^{-1} \phi_0, \]
getting that $[z^m u^r]P(z,u)$ is the number of $m$-compositions of $0$
over $\mathbb{Z}_2$ with $r$ occurrences of $U$.

Let $X^{\langle 0 \rangle}_m$ be the number of occurrences of $U$ in
a uniform-random $m$-composition of $0$.
We have
\[
E(X^{\langle 0 \rangle}_m) =
  \frac{[z^m]D_u P(z,u) |_{u=1}}{[z^m]P(z,1)} = \frac{1}{2}m + O(1),
\]
and
\begin{align*}
\Var(X^{\langle 0 \rangle}_m) =&
  \frac{[z^m]D_u^2 P(z,u) |_{u=1}}{[z^m]P(z,1)}
  + \frac{[z^m]D_u P(z,u) |_{u=1}}{[z^m]P(z,1)}
  - \left(  \frac{[z^m]D_u P(z,u) |_{u=1}}{[z^m]P(z,1)} \right)^2\\
  =& \frac{1}{4}m + O(1).
\end{align*}
So Proposition \ref{prop:markovvisits} (and an application of Theorem
\ref{thm:slut}) entail
\[
\frac{X_m^{\langle 0 \rangle} - \frac{1}{2}m}{
  \frac{1}{2}\sqrt{m}} \Rightarrow N(0,1). \qedhere
\]
\end{example}

Asymptotic joint distributions and local limit phenomena are derivable, under
conditions, based on \cite{matrix} and/or \cite{bertoni2003number}.
%r = alpha * m, large deviations, e.g. bender paper, pemantle book
% appears unavailable/too esoteric
One can also analyze additional parameters (longest runs, etc.)
in analogy to the existing local restriction theory.
However in these matters as in Proposition \ref{prop:markovvisits}
one expects to get results identical to those
for words (disregarding the total) since arbitrary start and finish
requirements do not affect asymptotic distributions.

\begin{lemma} \label{lem:patternsconnected}
Let the greatest letter in a subword pattern $\tau$ be $j^*$.
Assume $k \geq 2$ and $k \geq j^*$.
If $\tau$ has length $p+1 \geq 2$ and $\tau$ is not $1^p 2$ and its
symmetries ($1 2^p, 2^p 1$, and $2 1^p$), there is
a strongly connected de Bruijn subgraph $D$ with span $\sigma \geq p$ whose
walks represent words over $[k]$ that avoid $\tau$.
\end{lemma}
\begin{proof}
The patterns $1^p 2$ do not satisfy this because $1^{p}$ and $2^{p}$ are
both allowed but there is no allowed sequence of the form
${1^{p}} w 2^{p}$ where $w$ is some word.

Let $\tau = (\tau(1), \ldots, \tau(\sigma + 1))$.
Let $D$ be the de Bruijn subgraph of span $\sigma$ representing $k$-ary words
avoiding $\tau$.
Let $x = (x(1), \ldots, x(\sigma))$ and $y = (y(1), \ldots, y(\sigma))$ be vertices
of $D$.
We proceed by cases, establishing either that
$x \too y$ and $y \too x$ or $1^{\sigma} \too x$ and $x \too 1^{\sigma}$.

Case 1: $j^*=1$.
If $x(\sigma) \neq y(1)$, then the concatenation $x y$ is allowed.
Otherwise, take $c \neq x(\sigma) = y(1)$ and then $x c y$ is
allowed.

Case 2: $j^* \geq 3$.
Assume WLOG $\tau(1) > 1$.
Then ${1^{\sigma}} x$ is always allowed.
If $\tau(\sigma) > 1$ then $x 1^{\sigma}$ is allowed too.
Otherwise $\tau(\sigma) = 1$ and $x k^{\sigma} 1^{\sigma}$ is allowed.

Case 3: $j^*=2$.
Assume WLOG $\tau(1) = 2$.
Again ${1^{\sigma}} x$ is allowed.
If $\tau(\sigma) = 2$ then $x 1^{\sigma}$ is allowed too.
If $\tau(\sigma) = 1$ and $\tau$ is not monotonic then $x k^{\sigma}
1^{\sigma}$ is allowed.
Finally, if $\tau = 2^p 1^q$ with $p,q>1$ then
$x (k1)^p 1^{p-1}$ is allowed.

This shows that a satisfactory digraph exists with span $\sigma$.
It is now easily seen that a digraph with greater span would also be strongly
connected.
\end{proof}

\begin{lemma} \label{lem:subworddx}
Let $G$ be a totally ordered finite group and let $\tau$ be a subword pattern
of length at least $2$
other than $1^p 2$ and its symmetries ($1 2^p, 2^p 1$, and $2 1^p$).
If $j^*$ is the greatest letter in $\tau$, assume $|G| \geq \max(3, j^*)$.
The de Bruijn subgraph $\graphf{D}$ with span $\sigma = |\tau|$
representing compositions over $G$ avoiding $\tau$ is such that
$\graphf{D}_\times$ is strongly connected and aperiodic.
\end{lemma}
\begin{proof}
%Let $\sigma + 1$ be the length of the pattern, and take $D$ to have span
%$\sigma$.

%We show strong connectedness of $D_\times$.
%Let $a \in G$ be nonzero.
%If the pattern is not $i^p j i^q$ where $p, q \geq 1$ and $i \neq j$, then
%$0^\sigma a 0^\sigma$ is allowed and therefore
%$(0^{\sigma}, 0) \too (0^{\sigma}, a)$ in $\graphf{D}_\times$.
%For patterns in the form $i^p j i^q$, let $n$ be the order of $\sigma a$.
%Then $0^{\sigma} a^{\sigma n} a 0^{\sigma}$ shows
%$(0^{\sigma}, 0) \too (0^{\sigma}, a)$.

%We show aperiodicity of $D_\times$.
%The vertex $(0^{\sigma}, 0)$ exists in $\graphf{D}_\times$ and has a loop
%iff the pattern is not $1^{\sigma + 1}$.
%For the pattern $1^{\sigma + 1}$ where $\sigma \geq 2$, let $b\in G$ be nonzero
%and let $a=-\sigma b$.
%Then the two sequences $b^{\sigma} 0 a 0 b^{\sigma}$ and
%$b^{\sigma} 0 a 00 b^{\sigma}$
%are allowed and correspond to walks $(b^{\sigma}, 0) \too (b^{\sigma}, 0)$
%with lengths differing by $1$.
%Lastly, if $\tau = 11$, let $a,b \in G$ be distinct nonzero elements.
%The compositions $ab$ and $a0b$ represent paths from $((a), a)$ to $((b), a+b)$
%with lengths differing by $1$.

Let $a,b \in G$ be distinct and both nonzero.

We show strong connectedness of $D_\times$.
Let $c = -a$.
If $\tau = 1^\sigma$,
then the composition
$0^{\sigma -1}a0b0c0^{\sigma -1}a$
exhibits a path
$(0^{\sigma -1}a, 0) \too (0^{\sigma -1}a, b)$.
Otherwise, if $\tau \neq i^p j i^q$ where $p, q \geq 1$ and $i \neq
j$, then $0^\sigma a 0^\sigma$ is allowed and therefore
$(0^{\sigma}, 0) \too (0^{\sigma}, a)$ in $\graphf{D}_\times$.
Finally, if $\tau = i^p j i^q$, let $n$ be the order of $\sigma a$.
Then $0^{\sigma} a^{\sigma n} a 0^{\sigma}$ shows
$(0^{\sigma}, 0) \too (0^{\sigma}, a)$.

We show aperiodicity of $D_\times$.
The vertex $(0^{\sigma}, 0)$ exists in $\graphf{D}_\times$ and has a loop if and only if
$\tau \neq 1^\sigma$.
For the pattern $1^\sigma$ where $\sigma \geq 3$, the two sequences
$b^{\sigma-1} 0 b^{\sigma-1} 0$ and $b^{\sigma-1} 00 b^{\sigma-1}0$ are allowed
and correspond to walks $(b^{\sigma-1}0, 0) \too (b^{\sigma-1}0,
(\sigma-1)b)$ with lengths differing by $1$.
Lastly, if $\tau = 11$,
the compositions $abab$ and $ab0ab$ represent paths from $((a,b), 0)$ to
$((a, b), a+b)$ with lengths differing by $1$.
\end{proof}

\begin{theorem} \label{thm:pathsubword}
Let $G$ be a finite group with a total order and let $\tau$ be a subword
pattern of length at least $2$ not $1^p 2$ or its symmetries ($1 2^p, 2^p 1$,
and $2 1^p$).
If $j^*$ is the greatest letter in $\tau$, assume $|G| \geq \max(3, j^*)$.
The number of $m$-compositions of $a \in G$ containing $r$ occurrences of
$\tau$ is
\[ A_r m^r B^m (1 + O(m^{-1})), \qquad A_r>0, B>1,m \to \infty. \]
If $X_m^{\langle a \rangle}$ is the number of occurrences of $\tau$ in a
uniform random $m$-composition of $a \in G$ then
\[ \frac{X_m^{\langle a \rangle} - {E}(X_m^{\langle a \rangle})}{
  \sqrt{\Var(X_m^{\langle a \rangle})}} \Rightarrow N(0,1). \]
\end{theorem}
\begin{proof}
To satisfy the requirements of Theorem \ref{thm:pathsr} we must show that
there are arbitrarily long compositions with a single occurrence of $\tau$.
Let $a$ be the minimal element of $G$ and suppose $b > a$.
If $\tau = 1^p$, then $\cdots abab a^p baba \cdots$ is such a composition.
If $\tau \neq q^p$, let $y$ be an occurrence of $\tau$ with minimal-valued
parts; then
$a\cdots a y a \cdots a$ is such a composition.
We can now apply Lemma \ref{lem:subworddx}, Theorem \ref{thm:pathsr},
Proposition \ref{pro:2vert}, and Proposition \ref{prop:markovvisits} to
conclude the result.
\end{proof}

We are immediately able to modify results where they are available for words
containing subwords patterns.

\begin{proposition}
Let $G$ be a totally ordered group with $|G| = k$, e.g.\ $\mathbb{Z}_k$ where
$0 < 1 < \cdots < k-1$.
% offdef d to k
Define
\[
C(y) = [q^r]\frac{1}{1 - y - \sum_{p=3}^k \sum_{j=0}^{p-3}
  \binom{p-3}{j} \binom{k}{p+j} y^{p+j}(q-1)^{p-2}},
\]
as in \cite[p.\ 112]{cofc}.
Let $\rho > 0$ be the radius of convergence of $C(y)$, and let
$A_r = \lim_{y \to \rho} ((1-y/\rho)^{r+1} C(y))$.

The number of $m$-compositions of $a \in G$ containing $r$ occurrences of the
subword pattern $123$ is
\[ \frac{1}{k} A_r m^{r} \left( \frac{1}{\rho} \right)^m (1 + O(m^{-1})),
  \qquad m \to \infty. \]
\end{proposition}
\begin{proof}
Theorem 4.30 in \cite{cofc} states that $[y^m]C(y)$ is the number of
$m$-compositions with any total containing $r$ occurrences of $123$.
The result then follows from Theorem \ref{thm:pathsubword}.
\end{proof}

Table \ref{tab:132path} shows counts of $m$-compositions of $a$ avoiding
$\tau=132$ over $\mathbb{Z}_5$.
Figure \ref{fig:132path} gives randomly selected compositions avoiding $132$
over $\mathbb{Z}_5$.
Figure \ref{fig:121path} gives the same for compositions avoiding $121$.

\begin{table}
\centering
\begin{tabular}[c]{|c|rrr|}
\hline
\diagbox{$m$}{$a$} & $0$ & $1$ & $2$\\ \hline
2  &  5 & 5 & 5 \\
3  &23 & 23 & 23 \\
4  &105 & 105 & 105 \\
5  &478 & 477 & 477 \\
6  &2171 & 2171 & 2171 \\
7  &9869 & 9868 & 9868 \\
8  &44861 & 44861 & 44861 \\
9  &203930 & 203930 & 203930 \\
10 &927032 & 927032 & 927033 \\
11 &4214147 & 4214147 & 4214147 \\
12 &19156861 & 19156861 & 19156865 \\
13 &87084158 & 87084158 & 87084158 \\
14 &395871195 & 395871195 & 395871198 \\
15 &1799569607 & 1799569609 & 1799569610 \\
16 &8180566793 & 8180566793 & 8180566793 \\\hline
\end{tabular}
\caption{Counts of $m$-compositions of some $a$ avoiding $132$ over
  $\mathbb{Z}_5$.}
\label{tab:132path}
\end{table}

\begin{figure}
\centering
\includegraphics[width=5.5in]{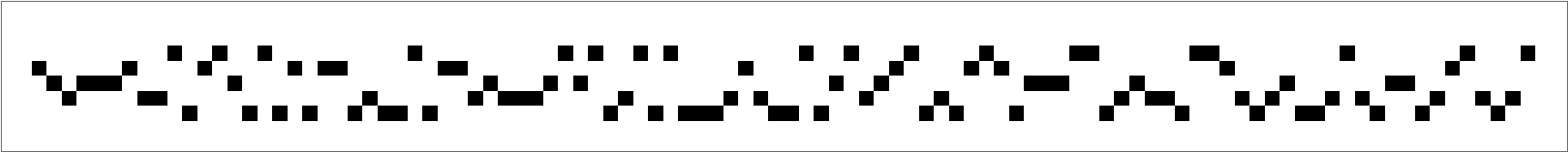}
\includegraphics[width=5.5in]{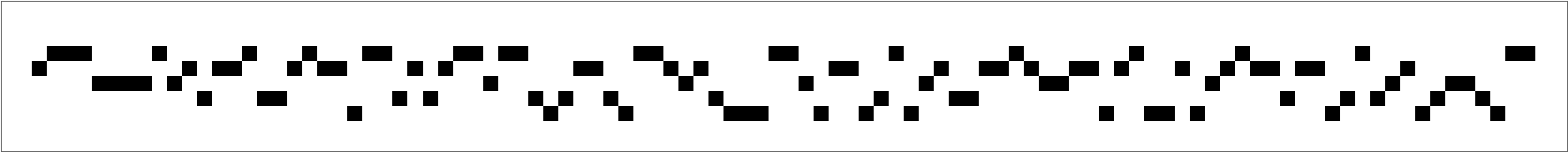}

\caption{Uniform-randomly generated $100$-compositions of $0$ (above)
  and $1$ (below) over $\mathbb{Z}_5$ which avoid $132$.
  \label{fig:132path}}
\end{figure}

\begin{figure}
\centering
\includegraphics[width=5.5in]{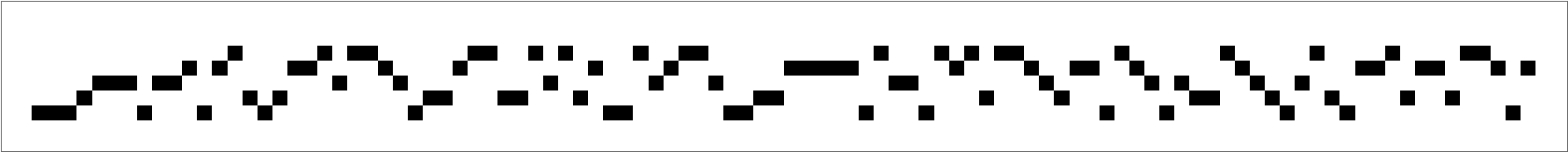}
\includegraphics[width=5.5in]{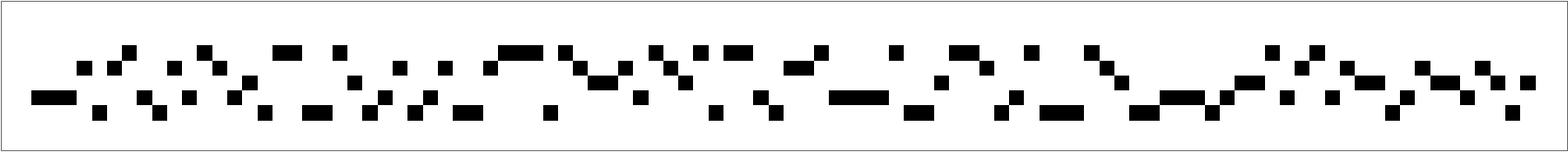}

\caption{Uniform-randomly generated $100$-compositions of $0$ (above)
  and $1$ (below) over $\mathbb{Z}_5$ which avoid $121$.
  \label{fig:121path}}
\end{figure}

\subsection{Note on minimization of transfer matrices} \label{sec:transfer}

Given a transfer matrix (adjacency matrix) $T$, we may compute counting
sequences by raising $T$ to a power, and if $T$ has a simple enough structure,
we may even be able to extract a closed form expression for the counting
sequence, or learn other information.
A simple transfer matrix corresponds to a simple digraph $D$.
Locally restricted compositions over a finite set constitute a regular
language, so equivalently we may say we are interested in simple finite
automata.
In those terms, a question arises: Given a finite automaton $A$, when is it
possible to find a smaller automaton $A'$ such that $A$ and $A'$ are equivalent
for counting purposes?

\begin{definition}
As in \cite{ravikumar2004weak}, two deterministic finite automata
(DFAs) $A, A'$ are \emph{weakly equivalent} if for each integer $m \geq 0$,
the automata $A$ and $A'$ accept the same number of words of length $m$.
\end{definition}

Our question is largely answered by an algorithm given in
\cite[\S ~4.2]{ravikumar2004weak}, which we refer to as the
Ravikumar-Eisman algorithm.
The algorithm is given a DFA $A$ and returns a weakly equivalent DFA $A'$
with the same number or fewer states.
While the Ravikumar-Eisman algorithm is not guaranteed to find the smallest
such $A'$, it is shown to be practically useful and no better technique is
currently available.
Roughly speaking, the Ravikumar-Eisman algorithm works by finding states which
are equivalent in the weak sense (there are equal numbers of accepted words of
length $m$ starting at each state for each $m$);
these states are then merged.

\begin{example}
Figure~\ref{fig:carlitz3wordsnaive} shows a naive automaton for Carlitz/Smirnov
words on the alphabet $\{a,b,c\}$.
In fact this automaton is minimal in the usual sense of number of states.
However, there is a length-preserving bijection between $3$-ary Carlitz words
and the language accepted by the automaton in
Figure~\ref{fig:carlitz3wordsminimal}, which is returned by the
Ravikumar-Eisman algorithm.

\begin{figure}
\centering  \includegraphics[width=32em]{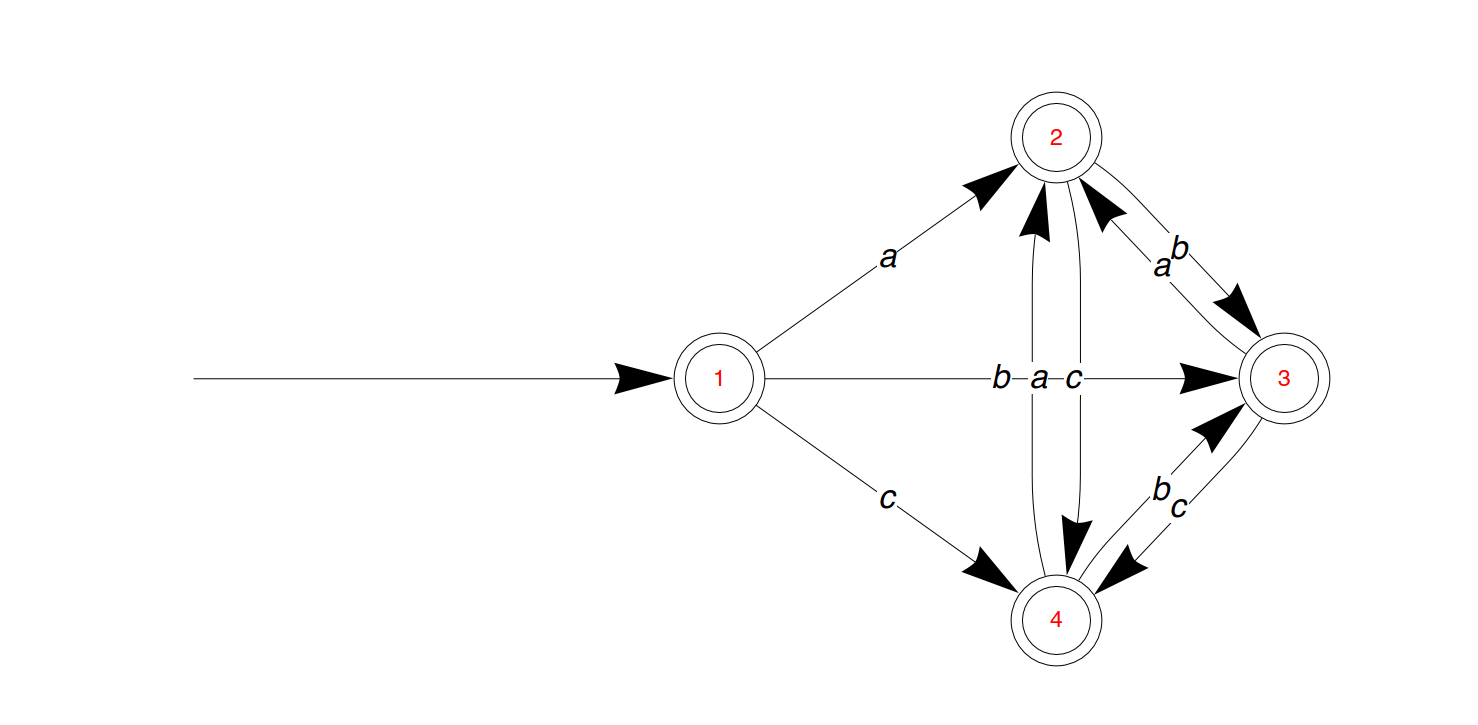}
\caption{A DFA that accepts $3$-ary Carlitz words over $\{a,b,c\}$.
\label{fig:carlitz3wordsnaive}}
\end{figure}

% a "two-state solution" hawhaw
\begin{figure}
\centering  \includegraphics[width=30em]{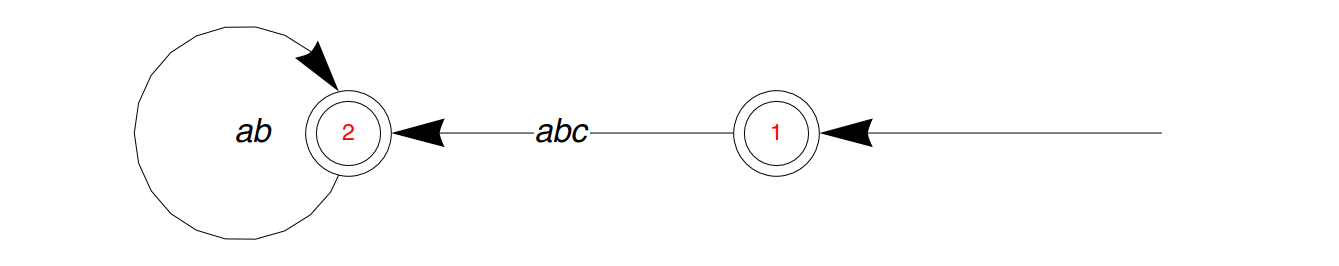}
\caption{A DFA weakly equivalent to one that accepts $3$-ary Carlitz words.
\label{fig:carlitz3wordsminimal}}
\end{figure}

For general $k$-ary Carlitz words, we still need only a $2$-state DFA
$A'$ rather than the naive $k+1$ states.
Suppose we number the start state of $A'$ as $1$ and the other state as $2$.
Let $f_i(m), i=1,2$ be the number of $m$-words accepted by $A'$ if state $i$
were the start state.
Either by converting to a regular grammar or using the transfer matrix
method we get
$f_2(m) = (k-1)^m, f_1(m) = k f_2(m-1)$ which allows us to conclude (the
obvious) $f_1(m) = k(k-1)^{m-1}$.
\end{example}

If we generalize $k$-ary Carlitz words to avoidance of the subword pattern
$1^p$, a naive automaton $A$ with $k^p + O(1)$ states has weakly equivalent
automaton $A'$ with $p + O(1)$ states including, for each
$1 \leq i < p$, a
state representing words ending with a run of length $i$.
Similar phenomena are seen for other subword patterns, with the general
theme that simpler patterns have simpler automata.

%Multivariate version

We can also consider the following refinement of weak equivalence for
multivariate counting.

\begin{definition}
Take two DFAs $A$ and $A'$ that recognize a language over a $k$-ary alphabet.
Then $A$ and $A'$ are
\emph{completely weakly equivalent} if for all
$j_1, \ldots, j_k$, the DFAs $A$ and $A'$ accept an equal number of words
with $j_i$ letters $i$, for $1 \leq i \leq k$.
\end{definition}

There is a brute-force algorithm for exact minimization of DFAs according to
complete weak equivalence:
Given a DFA $A$, enumerate all smaller DFAs $A'$ in ascending order by size.
Extract the multivariate rational generating functions for $A$ and $A'$ where
there is an indeterminate symbol marking each letter, and compare by
subtracting and testing for $0$.

A straightforward generalization of the faster Ravikumar-Eisman algorithm to
the multivariate problem depends on generalizing
Lemmas A1 and A2 in \cite{stearns1985equivalence} from sequences of real
numbers to sequences of real polynomials.
We give this generalization.

\begin{lemma}
Let $\Xi$ be a finite nonempty index set.
For all $\xi \in \Xi$ let
$A_\xi: \mathbb{Z}_{\geq 0} \to \mathbb{R}[x_1, \ldots, x_k]$ satisfy
\[ A_\xi(n+1) = \sum_{t \in \Xi} c_{\xi, t} A_t(n), \qquad n
  \geq 0 \]
where $c_{\xi,t} \in \mathbb{R}[x_1, \ldots, x_k]$.
Then each $A_\xi$ satisfies a linear difference (recurrence) equation of degree
$|\Xi|$ or less with coefficients in $\mathbb{R}[x_1, \ldots, x_k]$.
\end{lemma}
\begin{proof}
The proof of Lemma A1 in \cite{stearns1985equivalence} requires no
modification to prove this result, except that linear algebra is done over
$\mathbb{R}(x_1, \ldots, x_k)$ rather than $\mathbb{R}$.
\end{proof}

\begin{lemma}
Let $A, B:\mathbb{Z}_{\geq 0} \to \mathbb{R}[x_1, \ldots, x_k]$
be sequences satisfying linear difference
equations of degrees $a$ and $b$ with coefficients in
$\mathbb{R}[x_1, \ldots, x_k]$.
If for $0 \leq n \leq a + b$ we have $A(n) = B(n)$ then
the sequences $A(n)$ and $B(n)$ are identical for all $n$.
\end{lemma}
\begin{proof}
As above for \cite[Lemma A2]{stearns1985equivalence}.
\end{proof}

\subsection{Note on weighted trees} \label{sec:trees}
% Subtitle: "time to leave"

% the positive integers gives an infinite algebraic system.
% far too complicated, extrapolating from infinite transfer matrices.

% literature

The number of \emph{unweighted} binary plane trees avoiding certain local
structures is found in \cite{rowland2010pattern}.
In that paper, \S 5 gives an algorithm to compute a system of algebraic
equations specifying the relevant generating function.
An extension to ternary and $m$-ary trees is in \cite[\S 3]{gabriel2012pattern}.
Enumeration of unweighted trees by number of local occurrences (not just
avoidance) is done in \cite{chyzak2008distribution}.
The paper \cite{dairyko2012non} considers global pattern avoidance,
still in unweighted trees.

% discussion

Locally restricted trees weighted by a finite group are those that avoid
subgraphs from a fixed set $\Xi$ of weighted trees, where the set $\Xi$ has
a maximum size.
We specifically consider rooted trees where there is a directed edge from
parent to child.
Plane trees correspond, for example, to the family of trees where the
vertex set is $[m]$, parents are less than children, and all vertices at a
given depth form a contiguous interval of integers.
For non-abelian groups, there must be a stipulated
ordering on tree vertices such as depth-first search,
in order to define the total of a tree.
We require that the order be recursive, in the sense that the total of the
tree must be a sum made from the weight of the root and the sums of subtrees
rooted at children of the root.
Variations on trees with similar enumerative properties include
functional graphs, directed acyclic graphs, and cactus graphs.

\begin{example}
Let us consider a tree $T$ with weights from $\mathbb{Z}_2$ where the first $2$
levels are of the form

\begin{center}
\begin{forest}
[0
  [0
  ]
    [0
    ]
]
\end{forest}.
\end{center}

If $T$ avoids parent-child-grandchild paths with the same weight on all $3$
vertices, sibling subtrees are independent but no level-$3$ vertex of $T$ may
have the weight $0$.
\end{example}

An alternative visual representation of weighted trees for groups with
some total order uses a color behind vertices where darker means greater.
This plotting technique works well for larger trees because of higher
visibility, although it is less precise.
An example is shown in Figure \ref{fig:bwtree}.

\begin{figure}
\centering
%(2, ((3), (2, (2, ((3), (2))))))
\begin{forest}
[2
  [3
  ]
  [2 [
    2
      [3] [2]
  ]
  ]
]
\end{forest} \hspace{3cm}
\includegraphics[width=9em]{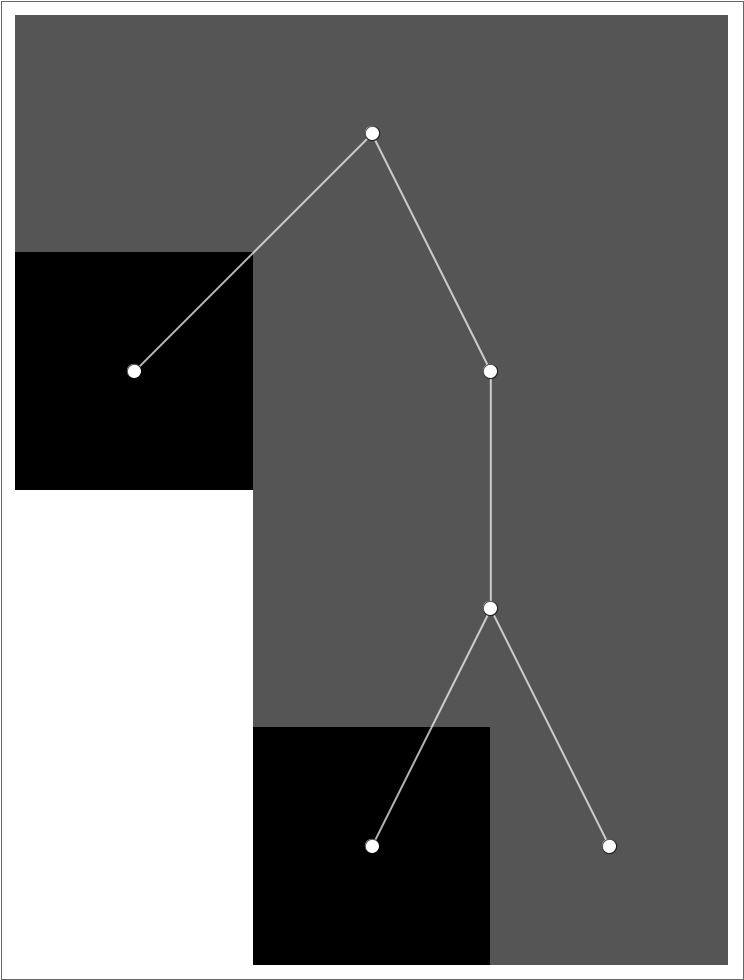}
\caption{A tree weighted by $\mathbb{Z}_4$ shown in two plotting styles.
\label{fig:bwtree}}
\end{figure}

It is straightforward (albeit a little cumbersome) to show that
the generating function $T_a(z)$, counting locally restricted trees with total
$a \in G$, where $z$ marks number of vertices, is expressible in terms of a
system of algebraic equations.
This is, of course, common in tree enumeration as in examples in
\cite{rowland2010pattern}, \cite[\S I.5]{ac}, \cite{drmotabook}, et cetera.
The theory of coefficients of algebraic functions in, e.g.\ \cite[\S
VII.6.1]{ac} may be applied, under conditions, to derive the usual
$[z^m]T_a(z) \sim A_a m^{-1/2} B^m, A_a>0, B>1$.
The works \cite{chyzak2008distribution,drmotabook} conduct analysis of
the number of pattern occurrences in uniform
random unweighted trees and
show convergence in distribution to the standard normal after normalization;
we expect that their method applies similarly in the present context.

Finally, we note that computing expansions of multivariate algebraic series
is possible using software packages available for computer algebra systems.
Newton iteration is a relatively efficient option.
The package Genfunlib \cite{genfunlib} for Mathematica implements Newton
iteration as the command \texttt{CoefsByNewton}, but only for single equations,
so there is a preliminary step of eliminating all but one component of the
original system.
An example expansion follows of the solution to
$f(z, u) = u + z(f(z,u)^2 + f(z,u))$.
\begin{verbatim}
In[1]:= CoefsByNewton[
  f[z, u] == u + z (f[z, u]^2 + f[z, u]),
  f[z, u], {z, 0, 5}]

Out[1]= u + (u+u^2)z + (u+3 u^2+2 u^3)z^2
  + (u+6 u^2+10 u^3+5 u^4)z^3 +(u+10 u^2+30 u^3+35 u^4+14 u^5)z^4
  + (u+15 u^2+70 u^3+140 u^4+126 u^5+42 u^6)z^5 + O[z]^6
\end{verbatim}

\section{Locally cyclically restricted compositions}

\begin{remark}
In the language of \S \ref{sec:intro},
cyclically restricted compositions are
weighted directed
cycles where we track the total weight.
Weighted undirected cycles may be counted in a similar manner.
Sets of cycles, i.e.\ $2$-regular graphs, are also closely related.
% pattern in undirected <=> set of patterns in directed, but
% you have to multiply occurrences in directed by number of ways for that
% pattern to be undirected occurrence of original pattern.
% for iso matching it's a bit simpler
\end{remark}

\subsection{Compositions over a finite group} \label{sec:cycliccomps}
% code

In \S \ref{sec:localpaths} we represented compositions by walks on any
de Bruijn subgraph $D$ over $\SEQ_\sigma(G)$.
Within the current section and \S \ref{sec:circcomps} we slightly specialize
the possibilities for $D$ as follows.
If $\sigma \geq 1$ is the span, let $\bar{D}$ be the $\sigma$-dimensional
de Bruijn graph on $G$, let $U \subset V(\bar{D})$, and let $D = \bar{D} - U$.
Note that given such a digraph $D$, the set $U$ is uniquely determined.

An $m$-composition $x$ is \emph{locally cyclically restricted} according to $D$
if
\[(x(1), \ldots, x(m), x(1), \ldots, x(\sigma -1))\]
avoids $U$ as subwords.
A number of observations about this definition should be made.
First, $m$-compositions where $m < \sigma$ do not correspond to walks over
$D$ but may or may not be cyclically restricted according to $D$.
Second, for $m < \sigma$, we are technically departing from the isomorphic
nature of pattern occurrences in the language of \S \ref{sec:intro}, and
really this corresponds to homomorphic pattern occurrences.
We do not remark on this point further.

Let $\mathcal{C}_a(m; D)$ be the set of all $m$-compositions of $a$ that are
cyclically restricted according to $D$, and define
\[
c_a(m; D) = |\mathcal{C}_a(m; D)|,\qquad
C_a(z; D) = \sum_{m \geq 0}c_a(m; D)z^m.
\]

\begin{lemma} \label{lem:cyclespaths}
For $v \in V(D)$, let $\Sigma' v = v(1) + \cdots + v(\sigma -1)$.
For $m \geq \sigma$ we have
\[
c_a(m; D) = \sum_{\substack{v \in V(D) \\ u \in N^-(v)}}
  p_{a + \Sigma' v}(m + \sigma - 1; D, \{v\}, \{u\}),
\]
for $a \in G$.
\end{lemma}
\begin{proof}
If $m \geq \sigma$, consider a walk $w_1, \ldots, w_{m}$ in $D_\times$,
where
$w_1 = (v, \sum v)$ and $(w_{m}, (v, a + \sum v)) \in E(D_\times)$.
Let $(x(1), \ldots, x(m+\sigma-1))$ be the composition represented by the walk.
Then $(x(1), \ldots, x(m))$ is precisely an $m$-composition of $a$ which is
{cyclically restricted} according to $D$.
\end{proof}

\begin{proposition}
Fix an ordering on $V(D_\times)$ and let $M_\times$ be the adjacency matrix
of $D_\times$.
For $(v, a) \in V(D_\times)$, let
$\xi_{v,a} \in \mathbb{R}^{|V(D_\times)|}$
be the indicator vector for vertex $(v, a)$.
Then for $m \geq \sigma$,
\[
c_a(m; D) =
  \sum_{\substack{v \in V(D) \\ u \in N^-(v)}}
  (\xi_{v, \Sigma v})^\top M_\times^{m-1} \xi_{u, a + \Sigma' v}.
\]
\end{proposition}
\begin{proof}
This follows from Lemma \ref{lem:cyclespaths} and Proposition
\ref{prop:dtimescount}.
\end{proof}

\begin{proposition} \label{pro:cyclicasympt}
Assume $D$ is regular.
We have either $c_a(m; D) = 0$ or
\[ c_a(m; D) = A_a \cdot B^m(1 + O(\theta^m)), \qquad m \to \infty \]
for $A_a, B >0, 0 \leq \theta < 1$.
\end{proposition}
\begin{proof}
This follows from Lemma \ref{lem:cyclespaths} and Proposition \ref{prop:count}.
\end{proof}

\begin{proposition} \label{pro:cyclesequal}
Assume $D_\times$ is strongly connected and aperiodic.
Then
$c_a(m; D) = A \cdot B^m(1 + O(\theta^m))$
where $A$ does not depend on $a \in G$.
\end{proposition}
\begin{proof}
This follows from Lemma \ref{lem:cyclespaths} and Theorem \ref{thm:equal}.
\end{proof}

Let $x = (x(1), \ldots, x(m))$ and $y = (y(1), \ldots, y(\sigma))$ be
compositions.
A local \emph{cyclic occurrence} of $y$ in $x$ is an occurrence of $y$ as a
subword in $(x(1), \ldots, x(m), x(1), \ldots, x(\sigma-1))$.

\begin{theorem} \label{thm:cyclicr}
Assume $U$ is nonempty and suppose $\graphf{D} = \bar{D} - U$ is regular with
strongly connected derived digraph $D_\times$.

For $u \in V(\bar{D})$,
let $\mu(u)$ be the minimum number of occurrences of
$U$ in a composition in $\mathcal{P}(\bar{D}, \{u\}, N^{-}(u))$ with at least
$1$ occurrence of $V(D)$.
Let $\mu$ be the minimal such $\mu(u)$.
Assume for all sufficiently large values of $m$ there exist compositions
in $\mathcal{P}(m; \bar{D}, V(D), V(D))$ with exactly $1$ occurrence of $U$,
and that $p(m; D, V(D), V(D)) \sim A \cdot B^m$.

If $r \geq \max(\mu, 1), \mu \geq 0$ then the number of $m$-compositions of $a
\in G$ with exactly $r$ cyclic occurrences of $U$ is
\[
c_a(m, r; D)
= m^{r-\mu} A_{r,\mu} \cdot B^{m}(1 + O(m^{-1})), \qquad m \to \infty.
\]
\end{theorem}
\begin{proof}
The result follows from Lemma \ref{lem:cyclespaths} and Theorem
\ref{thm:pathsr}.
\end{proof}

\begin{theorem} \label{thm:cyclicdist}
Assume that $|G| \geq 2$ and that $U \subset \SEQ_\sigma(G)$ is non-empty.
The number of cyclic occurrences of $U$ in a uniform random
$m$-composition of $a \in G$ is asymptotically normal with mean and variance
asymptotic to those of the number of occurrences of $U$ in a uniform random
word over $G$.
\end{theorem}
\begin{proof}
This follows directly from Proposition \ref{prop:markovvisits}.
If $D_\times$ is the derived digraph of the de Bruijn graph on
$\SEQ_\sigma(G)$, the allowed start-finish pairs are all
$(u, v) \in V(D_\times)^2$ such that $(v,u) \in E(D_\times)$.
\end{proof}

We consider some examples of cyclic restrictions.

\begin{proposition}
A composition $x = (x(1), \ldots, x(m))$ is a cyclic Carlitz composition if
$(x(1), \ldots, x(m), x(1))$ is a Carlitz composition.
The number of cyclic Carlitz $m$-compositions of $a \in G$ over a finite
group $G$ is
\[
  \frac{(|G|-1)^m}{{|G|}} (1 + O(\theta^m)), \qquad m \to \infty,
  0 \leq \theta < 1,
\]
provided $|G| \geq 3$.
\end{proposition}
\begin{proof}
First let us consider cyclic Carlitz $m$-words over $[k]$.
Assume the first letter is $k$.
A cyclic Carlitz (or Smirnov) word is then a sequence of
pairs of a single letter $k$ followed by a non-empty Carlitz word on $[k-1]$.
Let $\bar{H}_{k}(z) = kz/(1-(k-1)z)$ be the ordinary generating function for
non-empty Carlitz words on $[k]$.
Thus if $F_k(z)$ is the ordinary generating function for cyclic Carlitz
words on $[k]$, we have
% offdef
\begin{align*}
  F_k(z) =& k\frac{z \bar{H}_{k-1}(z)}{1 -z \bar{H}_{k-1}(z)} \\
  =& k\frac{(k-1) z^2}{(z+1) (1 - (k-1) z)}
\end{align*}
and
\[ [z^m]F_k(z) = {(k-1)^m+k (-1)^m+(-1)^{m+1}}, \qquad m > 1. \]
The above derivation is a special case of Theorem 4 in
\cite{hadjicostas2017cyclic}.
It remains to recall from Proposition \ref{prop:carlitz} that there is a
digraph $D$ representing Carlitz compositions such that $D_\times$ is aperiodic
and strongly connected; Proposition \ref{pro:cyclesequal} applies.
\end{proof}

Let $\Xi = \{\xi_1, \xi_2, \ldots \}$ be an ordered set.
A sequence $w = (w(1), \ldots, w(m))$ over $\Xi$ is $p$-smooth if
for all $i= 1, \ldots, m-1$ if we have $w(i) = \xi_j, w(i+1) = \xi_k$ then
$|k-j| \leq p$.
Additionally, $w$ is \emph{$p$-smooth cyclic} if $(w(1), \ldots w(m), w(1))$
is $p$-smooth.
% mention that one could study more general adjacent difference sets.
Research Direction~6.5 in \cite[p.~239]{cofc} asks for an explicit formula
for the number of $p$-smooth cyclic $k$-ary words of length $m$.

We apply Proposition \ref{pro:cyclesequal} in the case $p=1$.

\begin{proposition}
Let $\mathbb{Z}_k$ have ordering $0, 1, \ldots, k-1$.
Let
\[ C(z) = 1 + \frac{kz(1 + 3z)}{(1 + z)(1 - 3z)} -
  \frac{2(k+1) z}{(1+z)(1-3z)}
  \frac{U_{k-1}(\frac{1-z}{2z})}{U_k(\frac{1-z}{2z})} \]
be the ordinary generating function for $k$-ary $1$-smooth cyclic words
as in \cite[Exercise 6.10]{cofc} and
\cite{knopfmacher2010staircase, smooth2}
where $U_k$ is the $k$\textsuperscript{th}
Chebyshev polynomial of the second kind.

Let $\rho > 0$ be the radius of convergence of $C(z)$, and let
$A = \lim_{z \to \rho} (1-z/\rho) C(z)$.
Then the number of $1$-smooth cyclic $m$-compositions of $i \in \mathbb{Z}_k$
is asymptotic to
\[
\frac{1}{k} A \cdot \left( \frac{1}{\rho} \right)^m (1 + O(\theta^m)),
\qquad m \to \infty, 0 \leq \theta < 1.
\]
\end{proposition}
\begin{proof}
Let $\bar{D}$ be the de Bruijn graph on $\mathbb{Z}_k^2$,
let $U \subseteq \mathbb{Z}_k^2$ be all $(a,b)$ which are not smooth,
and let $D = \bar{D} - U$.
In the derived graph $D_\times$, for any $i \in \mathbb{Z}$, we exhibit a walk
from $((0,0), i)$ to $((0,0), 0)$ by taking the following sequence of elements
of $\mathbb{Z}_k$.
First, take $0, 1, 2, \ldots, -i-1, -i, -i-1, \ldots, 2, 1, 0, 0$.
Let $c = 1 + 2 + \cdots + -i-1$.
The sum of elements on this sequence is $2c -i$.
Let $n$ be the order of $2c$ in $\mathbb{Z}_k$.
Repeat the following $n-1$ times:
$0, 1, 2, \ldots, -i-2, -i-1, -i-1, -i-2, \ldots, 0, 0$.
The grand total of these sequences concatenated is $2c -i + (n-1)c = -i$ and
thus there is a walk in $D_\times$ starting at $((0, 0),i)$ and ending at
$((0, 0),0)$.
The digraph $D_\times$ is clearly aperiodic since there is a loop at the
vertex $((0, 0),0)$.
Thus we may apply Proposition \ref{pro:cyclesequal}.
\end{proof}

Table \ref{tab:132cyclic} shows counts of $m$-compositions of $a$ cyclically
avoiding $\tau=132$ over $\mathbb{Z}_5$.
Figure \ref{fig:132cyclic} contains uniform randomly generated compositions
over $\mathbb{Z}_5$ that cyclically avoid $132$.

\begin{table}
\centering
\begin{tabular}[c]{|c|rr|}
\hline
\diagbox{$m$}{$a$} & $0$ & $1$ \\ \hline
3  & 19         & 19 \\
4  & 85         & 85 \\
5  & 390        & 385 \\
6  & 1763       & 1763 \\
7  & 8023       & 8016 \\
8  & 36469      & 36469 \\
9  & 165790     & 165790 \\
10 & 753660     & 753660 \\
11 & 3426039    & 3426039 \\
12 & 15574231   & 15574231 \\
13 & 70798118   & 70798118 \\
14 & 321837325  & 321837325 \\
15 & 1463023035 & 1463023045 \\
16 & 6650677797 & 6650677797 \\\hline
\end{tabular}
\caption{Counts of $m$-compositions of $a$ cyclically avoiding $132$ over
  $\mathbb{Z}_5$.}
\label{tab:132cyclic}
\end{table}

\begin{figure}
\centering
\includegraphics[width=5.5in]{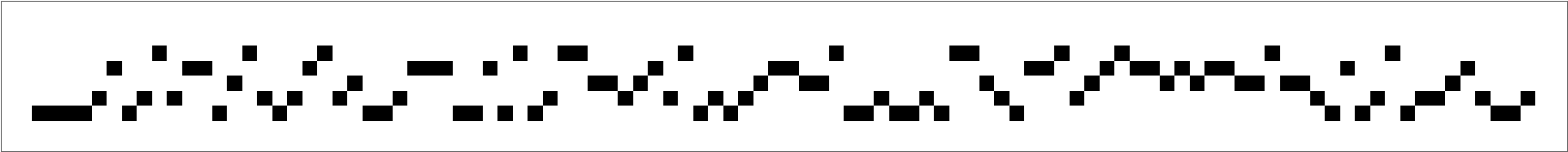}
\includegraphics[width=5.5in]{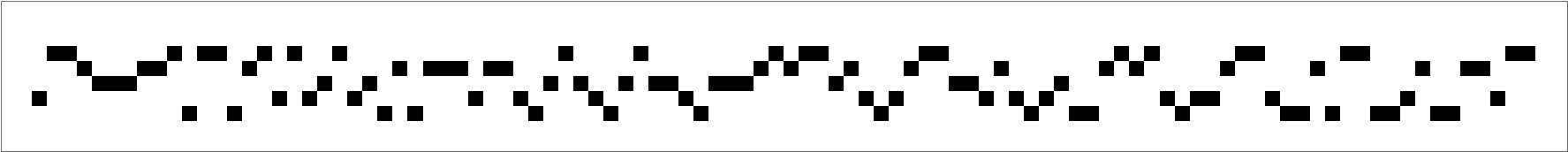}

\caption{Uniform-randomly generated $100$-compositions of $0$ (above)
  and $1$ (below) over $\mathbb{Z}_5$ which cyclically avoid $132$.
  \label{fig:132cyclic}}
\end{figure}

\begin{remark}
\emph{Wheel graphs} are
a variation on cycles with similar enumerative properties.
A wheel graph consists of a cycle $C$ with a vertex $v$ added
and (directed) edges from $v$ to each vertex in $C$.
%center vertex can now be the first vertex in a pattern occurrence.
%count cycles where in any $\pi_2 \dots \pi_l$ pattern the range of
%possible values for center vertex to create occurrence is at most $j$.
\end{remark}

\subsection{Note on integer compositions} \label{sec:intcomps}

Let $x = (x(1), \ldots, x(m))$ be an integer composition, i.e.\
$x(i) \in \mathbb{Z}_{> 0}, 1 \leq i \leq m$.
To define locally restricted integer compositions with span
$\sigma \in \mathbb{Z}_{> 0}$ we use a \emph{local
restriction function} $R: \mathbb{Z}_{> 0}^{\sigma} \to \{0,1\}$ which
encodes the $\sigma$-tuples that are allowed as a subword inside an integer
composition.
If $\SEQ(\mathbb{Z}_{>0})$ is the set of all integer compositions of any length,
define $\bar{R}: \SEQ(\mathbb{Z}_{>0}) \to \{0,1\}$ so that
$\bar{R}(x) = 1$ if and only if
$R(x(i), x(i+1), \ldots, x(i+\sigma-1)) = 1$ for all $1 \leq i \leq m-\sigma +
1$, in which case $x$ is allowed according to $R$.
As an expedient it is also helpful to define an infinite digraph $\mathcal{D}$
with vertex set $V(\mathcal{D}) = R^{-1}(1)$ and where $(u,v) \in
E(\mathcal{D})$ if and only if $\bar{R}(u v) = 1$.
Note that a walk in $\mathcal{D}$ represents a composition obtained by
\emph{concatenating} the vertices; as such, walks do not represent all
restricted compositions, only those whose length is a multiple of $\sigma$.
In this way the infinite digraph $\mathcal{D}$ is interpreted differently
from the de Bruijn graphs used in other sections.
We assume there is some vertex ordering $V(\mathcal{D}) = \{v_1, v_2, \ldots
\}$.
We define the \emph{transfer operator} $T(z)$ formally as the infinite matrix
where
% offdef V to E
$[T(z)]_{i,j} = [(v_i, v_j) \in E(\mathcal{D})]z^{\Sigma v_i + \Sigma v_j}$.

A research direction suggested in \cite[\S 4]{hadjicostas2017cyclic} is
developing a framework for locally cyclically restricted integer
compositions.
The framework for locally restricted integer compositions in \cite{infinite}
can be used with little modification.
In this section we follow the definitions of \cite{infinite} with some
simplifications.

We say that $x$ is a composition which is \emph{cyclically} restricted by $R$
if
\[\bar{R}(x(1), \ldots, x(m), x(1), \ldots, x(\sigma)) = 1.\]
The \emph{endpoint operator} $E(z,y)$ is a formal infinite matrix defined by
\[
[E(z,y)]_{i,j} = \sum_{k \geq 1} y^{2\sigma} [T(z)]_{k,i}
  \sum_x z^{\Sigma v_j + 2\Sigma x + \Sigma v_k} y^{2\sigma + 2|x|},
\]
where the second sum ranges over compositions $x$ with length in
$\{0, \ldots, \sigma -1\}$ such that $\bar{R}({v_j} x v_k)=1$.
The endpoint operator plays the role of the start and finish vectors of
\cite{infinite}.

\newif\ifasymptotics
\asymptoticsfalse

\ifasymptotics
In this section, $\sum_{i,j \geq 1} f_{i,j}(z) =
\lim_{N \to \infty} \sum_{i = 1}^N \sum_{j=1}^N f_{i,j}(z)$.
The analogous definition is used for triple sums.
% technically lim has different meaning for C -> C versus Q[[z]]
(By \cite[Theorem 4.48]{series} we have the basic fact that $\sum_{i \geq 1}
\sum_{j \geq 1} |s_{i,j}| =
  \sum_{i, j \geq 1} |s_{i,j}|$
for $s_{i,j} \in \mathbb{C}$, provided at least one side converges.)
\fi

\begin{proposition}
Let $S(z,y) = \sum_{j \geq 0} (y^{2\sigma})^j T(z)^j$.
Define
\[ C(z^2,y^2) = \sum_{i,j\geq 1} [S(z,y)]_{i,j} [E(z,y)]_{i,j}. \]
Then for $m \geq 3\sigma$, the coefficient $[z^n y^m]C(z,y)$ is the number of
integer $m$-compositions of $n$ that are cyclically restricted according to
$R$.
\end{proposition}
\begin{proof}
We have
\begin{align*}
&\sum_{i,j\geq 1} [S(z,y)]_{i,j} [E(z,y)]_{i,j} \\
&=\sum_{k,i,j\geq 1} y^\sigma z^{\Sigma v_k} y^{2\sigma} [T(z)]_{k,i}
  [S(z,y)]_{i,j} y^\sigma z^{\Sigma v_j} \sum_x z^{2\Sigma x} y^{2|x|}.
\end{align*}
Consider a term
\[
y^\sigma z^{\Sigma v_k} y^{2\sigma} [T(z)]_{k,i}
  [S(z,y)]_{i,j} y^\sigma z^{\Sigma v_j} \sum_x z^{2\Sigma x} y^{2|x|}.
\]
This is the generating function for restricted compositions of the form
\[{v_k} {v_i} w {v_j} x v_k,\]
where $w$ is a concatenation of vertices in $V(\mathcal{D})$,
with $x^2$ marking total sum and $y^2$ marking length, and such that the
final $v_k$ does not count.
Summing over all $i,j,k$ enumerates cyclically restricted $m$-compositions
where $m \geq 3\sigma$.
\end{proof}

\ifasymptotics
We now wish to extract asymptotic information from $C(z,y)$.
The technicalities that arise come almost entirely from the transfer operator
rather than the endpoint operator, and this analysis is available in
\cite{infinite} which uses advanced tools from functional analysis that
generalize finite dimensional matrix theory.
Here we simply provide the required minor adaptations of \cite{infinite}.
The use of $C(z^2,y^2)$ rather than $C(z,y)$ is a technical requirement
for some manipulations of the infinite matrices.

% def of recurrent vertex is same.
% the definition of T is the same since they do it only for recurrent verts

\begin{definition} \label{def:regular}
We say a local restriction function $R$ is \emph{regular} if the following
hold.
\begin{enumerate}
\item The infinite digraph $\mathcal{D}$ contains at least two vertices,
  is strongly connected, and is aperiodic.
\item There is $p \in \mathbb{Z}_{>0}$ and (possibly equal) vertices $u, v$ such that
  \[ \gcd\{m-n: m,n \in \Xi\} = 1, \]
  where
  \[ \Xi = \left\{
    \Sigma x_1 + \cdots + \Sigma x_{p-1} :
    u, x_1, \ldots, x_{p-1}, v \text{ is a walk in } \mathcal{D}
  \right\}. \]
\end{enumerate}
\end{definition}

% \S 3...

\begin{definition}
Let $\Omega \subseteq \mathbb{C}$ be a domain.
Then $\mathcal{M}(\Omega)$ is the set of infinite matrices $M(z)$ such that
each entry $[M(z)]_{i,j}$ is holomorphic in $\Omega$ and such that for every
compact $K \subseteq \Omega$ there exists $c >0$ with
\[ \sum_{i,j \geq 1} |[M(z)]_{i,j}|^2 \leq c, \qquad \forall z \in K. \]
\end{definition}

\begin{proposition} \label{prop:step1}
Let $R$ be a regular local restriction function, and let $T(z)$ and $E(z, y)$
be the associated transfer and endpoint operators.
Let $S(z,1) = \sum_{j \geq 0} T(z)^j$.
Then $S(z, 1)$ has radius of convergence $r$ satisfying $0<r<1$.
There exists a domain $\Omega \supseteq \{ z: |z| \leq r, z \neq \pm r\}$
such that
$S(z,1), E(z,1) \in \mathcal{M}(\Omega)$.
\textbf{*False!* see emails w/ R. Canfield}
% http://www-users.math.umn.edu/~garrett/m/real/notes_2016-17/08a-ops_on_Hsp.pdf
% Proposition 6.3
% https://en.wikipedia.org/wiki/Spectral_radius#Bounded_Linear_Operators
\end{proposition}
\begin{proof}
We observe that $\sum_{i,j} [E(z,1)]_{i,j} = \sum_n n^{O(1)} z^n$;
the remainder of the proof follows directly from (the proofs of) Theorems 1 and
2 in \cite{infinite}.
Specifically, $\Omega$ can be taken to be a subset of $\{z : \spr(T(z)) < 1\}$,
where $\spr(T(z))$ is the spectral radius of $T(z)$.
\end{proof}
\begin{proposition} \label{prop:step2}
Let $\Omega \subseteq \mathbb{C}$ be a domain.
If $S(z), E(z) \in \mathcal{M}(\Omega)$, then
$\sum_{i,j \geq 1} [S(z)]_{i,j} [E(z)]_{i,j}$ is holomorphic on $\Omega$.
\end{proposition}
\begin{proof}
This follows from \cite[Proposition 4 (b)]{infinite} if we suitably interpret
$S(z)$ and $E(z)$ as vectors and $\sum_{i,j \geq 1} [S(z)]_{i,j} [E(z)]_{i,j}$
as their dot product.
The proof is omitted in \cite{infinite} so we include a direct proof here for
good measure.

We use $c(z)$ to denote the sum
$\sum_{i,j \geq 1} [S(z)]_{i,j} [E(z)]_{i,j}$.
Let $K \subset \Omega$ be compact, and let
$c_1, c_2 >0$ satisfy
\[ \sum_{i,j \geq 1} |[S(z)]_{i,j}|^2 \leq c_1,\,\,\,
  \sum_{i,j \geq 1} |[E(z)]_{i,j}|^2 \leq c_2 \]
for all $z \in K$.
Then by the Cauchy-Schwarz inequality,
\begin{equation*} \label{eq:compact}
\left( \sum_{i,j \geq 1} |[S(z)]_{i,j} [E(z)]_{i,j}|\right)^2 \leq c_1 c_2,
  \qquad \forall z \in K.
\end{equation*}
This implies that the partial sums
$\sum_{i=1}^N \sum_{j=1}^N [S(z)]_{i,j} [E(z)]_{i,j}$ of
$c(z)$ are uniformly bounded on compact $K$ since
\begin{align*}
\left| \sum_{i=1}^N \sum_{j=1}^N [S(z)]_{i,j} [E(z)]_{i,j} \right| &\leq
  \sum_{i=1}^N \sum_{j=1}^N |[S(z)]_{i,j} [E(z)]_{i,j}| \\
  &\leq \sum_{i,j \geq 1} |[S(z)]_{i,j} [E(z)]_{i,j}| \\
  &\leq \sqrt{c_1 c_2}.
\end{align*}
In particular, we have pointwise absolute convergence of $c(z)$ on $\Omega$.
% some details elided

By \cite[Theorem 10.28]{rudin2} it suffices to show that $c(z)$ converges
uniformly on compact subsets of $\Omega$.
If $\mathcal{F}$ is the set of all partial sums of $c(z)$, then
\cite[Theorem 14.6]{rudin2} implies that since $\mathcal{F}$ is uniformly
bounded on compact subsets of $\Omega$, every sequence in $\mathcal{F}$ has
a subsequence that converges uniformly on compact subsets of $\Omega$.
Let $s^{\langle n \rangle}$ be a sequence in $\mathcal{F}$ that converges
pointwise to $c(z)$.
For a given compact $K$, every subsequence of $s^{\langle n \rangle}_{|K}$ that
converges uniformly
must converge to the same limit, namely $c_{|K}(z)$.
Therefore $s^{\langle n \rangle}_{|K}$ itself converges uniformly.
\end{proof}

\begin{proposition} \label{prop:step3}
Assume $R$ is a regular local restriction function with transfer operator
$T(z)$, endpoint operator $E(z,y)$, and generating function $C(z,y)$.
Let $S(z,1) = \sum_j T(z)^j$ be an infinite matrix with radius of convergence
$0 < r < 1$.
If $C(z^2, 1) = \sum_{i,j} [S(z,1)]_{i,j} [E(z,1)]_{i,j}$ is holomorphic for
$|z| \leq r$ with possible exceptions $z = \pm r$, then
$C(z, 1)$ has radius of convergence $\sqrt{r}$ and is holomorphic for $|z| =
\sqrt{r}$ except for a simple pole at $z = \sqrt{r}$.
\end{proposition}
\begin{proof}
See Theorems 1--3 and their proofs in \cite{infinite}.
\end{proof}

\begin{proposition} \label{prop:intcomp}
If $R$ is a regular local restriction function with generating function
$C(z,y)$ enumerating locally cyclically restricted integer compositions, then
\[ [z^n]C(z,1) = A \cdot B^n (1 + O(\theta^n)), \qquad n \to \infty, A > 0, B >
  1, 0 \leq \theta < 1. \]
\end{proposition}
\begin{proof}
This is direct from Propositions \ref{prop:step1}--\ref{prop:step3} and an
application of meromorphic generating function asymptotics.
Since $C(z, 1)$ has a single pole on $|z| =r$, there is $\epsilon >0$
such that $C(z, 1)$ is meromorphic on $|z| \leq r + \epsilon$ and we may
employ \cite[Theorem IV.10]{ac}.
% subtract negative powers in expansion around r. result is analytic on |z|<=r.
% therefore the result has greater radius of convergence since RoC is location
% of closest singularity
\end{proof}
\fi

We consider some examples.

\ifasymptotics
\begin{lemma} \label{lem:compsubwordgraph}
There is a regular local restriction function associated with integer
compositions avoiding a non-trivial (length $\geq 2$) subword pattern $\tau$,
as long as $\tau$ is not $1 2^{\sigma-1}$ or its symmetries ($1 2^p, 2^p 1$,
and $2 1^p$).
\end{lemma}
\begin{proof}
We verify the conditions in Definition \ref{def:regular} for the corresponding
local restriction function $R$ of span $\sigma = |\tau|$.

We consider condition 1.
Aperiodicity is particularly easy:
unless $\tau = 1^\sigma$, there is a vertex $1^\sigma$ with a loop.
If $\tau = 1^\sigma$, the vertex $1^{\lfloor \sigma/2 \rfloor}
2^{\lceil \sigma/2 \rceil}$ has a loop.
Strong connectedness follows from
Lemmas \ref{lem:patternsconnected} and \ref{lem:subworddx}.
% using aperiodicity to generate walk with length which is multiple of sigma

For condition 2, we look at the following cases.
For $\tau = 1^{\sigma}$, take $v_1 = v_2 = 1^{\sigma-1}2$.
Two walks of equal length with sum
differing by $1$ are given by
$1^{\sigma-1}2, 3^{\sigma-1}4, 1^{\sigma-1}2$ and
$1^{\sigma-1}2, 3^{\sigma-1}5, 1^{\sigma-1}2$.
For $\tau = 1^p 2^q$ with $p,q \geq 2$, we take the walks
$1^\sigma, 23 \cdots \sigma(\sigma + 1), 1^\sigma$ and
$1^\sigma, 23 \cdots \sigma (\sigma + 2), 1^\sigma$.
For other $\tau$, two suitable walks are given by
$1^\sigma, 1^{\sigma - 1}2, 2^\sigma, 1^\sigma$ and
$1^\sigma, 1^{\sigma -2} 22, 2^\sigma, 1^\sigma$.
\end{proof}
\fi

\ifasymptotics
\begin{example}
By Corollary 5 in \cite{hadjicostas2017cyclic}, the generating function
for cyclic Carlitz compositions with $z$ marking the total is
\[
C(z) =
\frac{\sum _{n=1}^{\infty } \frac{z^n}{\left(z^n+1\right)^2}}{1-
  \sum _{n=1}^{\infty } \frac{z^n}{z^n+1}}
  +\sum _{n=1}^{\infty } \frac{z^{2 n}}{z^n+1}.
\]
Proposition \ref{prop:intcomp} implies the following.
Let $\rho > 0$ be the radius of convergence of $C(z)$, and let
$A = \lim_{z \to \rho} (1-z/\rho)C(z)$.
We have $[z^n]C(z) = A \cdot \left( \frac{1}{\rho} \right)^n (1 + O(\theta^n))$.
\end{example}
\fi

Research Direction~4.4 in \cite{cofc} begins as follows.
\enquote{We say that a sequence (composition, word, partition) $s_1\cdots s_m$
cyclically avoids a subword $\tau = \tau_1 \cdots \tau_\ell$ if $s_1 \cdots s_m
s_1 \cdots s_{\ell - 1}$ avoids $\tau$.
For example, the composition $33412$ avoids the subword $123$, but does not
cyclically avoid $123$ (since $3341233$ contains $123$).}
The problem is to find the generating function for the number of compositions
of $n$ that cyclically avoid a subword pattern of length $k$.
\ifasymptotics
Lemma \ref{lem:compsubwordgraph} implies that we usually get a regular
local restriction function for subword pattern avoidance.
\fi
%Presumably we cannot get a GF that is both explicit and incorporates a large
%set of patterns at once.
We consider the patterns $122$ and $321$.

\begin{example}
Compositions cyclically avoiding $122$ over $[k]$ take the following form.
Either there is no part $k$, the composition only contains $k$, or
there is at least one $k$ and and least one other part.
In this third case, the subwords between any parts $k$ are nonempty
$122$-avoiding integer compositions over $[k-1]$ and so is the composition
obtained by concatenating the subword after the final $k$ and the subword
before the first $k$.

Let $C_{k}(z,u)$ be the generating function for nonempty cyclic $122$-avoiding
compositions where $z$ marks total and $u$ marks length, and let $P_k(z,u)$
be the generating function for nonempty $122$-avoiding compositions.
The above reasoning yields
\[ C_k(z,u) = C_{k-1}(z,u) + \frac{uz^k}{1-uz^k} +
  uz^k \frac{1}{1-P_{k-1}(z,u)u z^k} (u D_u + 1)P_{k-1}(z,u), \]
for $k \geq 2$.

The generating function $P_k(z,u)$ is given in \cite[Theorem 4.35]{cofc} as
\[
P_k(z,u) = \left(
1- \sum_{j=1}^k z^j u \prod_{i=j+1}^{k}(1-z^{2i}u^2)
\right)^{-1} -1.
\]
Let $C(z) = \lim_{k \to \infty}C_k(z,1)$.
\ifasymptotics
Since $\tau = 122$ does not correspond to a strongly connected digraph
$\mathcal{D}$, Proposition \ref{prop:intcomp} does not apply to $C(z)$.
\fi
The coefficients $[z^n]C(z)$ for $n=1, \ldots, 10$ are
$1, 2, 4, 8, 13, 28, 52, 101, 196, 383$.
\end{example}

\begin{example}
For $\tau = 321$, we consider two counting sequences.
Let $\caret 21$ be the pattern $21$ except that it only counts if it appears
at the beginning of a composition.
We count compositions over $[k]$ that avoid both $321$ and $\caret 21$.
Such a composition either has no parts $k$ or has at least one $k$.
In the latter case,
say the composition can be written $\sigma_1 k \sigma'$, where
$\sigma_1$ is a composition on $[k-1]$ and avoids $\{321, \caret 21\}$, and
$\sigma'$ is a composition on $[k]$ avoiding $\{321, \caret 21\}$.
If the composition $\sigma_1$ is empty then either $\sigma'$ is empty or $k
\sigma' = kk \sigma''$ where $\sigma''$ is a composition on $[k]$ avoiding
$\{321, \caret 21\}$.
This method proceeds similarly to the proof of Lemma 4.29 in \cite{cofc}.
Let $\bar{P}_k(z,u)$ be the generating function for compositions
avoiding $\{321, \caret 21\}$ where $z$ marks total and $u$ marks length.
This gives
\[
\bar{P}_k(z,u) = \bar{P}_{k-1}(z,u) +
  (\bar{P}_{k-1}(z,u) -1) u z^k \bar{P}_{k}(z,u) + u z^k
  + u^2 z^{2k} \bar{P}_{k}(z,u),
\]
for $k \geq 2$.
Now we go back to compositions cyclically avoiding just $321$.
Case 1: The composition has no part $k$.
Case 2: There are at least $2$ parts $k$.
Such a composition can be written $\sigma_1 k \sigma' k \sigma_2$, where
$\sigma'$ avoids $\{321, \caret 21\}$ and $\sigma_2 \sigma_1$ is a composition
over $[k-1]$ avoiding $\{321,\caret 21\}$
Case 3: There is $1$ part $k$.
Then the composition is $\sigma_1 k \sigma_2$ where $\sigma_2 \sigma_1$ is a
composition over $[k-1]$ avoiding $\{321,\caret 21\}$.
If $C_k(z,u)$ is the generating function for compositions cyclically avoiding
$321$, we have
\begin{align*}
C_k(z,u) =& C_{k-1}(z,u)
  + u z^k \bar{P}_k(z,u) u z^k (u D_u + 1)\bar{P}_{k-1}(z, u) \\
  &\qquad + u z^k (u D_u + 1)\bar{P}_{k-1}(z, u),
\end{align*}
for $k \geq 2$.
\ifasymptotics
If we let $C(z) = \lim_{k \to \infty}C_k(z,1)$ then
again, Proposition \ref{prop:intcomp} and Lemma \ref{lem:compsubwordgraph}
imply $[z^n]C(z) \sim A \cdot B^n$ with $A, B>0$ determined by $C(z)$.
\fi
\end{example}

\ifasymptotics
\else
The paper \cite{infinite} obtains asymptotics for locally restricted integer
compositions using advanced tools from functional analysis that generalize
finite dimensional matrix theory.
We expect that analogous results hold for locally cyclically restricted integer
compositions.
\fi

%\paragraph*{occurrence distributions}
%
%define unrelated local events.
%note that local events can be infinite sets.
%
%The proof of Theorem \ref{thm:thm4} given in \cite{infinite} needs no
%modification to prove the same result for cyclically restricted compositions.
%The authors further state,
%\enquote{With further restrictions on the $Y_i(n)$, it would be possible to
%extend this central limit theorem to a local limit theorem, but we have not
%worked out the details}.

The method of random generation given in Remark \ref{rem:randpaths} achieves
an exact uniform distribution but for compositions over an infinite set
such as $\mathbb{Z}_{>0}$ its performance becomes poor.
Instead we employ a Markov chain Monte Carlo (MCMC) method inspired by the
article \cite{madras2010random} which concerns pattern-avoiding permutations.

The method is as follows.
Let $\tau$ be a permutation pattern, i.e.\ where no letters are repeated,
and assume the length of $\tau$ is at least $3$.
Let $n, m> 0$ be fixed, where $n$ represents a total and $m$ represents
a length.
(The length $m$ can itself be randomly chosen first using exact counting.)
Assume $X_0$ is an $m$-composition of $n$ with at most $2$ distinct
part sizes.
Given $X_h$, $h \geq 0$, we generate $X_{h+1}$ as follows.
Let $j,k$ be independently selected uniformly at random from $[m]$.
If $X_h(j) = 1$ or $j=k$, then $X_{h+1} = X_h$.
Otherwise, let $Y$ be the following composition.
We have $Y(j) = X_h(j) -1$, $Y(k) = X_h(k) + 1$, and
$Y(i) = X_h(i)$ for $i \neq j,k$.
If $Y$ avoids $\tau$, then $X_{h+1} = Y$, otherwise $X_{h+1} = X_h$.
A composition avoiding $123$ generated by this procedure is shown in
Figure \ref{fig:123intm100n300}.

\begin{proposition}
The limiting distribution of the Markov chain $X_h$ is uniform over
$m$-compositions of $n$ that avoid $\tau$.
\end{proposition}
\begin{proof}
By the theory of Markov chains
\cite[Ex.\ 8.20]{billingsley2008probability},
it suffices to show that $X_h$ is aperiodic, irreducible, and has symmetric
transition probabilities.
Let $p(x,y)$ be the transition probability from a composition $x$ to $y$.
Aperiodicity is clear since $p(x, x) > 0$.
For symmetry, if $x \neq y$ are compositions with $p(x,y) > 0$,
then $p(x,y) = p(y,x) = 1/m^2$.
With symmetry established, irreducibility requires that for any $x$ there is
a sequence of transitions with nonzero probability that lead from $x$ to, say,
$X_0$.
We construct such a sequence.
Repeat the following until there are at most $2$ distinct part
sizes, at which point reaching $X(0)$ is clearly possible.
Let $y$ be the current composition and let $K = |\tau|$.
Let $j$ be the index of the maximum part in $y$; if this is not unique,
take the least such index if $(K, K-1)$ is a subsequence of $\tau$ and take
the greatest such index if $(K-1, K)$ is a subsequence of $\tau$.
Let $k$ be the index of the minimum part in $y$; if this is not unique,
take the least such index if $(1,2)$ is a subsequence of $\tau$, and take
the greatest such index if $(2,1)$ is a subsequence of $\tau$.
Decrement $y(j)$ and increment $y(k)$.
\end{proof}

\begin{figure}
\centering
\includegraphics[width=5.5in]{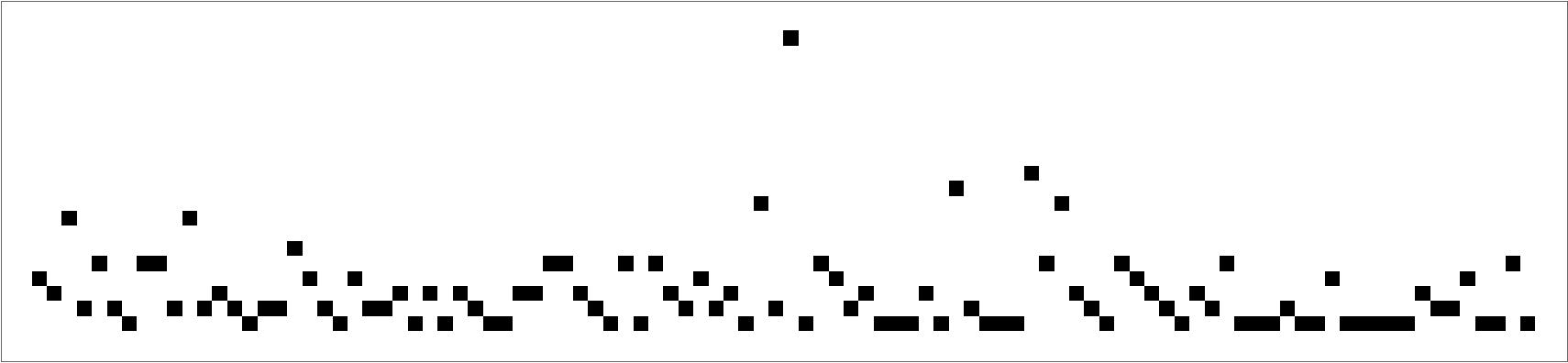}

\caption{
Integer composition of 300 avoiding the subword pattern $123$ generated by
10,000,000 iterations of an MCMC method.
  \label{fig:123intm100n300}}
\end{figure}

\section{Locally restricted compositions with symmetry}

Here we consider locally restricted compositions with symmetry, which
corresponds to local patterns in unlabeled weighted digraphs, in the language
of \S \ref{sec:intro}.
In this section, groups $G$ are assumed to be \emph{abelian}, since the order
of the parts in a composition is no longer well defined.

%\emph{aperiodic} generalizes to \emph{core}
Although we do not directly invoke it here, general counting with symmetry
typically involves Burnside's lemma.
\begin{lemma}[Burnside]
\label{lem:burnside}
The number of orbits of a permutation group \( S \) on a set \( X \) is
\[ | X / S | = \frac{1}{|S|} \sum_{s \in S} \fix(s), \]
where \( \fix(s) \) is the number of fixed points of \( s \).
\end{lemma}
Further background may be found in \cite[\S 6]{hp}.

\subsection{Circular compositions} \label{sec:circcomps}
% code

As in \S \ref{sec:cycliccomps}, here $G$ is a finite group and $\bar{D}$
is a $\sigma$-dimensional de Bruijn graph over $G$.
We speak of digraphs $D = \bar{D} - U$ for some
$U \subset V(\bar{D})$.

\begin{lemma} \label{lem:uu}
Assume $x$ is a composition and $x = u \cdots u = u^d$.
If $u$ is a subword containing $r$ cyclic occurrences of $U$, then $x$ contains
$dr$ cyclic occurrences of $U$.
\end{lemma}
\begin{proof}
A cyclic occurrence of $U$ in a composition is fully determined by the starting
index.
All cyclic occurrences of $U$ in $x$ must correspond to an occurrence in
some $u$, and vice versa.
\end{proof}

The \emph{circular shift} of the finite sequence $(x(1), \ldots, x(m))$ is
\[ (x(j), x(j+1), \ldots, x(m), x(1), x(2), \ldots, x(j)), \]
for some $1 \leq j \leq m$.
A circular composition is an equivalence class of cyclically restricted
compositions where the equivalence is under circular shift.
For example, there are two possible circular Carlitz $3$-compositions over
$\mathbb{Z}_3$, each with the same total:
\[ \{012, 201, 120\}, \{021, 210, 102\}. \]

Let $\tilde{\mathcal{C}}_a(m; D)$ be the set of all circular $m$-compositions
of $a$ that are cyclically restricted according to $D$, and define
\[
\tilde{c}_a(m; D) = |\tilde{\mathcal{C}}_a(m; D)|,\qquad
\tilde{C}_a(z; D) = \sum_{m \geq 0}\tilde{c}_a(m; D)z^m.
\]

Let $P = \mathbb{Z}_{> 0} \times G$ be the poset where
$(j, a) \preceq (k, b)$ if and only if $j | k$ and $(k/j)a = b$.
The Moebius function $\mu_P$ of $P$ is defined recursively by
$\mu_P(s, s) = 1$ for $s \in P$ and $\mu_P(s, u) = -\sum_{s \preceq t \prec u}
\mu_P(s,t)$ for $s \prec u$ in $P$.
A finite sequence is aperiodic if it is not equal to any of its circular
shifts.

\begin{proposition} \label{pro:circlecount}
We have
\[
\tilde{c}_a(m; D) = \sum_{(d,b) \preceq (m,a)} \frac{1}{d}
  \sum_{(d', b') \preceq (d, b)} c_{b'}(d'; D) \mu_P((d',b'), (d,b)).
\]
\end{proposition}
\begin{proof}
Let $\acyc(m, a)$ be the number of aperiodic cyclically restricted
$m$-compositions of $a$.
For any $m$-composition $x$ of $a \in G$, we have
$x = u \cdots u = u^{m/d}$ for some aperiodic $u$ and some $d$
which divides $m$, by \cite[Theorem 2.3.4]{shallit2008second}.
Thus by Lemma \ref{lem:uu},
\[
c_a(m; D) = \sum_{(d, b) \preceq (m, a)} \acyc(d, b).
\]
By the Moebius inversion formula \cite[Proposition 3.7.1]{stanley1},
% "principle order ideal" of p in P is {x in P: x <= p}
\[
\acyc(m, a) = \sum_{(d, b) \preceq (m, a)} c_b(d; D) \mu_P((d,b), (m,a)).
\]
Now, a circular composition consists of all possible shifts of some composition
$x = u^{m/d}$ where $u$ is aperiodic, by \cite[Theorem
2.4.2]{shallit2008second}, so
% the theorem characterizes circular shifts implying how many there are
\[
\tilde{c}_a(m; D) = \sum_{(d,b) \preceq (m,a)} \frac{1}{d} \acyc(d, b),
\]
which gives the result.
\end{proof}

\begin{theorem} \label{thm:circasympt}
Assume $D$ is regular and $c_a(m; D) \sim A_a \cdot B^m$ for $a \in G$.
We have
\[ \tilde{c}_a(m; D) = \frac{1}{m} A_a \cdot B^m (1 + O(\omega^m)), \qquad
  m \to \infty, 0 \leq \omega < 1. \]
All but an exponentially small proportion of
$\tilde{\mathcal{C}}_a(m; D)$ and
$\mathcal{C}_a(m; D)$ are aperiodic.
If $D$ satisfies the assumptions of Theorem \ref{thm:equal}, then $A_a$ does
not depend on $a$.
\end{theorem}
\begin{proof}
From above we know
\[
\tilde{c}_a(m; D) = \sum_{(d,b) \preceq (m,a)} \frac{1}{d} \acyc(d, b),
\]
where $\acyc(m, a)$ is the number of aperiodic cyclically restricted
$m$-compositions of $a$.
We claim that
$\tilde{c}_a(m; D) \sim \frac{1}{m} \acyc(m, a) \sim \frac{1}{m} c_a(m; D)$.

From Proposition \ref{pro:cyclicasympt} we have
$c_a(m; D) = A_a \cdot B^m(1 + O(\theta^m))$, where $B > 1$.
Now
\begin{align*}
  \sum_{(d,b) \prec (m,a)} \frac{1}{d} \acyc(d,b) &\leq
  \sum_{(d,b) \prec (m,a)} \frac{1}{d} c_b(d; D) \\
  &\leq |G| \frac{m}{2} A_a \cdot B^{m/2}
    \left(1 + O( \max(B^{-m/2}, \theta^{m/2}) ) \right).
\end{align*}
% uses bound c_a(m) <= A_a B^m + C D^m, 0 <= D < B for *all* m,

On the other hand, we have
\begin{align*}
c_a(m; D)  \geq \acyc(m, a) &=  c_a(m; D) - \sum_{(d,b) \prec (m, a)} \acyc(d,
  b) \\
&\geq c_a(m; D) - \sum_{(d,b) \prec (m, a)} c_b(d; D),
\end{align*}
and so
\[
- \sum_{(d,b) \prec (m, a)} c_b(d; D) \leq \acyc(m, a) - c_a(m; D) \leq 0.
\]
Thus
$\acyc(m, a) = A_a \cdot B^m(1 + O(\omega^m))$ where
$\omega = \max(\theta, B^{-1/2})$.
\end{proof}

\begin{theorem}
Assume $U$ is nonempty and suppose $\graphf{D} = \bar{D} - U$ is regular with
strongly connected derived digraph $D_\times$.

For $u \in V(\bar{D})$,
let $\mu(u)$ be the minimum number of occurrences of
$U$ in a composition in $\mathcal{P}(\bar{D}, \{u\}, N^{-}(u))$ with at least
$1$ occurrence of $V(D)$.
Let $\mu$ be the minimal such $\mu(u)$.
Assume for all sufficiently large values of $m$ there exist compositions
in $\mathcal{P}(m; \bar{D}, V(D), V(D))$ with exactly $1$ occurrence of $U$,
and that $p(m; D, V(D), V(D)) \sim A \cdot B^m$.

If $r \geq \max(\mu, 1), \mu \geq 0$ then the number of circular
$m$-compositions of $a \in G$ with exactly $r$ cyclic occurrences of $U$ is
\[
\tilde{c}_a(m, r; D)
= m^{r-\mu-1} A_{r,\mu} \cdot B^{m}(1 + O(m^{-1})), \qquad m \to \infty.
\]
\end{theorem}
\begin{proof}
Let $Q = \mathbb{Z}_{> 0} \times G \times \mathbb{Z}_{>0}$ be a poset where
$(j_1, a, j_2) \preceq (k_1, b, k_2)$ if $j_1 | k_1$, $(k_1/j_1)a = b$,
and $(k_1/j_1)j_2 = k_2$.
The Moebius function $\mu_Q$ of $Q$ is defined recursively by
$\mu_Q(s, s) = 1$ for $s \in Q$ and $\mu_Q(s, u) = -\sum_{s \preceq t \prec u}
\mu_Q(s,t)$ for $s \prec u$ in $Q$.
By analogy to Proposition \ref{pro:circlecount} we have
\begin{align*}
&\tilde{c}_a(m, r) \\
=& \sum_{(d_1,b,d_2) \preceq (m,a,r)}
\frac{1}{d}
  \sum_{(d_1', b',d_2') \preceq (d_1, b, d_2)}
  c_{b'}(d_1', d_2'; D) \mu_Q((d_1',b',d_2'), (d_1,b,d_2)).
\end{align*}

Following the proof of Theorem \ref{thm:circasympt}, the dominant term is
$m^{-1} c_{a}(m, r; D)$, so we conclude with reference to
Theorem \ref{thm:cyclicr}.
\end{proof}

\begin{definition}
A \emph{mixture} of two random variables $X, Y$ with weights $0 \leq p, 1-p
\leq 1$ is a random variable $Z$ such that the distribution functions satisfy
$F_Z(x) = p F_X(x) + (1-p)F_Y(x), x \in \mathbb{R}$.
\end{definition}

The following lemma is an expedient used to show when normalized convergence in
distribution holds up to low-probability events.
\begin{lemma} \label{lem:mixture}
Let $X_n, Y_n \geq 0$ be $L^2$ random variables for
$n \in \mathbb{Z}_{>0}$.
Let $Z_n$ be a mixture of $X_n$ and $Y_n$ with weights $p_n$ and $1-p_n$, where
$p_n \to 1$.
Assume that $E(X_n)$ or $E(Z_n)$ are bounded away from $0$, and that
$\Var(Z_n)$ or $\Var(X_n)$ are bounded away from $0$, and that
\[(1-p_n) \left(E(Y_n^2) + E(Y_n)E(X_n) + E(X_n^2) \right) = o(1).\]
Then we have
$(X_n - E(X_n))/\sqrt{\Var(X_n)} \Rightarrow F$
if and only if
$(Z_n - E(Z_n))/\sqrt{\Var(Z_n)} \Rightarrow F$,
and
$E(X_n) \sim E(Z_n), \Var(X_n) \sim \Var(Z_n)$.
\end{lemma}
\begin{proof}
We have $E(Z_n) = p_n E(X_n) + (1-p_n)E(Y_n)$, and in general
\begin{align*}
E(Z_n^2) &= 2 \int_0^\infty x P(Z_n > x)dx \\
  &= 2 \int_0^\infty x p_n P(X_n > x) + x(1-p_n) P(Y_n > x)dx \\
  &= p_n E(X_n^2) + (1-p_n) E(Y_n^2)
\end{align*}
by \cite[Ex.\ 22b]{resnick2013probability}.
Now
\begin{align*}
\Var(Z_n) &= E(Z_n^2) - E(Z_n)^2 \\
  &= p_n E(X_n^2) + (1-p_n) E(Y_n^2) - (p_n E(X_n) + (1-p_n)E(Y_n))^2 \\
  &= p_n \Var(X_n) + (1-p_n) \Var(Y_n) + p_n(1-p_n)(E(X_n) - E(Y_n))^2.
\end{align*}
% https://stats.stackexchange.com/questions/16608/what-is-the-variance-of-the-weighted-mixture-of-two-gaussians
From the assumptions we know $(1-p_n) \Var(Y_n) \leq (1-p_n) E(Y_n^2) = o(1)$.
And
\begin{align*}
p_n(1-p_n)(E(X_n) - E(Y_n))^2 &\leq (1-p_n)2(E(X_n)^2 + E(Y_n)E(X_n) +
    E(Y_n)^2) \\
  &\leq (1-p_n)2(E(X_n^2) + E(Y_n)E(X_n) + E(Y_n^2)) \\
  &= o(1).
\end{align*}
Thus
\[ E(Z_n) \sim E(X_n) \text{ and } \Var(Z_n) \sim \Var(X_n). \]
By Theorem \ref{thm:slut} theorem we have
%\[
%  \lim_{n \to \infty} F_{X_n} \left(\sqrt{\Var(X_n)}x + E(X_n) \right)
%=\lim_{n \to \infty} F_{X_n} \left(\sqrt{\Var(Z_n)}x + E(Z_n) \right).
%\]
\[
\frac{X_n - E(X_n)}{\sqrt{\Var(X_n)}} \Rightarrow F \textrm{ iff }
  \frac{X_n - E(Z_n)}{\sqrt{\Var(Z_n)}} \Rightarrow F
\]
and
\[
\frac{Z_n - E(Z_n)}{\sqrt{\Var(Z_n)}} \Rightarrow F \textrm{ iff }
  \frac{Z_n - E(X_n)}{\sqrt{\Var(X_n)}} \Rightarrow F.
\]
If $\mathcal{C}(F) \subseteq \mathbb{R}$ is the set of points where $F$ is
continuous, then for $x \in \mathcal{C}(F)$ we have
\begin{align*}
\lim_{n \to \infty} F_{Z_n}\left(\sqrt{\Var(Z_n)}x + E(Z_n)\right) =&
\lim_{n \to \infty} p_n F_{X_n}\left(\sqrt{\Var(Z_n)}x + E(Z_n)\right) \\
& \  + (1-p_n)F_{Y_n}\left(\sqrt{\Var(Z_n)}x + E(Z_n)\right) \\
=& \lim_{n \to \infty}p_n F_{X_n}\left(\sqrt{\Var(Z_n)}x + E(Z_n)\right) \\
=& \lim_{n \to \infty}F_{X_n}\left(\sqrt{\Var(Z_n)}x + E(Z_n)\right) \\
=& \lim_{n \to \infty}F_{X_n}\left(\sqrt{\Var(X_n)}x + E(X_n)\right).
  \qedhere
\end{align*}
\end{proof}

\begin{theorem} \label{thm:circdist}
Assume that $|G| \geq 2$ and that $U \subset \SEQ_\sigma(G)$ is non-empty.
Then the number of cyclic occurrences of $U$ in a uniform random
circular $m$-composition of $a \in G$ is asymptotically normal with mean and
variance asymptotic to those of the number of occurrences of $U$ in a uniform
random word over $G$.
\end{theorem}
\begin{proof}
Let $X_m, X_m^{\langle ap \rangle}, X_m^{\langle ap, c \rangle},
X_m^{\langle c \rangle}$
be the number of cyclic occurrences of $U$ in a uniform random
$m$-composition of $a$, aperiodic $m$-composition of $a$,
aperiodic circular $m$-composition of $a$,
and
circular $m$-composition of $a$.

By Theorem \ref{thm:cyclicdist} we have $(X_m - E(X_m)) /\Var(X_m) \Rightarrow
N(0,1)$.
The quantities $E(X_m)$ and $\Var(X_m)$ are asymptotically proportional to $m$
thus bounded away from $0$.

The number of occurrences in a uniform random (circular) composition is a
mixture of the number of occurrences in a uniform random periodic (circular)
composition and the number of occurrences in a uniform random aperiodic
(circular) composition.
The weights are simply the proportion of (circular) compositions
that are periodic and aperiodic, respectively.
By Theorem \ref{thm:circasympt}, the proportion of $m$-compositions,
circular or not, that are periodic is exponentially small.
There can be at most $m$ occurrences in an $m$-composition, so
moments of the number of occurrences of $U$ in a (circular) $m$-composition are
$m^{O(1)}$.

We are set up to apply Lemma \ref{lem:mixture} twice.
The first application allows us to conclude
that $X_m$ and $X_m^{\langle ap \rangle}$ have the same
limiting distribution.
The second gives that $X_m^{\langle c \rangle}$ and
$X_m^{\langle ap,c \rangle}$ have the same limiting distribution.
Clearly $X_m^{\langle ap \rangle}$ and $X_m^{\langle ap, c \rangle}$ have
the same distributions for all $m$, so we are done.
\end{proof}

Some examples of circular objects follow.

\begin{example}
For a composition $x = (x(1), \ldots, x(m))$, we define
\[ \gap(x) = \max_i x(i) - \min_i x(i) + 1 - |\{x(i) : i = 1,2,\ldots,m\}|, \]
which is the number of parts missing between the minimum and maximum parts
of $x$.
If $\gap(x) = 0$ we say $x$ is \emph{gap-free}.
Research Direction 3.1 parts (3) and (4) in \cite[p.~86]{cofc} ask
for an explicit generating function for the number of circular
compositions/words $x$ such that $\gap(x) = \ell$.

Let $c(m)$ be the number of gap-free $k$-ary words and let
$\tilde{c}(m)$ be the number of circular gap-free words.
The number of gap-free $k$-ary words with $j$ distinct letters is $(k-j+1)j! \{
{m \atop j} \}$.
Thus
\[c(m) = \sum_{j=1}^k (k-j+1) j! \left\{ {m \atop j} \right\} \sim
\sum_{j=1}^k (k-j+1) j^m \sim k^m,\]
where we apply the asymptotics of the Stirling subset numbers \cite{NIST:DLMF}.
The first letter in a gap-free $m$-word is arbitrary if the remaining
$(m-1)$-word has $k$ distinct letters.
The number of such words is $k! \left\{ {m-1 \atop k} \right\} \sim k^{m-1}$,
so the first letter is arbitrary in almost all gap-free words.
Thus for an abelian finite group $G$, the number of gap-free $m$-compositions
of $a$ is \( c_a(m) \sim k^{m-1}\).

Using the familiar Moebius function $\mu$, as in \cite{bender1975applications},
we have
\begin{align*}
\tilde{c}(m) =& \sum_{d | m} \frac{1}{d} \sum_{d'|d} \mu(d/d')
  c(d') \\
\sim& \frac{1}{m} k^m.
\end{align*}
The number of circular gap-free $m$-compositions of $a \in G$ is
\( \tilde{c}_a(m) \sim \frac{1}{m} k^{m-1}\).
\end{example}

\begin{example}
Considering avoidance of the subword pattern $132$, for any total $a$ there
is $1$ composition with $1$ part, namely $(a)$.
For $m \geq 2$, some compositions are grouped into non-trivial equivalence
classes.
For $m=1, \ldots, 5$, the numbers of $132$-avoiding circular $m$-compositions
of $0$ over $\mathbb{Z}_5$ are $1,3,7,23,82$, and the counts for
$m$-compositions of $1$ are $1,3,7,23,77$.
\end{example}

\subsection{Note on counting palindromic compositions} \label{sec:palcomps}

\begin{notation}
\label{not:rev}
For a finite sequence $x=(x(1), \ldots, x(m))$, the reversed sequence is
written $\cev{x} = (x(m), \ldots, x(1))$.
\end{notation}

An unlabeled undirected weighted path of length $m$ restricted according to $D$
is equivalent to an unordered pair $\{x, \cev{x}\}$ where $x, \cev{x} \in
\mathcal{P}(m; D)$, or the singleton $\{x\}$ if $x \in \mathcal{P}(m; D)$
and $x = \cev{x}$.
These may also be called undirected words.
For simplicity we assume all vertices of $D$ are allowed as start and
finish vertices.

\begin{proposition}
Assume $D$ is such that $x \in \mathcal{P}(D) \implies \cev{x} \in
\mathcal{P}(D)$.
Let $\Xi = \{ \xi : \xi \cev{\xi} \in \mathcal{P}(D), \xi \in V(D) \}$.
If $m \geq 2\sigma$ is even,
the number of $G$-weighted undirected paths of length $m$ with total $a$
restricted by $D$ is
\begin{align*}
\tilde{p}_a(m; D) = &\frac{1}{2}p_a(m; D)
+ \sum_{b: 2b = a} \frac{1}{2}p_{b}(m/2; D, V(D), \Xi).
\end{align*}

For $c \in G$, let $\Xi_c =
\{\xi : \xi c \cev{\xi} \in \mathcal{P}(D), \xi \in V(D)\}$.
If $m \geq 2\sigma$ is odd,
the number of $G$-weighted undirected paths of length $m$ with total $a$
restricted by $D$ is
\begin{align*}
\tilde{p}_a(m; D) = &\frac{1}{2}p_a(m; D)
+ \frac{1}{2}\sum_{c \in G} \sum_{b: 2b + c = a} p_{b}((m-1)/2; D, V(D), \Xi_c).
\end{align*}
\end{proposition}
\begin{proof}
The number of undirected paths is determined by dividing by $2$, with an
adjustment for palindromic compositions: those $x$ such that $x = \cev{x}$.
If $m$ is even,
the set of palindromic $m$-compositions is in correspondence with
$\mathcal{P}(m/2; D, V(D), \Xi)$ and
\begin{align*}
\tilde{p}_a(m; D) &=
\frac{1}{2} \left(p_a(m; D) - \sum_{b: 2b = a} p_{b}(m/2; D, V(D), \Xi)
  \right) \\
  & \qquad {} + \sum_{b: 2b = a} p_{b}(m/2; D, V(D), \Xi).
\end{align*}
The case of even $m$ is similar.
\end{proof}

The analogous result for integer compositions is found in
\cite[\S 11]{infinite}.

%maybe remove this next proposition
%
%\begin{proposition}
%Distinguish a nonempty set of words $\Xi \subset \SEQ_\sigma(G)$.
%Then the number of occurrences of $\Xi$ in a uniform random
%$G$-weighted undirected path of length $m$ with total $a$
%is asymptotically normal with mean and variance asymptotic to those of the
%number of occurrences of $\Xi$ in a uniform random word over $G$.
%\end{proposition}
%\begin{proof}
%Omitted;
%see proof of Theorem \ref{thm:circdist}
%\end{proof}

\section{Subsequence pattern avoidance}
\label{sec:subseq}

Given a word $w$ over $[k]$, the \emph{reduction} of $w$, written $\red(w)$, is
obtained by replacing the $j$\textsuperscript{th} smallest letters
of $w$ with $j$'s, for all $j$.
For example, $\red(46632) = 34421$.
A \emph{subsequence pattern}, sometimes called a classical pattern, is a word
over some $[k]$ written with hyphens between letters:
$1 \hype 1 \hype 1 \hype 1 \hype 3 \hype 2 \hype 2 = 1^4 \hype 3 \hype 2^2 \in
[3]^7 = \SEQ_7([3])$.
Given words $w$ of length $m$ and $\tau$ of length $l$,
an occurrence of $\tau$, as a subsequence pattern, in
$w$ is a sequence of indices $1 \leq i_1 < \cdots < i_l \leq m$
such that $\red(w(i_1), \ldots, w(i_l)) = \tau$.

A \emph{partially ordered pattern} is similar to a
subsequence pattern except that not all letters are comparable.
The letters in a partially ordered pattern are from a partially ordered
alphabet; letters shown with the same number of primes are comparable to each
other (e.g.~$1''$ and $2''$), while letters shown without primes are comparable
to all letters of the alphabet.
An occurrence of a partially ordered pattern
in a word $w$ is a distinguished subsequence of terms of $w$ such that the
relative order of two entries in the subsequence need be the same as that of
the corresponding letters in the pattern only if the corresponding letters in
the pattern are comparable; e.g.~the partially ordered pattern
$1' \hype 1'' \hype 2$ is found in the
word $42213$ three times as $4 \underline{2}\underline{2} 1 \underline{3}$, $4
\underline{2} 2 \underline{1} \underline{3}$ and
$42\underline{2}\underline{1}\underline{3}$ (the subsequences of length three
in which the third letter is larger than the first two).

A \emph{generalized pattern} is again similar to a subsequence pattern
except there may or may not be a hyphen between adjacent letters.
If there is no hyphen, those two letters can only match with adjacent letters
in a word.
For example, if $\tau = 11 \hype 2$, then $424135$ has no occurrences of $\tau$
but $244135$ has the occurrence
$2 \underline{4} \underline{4} 13 \underline{5}$.

Subsequence patterns were first studied in the context of permutations
\cite{stanley2007increasing} but are now adapted to different objects.
The number of $k$-ary words of length $m$ avoiding a given subsequence or
generalized pattern
has been studied for a number of different patterns \cite{automata,
burstein1998enumeration, functionaleqs, cofc, restricted132,
pudwell2008enumeration, regev, dmtcs:2140, jelinek2009wilf}.
Specifically, exact results for the avoidance of various subsequence patterns
with at most 2 distinct letters were found in \cite{atmost2}.
For partially ordered pattern-based enumeration for words and other objects,
see \cite{combinterp, popcomp, intropogp, pogpwords}.
The article \cite{gao2011counting} counts words with $r \geq 0$ occurrences of
a some simple subsequence patterns.

Occurrences of subsequence, partially ordered, and generalized patterns
are defined for compositions as they are for words.
The counting question simply changes to, how many compositions with
length $m$ and total $n$ avoid the pattern?

A generating function counting
integer compositions avoiding some $3$-letter patterns
is given in the note \cite{savage2006pattern}, which is a simplification
of earlier work in \cite{atkinson1995priority}.
A recurrence relation is also given in \cite{albert2001permutations}.
Compositions avoiding the remaining $3$-letter patterns, and pairs of
$3$-letter patterns are counted in \cite{heubach2006avoiding}.
That paper also looks at the subsequence pattern $1^p \hype 2\hype 1^q$.
Partially ordered patterns in compositions are considered in
\cite{popcomp}.
Compositions avoiding a generalized pattern of length $3$ are counted
using generating functions in \cite[\S~5.3]{cofc}.

\begin{remark}
Let $p_k(m, r)$ be the number of $k$-ary $m$-words with $r$ occurrences of
the pattern $1 \hype \cdots \hype 1 = 1^p$.
A simple argument shows
\[
\sum_{m, r \geq 0} p_k(m,r) \frac{z^m}{m!} u^r
= \left( \sum_{i \geq 0} u^{\binom{i}{p}} \frac{z^i}{i!} \right)^k.
\]
The recent paper \cite{1p} sheds light on expressions of this form, by
establishing integral representations such as
% page 11
\[
\sum_{n \geq 0} g_n q^{n^2} z^n = \frac{1}{\sqrt{2\pi}} \int_0^{\infty}
  \left(\sum_{b = \pm 1} G\left(e^{bt \sqrt{2\log(q)}} z\right)\right)
  e^{-t^2 / 2} dt,
\]
where $G(z) = \sum_{n \geq 0} g_n z^n$.
Enumerative applications of these representations have yet to be explored.
\end{remark}

\begin{remark} \label{rem:globalcycle}
In the language of \S \ref{sec:intro}, we deal with paths avoiding global
occurrences of digraph patterns.
Undirected paths and directed and undirected cycles are approached in a
similar manner.

For a weighted digraph $\Gamma$, let $s(\Gamma)$ be the symmetric closure of
$\Gamma$, i.e.\ the underlying undirected graph.
Given a weighted path $\Gamma_p$, and digraph pattern instance $P$
which is also a weighted path,
an occurrence of $P$ in $s(\Gamma_p)$ is either an occurrence of $P$ in
$\Gamma_p$ or an occurrence of $P^{-1}$ in $\Gamma_p$, where $P^{-1}$ is
$P$ with arcs reversed.
Let $c(\Gamma_p)$ be the directed cycle formed by adding an arc to $\Gamma_p$.
Then an occurrence of $P$ in $c(\Gamma_p)$ is the occurrence of some
circular shift of $P$ in $\Gamma_p$.
Occurrences of $P$ in $s(c(\Gamma_p))$ are occurrences of circular shifts
and/or reversals of $P$ in $\Gamma_p$.
\end{remark}

\subsection{Words and integer compositions}

This section fills some gaps in the literature on words and integer
compositions that avoid a pattern.
Our main tools are recurrence relations and generating functions, and we
use various standard counting techniques.

\begin{remark}
The random sampling in this section is performed by exploiting the structure
of recurrence relations.
The method achieves exact uniform sampling and makes use of
two rules, one for addition and one for multiplication.
Assume there are three classes of objects, $A, B, C$ and the number
of objects in each are $a,b,c$.
We have the relation $a = b + c$ if $A = B \disjun C$.
Then to draw an object uniformly randomly from $A$, we may draw an
object from $B$ with probability $b/(b + c)$ or an object
from $C$ with probability $c/(b + c)$.
Now if $A = B \times C$, we have $a = bc$.
Here we may draw uniformly at random from $A$ by independently drawing from
both $B$ and $C$.
This simple method is often applicable where we have a recurrence
relation, in which case we recurse until reaching a base case.
\end{remark}

\subsubsection{Pairs of generalized patterns of length $3$} \label{sec:pairs}

While we do not consider every possible pair of generalized patterns of length
$3$ in this section, we give a number of representative examples.
We expect similar techniques apply to most of the remaining such pattern pairs.

% {12-1, 2-11} is annoying

\paragraph*{The pair $\{11\hype 2, 12\hype 3\}$}

%compositions avoiding $12\hype 3$ counted in Theorem~5.21 from
%\cite[p.~147]{cofc}.

We use the generating function $P_k(w| z, u)$ to enumerate integer compositions
over $[k]$ starting with
the subword $w$ and avoiding $\{11\hype 2, 12\hype 3\}$, where $z$ marks the
total and $u$ marks the length.
We write $P_k(z, u)$ for $P_k(e| z, u)$ where $e$ is the sequence of length
$0$, and we write $P(z, u)$ to refer to $\lim_{k \to \infty} P_k(z,u)$.

\begin{proposition}
We have
\[
P(z,u) = \frac{1}{1-u z} \prod_{i \geq 2}
  \left( 1- u z^i \prod_{j=1}^{i-1} (1 + u z^j) \right)^{-1}.
\]
\end{proposition}
\begin{proof}
We follow the proof of Theorem~5.21 \cite{cofc} at least in spirit.

Take a composition $x$ over $[k]$ that avoids $\{11\hype 2, 12\hype 3\}$.
Assume $x$ begins with the part $i$ (where $1 \leq i \leq
k$).
Then $x$ is either $i$ by itself or begins $(i,j, \ldots)$ for some part
$j$.
Now if $j<i$, then the first part of $x$ cannot be involved in an occurrence of
the pattern set, so the composition $(j,x(3), \ldots, x(m))$ is arbitrary as
long as it avoids the pattern set.
On the other hand, if $j \geq i$, no later parts may be greater than $j$
so the composition $(x(3), \ldots, x(m))$ is an arbitrary composition
over $[j]$ avoiding the pattern set.
This gives, for $k \geq 1, 1 \leq i \leq k$,
\begin{align} \label{eq:112123rec}
P_k(i| z, u) &= z^i u +
    \sum_{j = 1}^{i-1} P_k(ij| z, u) + \sum_{j = i}^k P_k(ij| z, u ) \\\nonumber
  &= z^i u + z^i u \left( \sum_{j=1}^{i-1} P_k(j| z,u)
    + \sum_{j=i}^k z^j u P_j(z, u) \right).
\end{align}

Define $G_k(i) = P_k(i| z,u) - P_{k-1}(i|z,u)$ for $k \geq 2$ and
$1 \leq i < k$.
By Equation (\ref{eq:112123rec}) we have
$G_k(i) = z^i u \left( \sum_{j=1}^{i-1}G_k(j) + z^k u P_k(z,u)\right)$
for $1 \leq i < k$.
It can then be seen by induction that
$G_k(i) = u^2 z^{i+k} P_k(z,u) \prod_{j=1}^{i-1}(1+u z^j)$,
$1 \leq i < k$.
We naturally define $G_k(k) = u z^k P_k(z,u)$.
Induction or a combinatorial argument also show that for $k \geq 2$ we have
\begin{align*}
P_k(z,u) - P_{k-1}(z,u)
=& \sum_{i=1}^{k-1} G_k(i) + G_k(k) \\
=& u z^k P_{k}(z,u) \sum_{i=1}^{k-1} u z^i \prod_{j=1}^{i-1}(1+u z^j) +
  u z^k P_k(z,u)\\
=& u z^k P_k(z,u) \left(-1 + \prod_{j=1}^{k-1} (1 + u z^j) \right) +
  u z^k P_k(z,u) \\
=& u z^k P_k(z,u) \prod_{j=1}^{k-1} (1 + u z^j),
\end{align*}
so
\[ P_k(z,u) = \left( 1- u z^k \prod_{j=1}^{k-1} (1 + u z^j)
\right)^{-1} P_{k-1}(z,u). \]
With the initial condition $P_1(z,u) = 1/(1-u z)$ we have
\[
P_k(z,u) = \frac{1}{1-u z} \prod_{i=2}^k
  \left( 1- u z^i \prod_{j=1}^{i-1} (1 + u z^j) \right)^{-1}.
\]
We conclude the result by letting $k \to \infty$.
\end{proof}

Figure \ref{fig:122123} shows randomly generated compositions avoiding
$\{12 \hype 2, 12 \hype 3\}$.
Table \ref{tab:122123} show initial counts.

\begin{table}
\centering
\begin{tabular}[c]{|c|rrrrrrrrrrr|}
\hline
\diagbox{$n$}{$m$} & $0$ & $1$ & $2$ & $3$ & $4$ & $5$ & $6$
  & $7$ & $8$ & $9$ & $10$\\ \hline
0& 1 & 0 & 0 & 0 & 0 & 0 & 0 & 0 & 0 & 0 & 0 \\
1& 0 & 1 & 0 & 0 & 0 & 0 & 0 & 0 & 0 & 0 & 0 \\
2& 0 & 1 & 1 & 0 & 0 & 0 & 0 & 0 & 0 & 0 & 0 \\
3& 0 & 1 & 2 & 1 & 0 & 0 & 0 & 0 & 0 & 0 & 0 \\
4& 0 & 1 & 3 & 2 & 1 & 0 & 0 & 0 & 0 & 0 & 0 \\
5& 0 & 1 & 4 & 5 & 2 & 1 & 0 & 0 & 0 & 0 & 0 \\
6& 0 & 1 & 5 & 8 & 6 & 2 & 1 & 0 & 0 & 0 & 0 \\
7& 0 & 1 & 6 & 12 & 12 & 6 & 2 & 1 & 0 & 0 & 0 \\
8& 0 & 1 & 7 & 17 & 20 & 15 & 6 & 2 & 1 & 0 & 0 \\
9& 0 & 1 & 8 & 23 & 33 & 28 & 16 & 6 & 2 & 1 & 0 \\
10& 0 & 1 & 9 & 29 & 50 & 50 & 35 & 16 & 6 & 2 & 1 \\\hline
\end{tabular}
\caption{Counts of the $m$-compositions of $n$ avoiding
  $\{12 \hype 2, 12 \hype 3\}$.}
\label{tab:122123}
\end{table}

\begin{figure}
\centering
\includegraphics[width=5in]{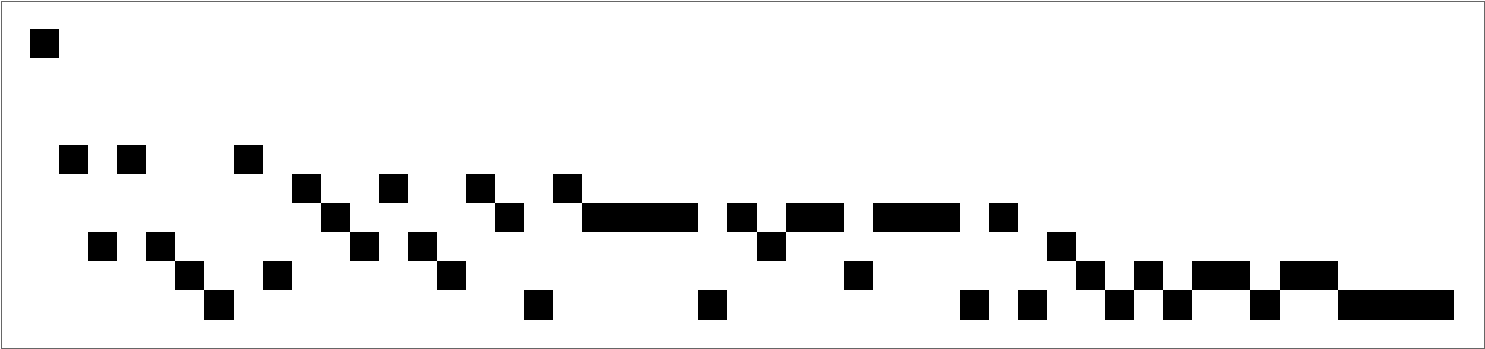}
\includegraphics[width=4.8in]{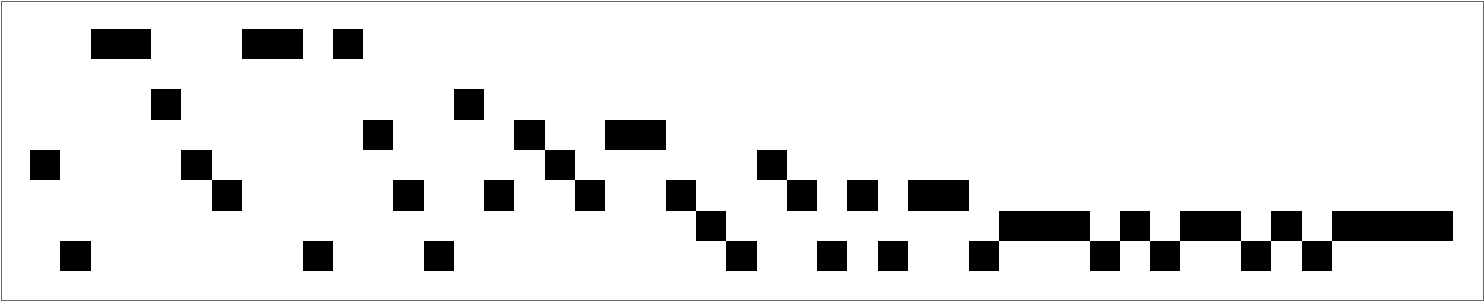}

\caption{Uniform-randomly generated compositions of $150$ avoiding
  $\{12 \hype 2, 12 \hype 3\}$.
  \label{fig:122123}}
\end{figure}

\paragraph*{The pair $\{21\hype 2, 2\hype 12\}$}

We count $k$-ary words avoiding the set of generalized patterns
$\{21\hype 2, 2\hype 12\}$.
Note that it is not true that letters $1$ must be found only in contiguous
blocks at the very beginning and/or very end of a word.
For example, $312$ avoids the patterns and has the least letter in the
middle.

Let $p_k(m)$ be the number of $k$-ary $m$-words that avoid
$\{21\hype 2, 2\hype 12\}$.
\begin{proposition}
For $k \geq 1$ we have
\[ p_k(m) = \frac{2^{1-k}}{(k-1)!}m^{2k-2} + O(m^{2k-3}),
  \qquad m \to \infty. \]
\end{proposition}
\begin{proof}
Take a $k$-ary word $w$ avoiding $\{21\hype 2, 2\hype 12\}$.
There are $p_{k-1}(m)$ such words with no letters $k$.
We assume the greatest letter present in $w$ is $k$.
All copies of $k$ must be contiguous in order to avoid the patterns.
If we delete these copies of $k$ from $w$, the remaining word has the same
structure but is a word over $[k-1]$.
If there are $b$ letters $k$, there are $m-b+1$ possible positions of the
contiguous run of these letters.
Thus
\[ p_k(m) = p_{k-1}(m) + \sum_{b=1}^m (m-b+1) p_{k-1}(m-b),
  \qquad k \geq 1, m \geq 0, \]
and $p_0(m) = [m=0]$.
Passing to the generating function
$P_k(z) = \sum_{m \geq 0} p_k(m) z^m$
gives
\[ P_k(z) = P_{k-1}(z) + \frac{z}{1-z} D_z(z P_{k-1}(z)), P_0(z) =
  1. \]
By induction for $k \geq 1$, $P_k(z)$ has a unique singularity at $1$ where
it has a pole of order $2k-1$ and
 \[ P_k(z) = \frac{\prod_{i=1}^{k-1}(2i-1)}{(1-z)^{2k-1}} + O((1-z)^{-(2k-2)}),
  \qquad z \to 1.\]
Using Theorem \ref{thm:ratgf} we extract asymptotics for the coefficients of
$P_k(z)$ to obtain
\[ p_k(m) \sim \frac{\prod_{i=1}^{k-1} (2i-1)}{(2k-2)!}m^{2k-2}. \]
% offdef
We have
\[
\frac{\prod_{i=1}^{k-1} (2i-1)}{(2k-2)!} =
\frac{2k-3}{(2k-2)(2k-3)}
\frac{\prod_{i=1}^{k-2} (2i-1)}{(2(k-1)-2)!} = \frac{2^{1-k}}{(k-1)!}
\]
so we may conclude the result.
\end{proof}

Table \ref{tab:212212} gives initial counts of words avoiding this pattern
set $\{21\hype 2, 2\hype 12\}$, and Figure \ref{fig:212212} has
randomly generated examples.

\begin{table}
\centering
\begin{tabular}[c]{|c|rrrrrrrrrrr|}
\hline
\diagbox{$k$}{$m$} & $0$ & $1$ & $2$ & $3$ & $4$ & $5$ & $6$
  & $7$ & $8$ & $9$ & $10$\\ \hline
2& 1 & 2 & 4 & 7 & 11 & 16 & 22 & 29 & 37 & 46 & 56 \\
3& 1 & 3 & 9 & 24 & 56 & 116 & 218 & 379 & 619 & 961 & 1431 \\
4& 1 & 4 & 16 & 58 & 186 & 526 & 1324 & 3011 & 6283 & 12196 & 22276 \\\hline
\end{tabular}
\caption{Counts of the $k$-ary $m$-words avoiding $\{21 \hype 2, 2 \hype 12\}$.}
\label{tab:212212}
\end{table}

\begin{figure}
\centering
\includegraphics[width=5.5in]{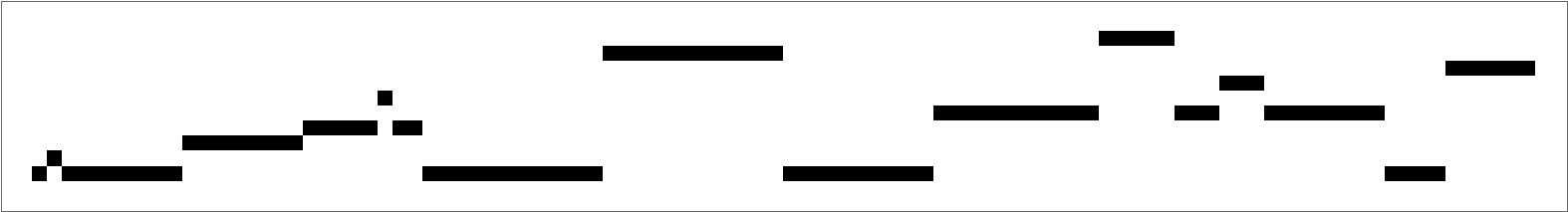}
\includegraphics[width=5.5in]{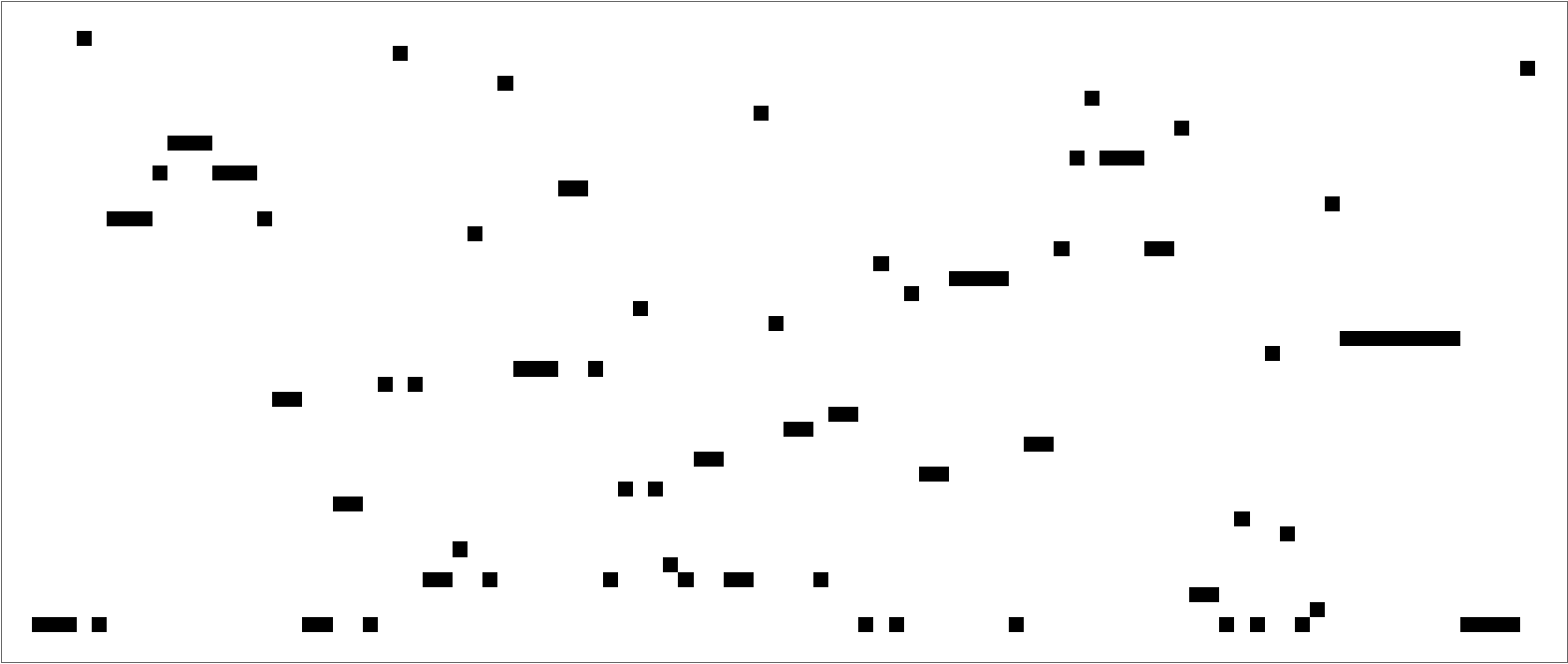}

\caption{Uniform-randomly generated $10$-ary (above) and $40$-ary (below)
  $100$-words avoiding $\{21 \hype 2, 2 \hype 12\}$.
  \label{fig:212212}}
\end{figure}

\paragraph*{The pair $\{11\hype 2, 12\hype 1\}$}

Let $P_A(w|z,u)$ be the generating function for integer compositions over the
finite set $A \subset \mathbb{Z}_{>0}$, starting with the subword $w$, that
avoid the pattern set $\{11\hype 2, 12\hype 1\}$, where $z$ marks total and $u$
marks length.
The generating function $P_A(z,u)$ refers to $P_A(e|z,u)$ where $e$ is the
empty word.
We use the notation $M(A,i) = \{ j: j \in A, j \leq i \}$.

\begin{proposition}
We have
\[ P_A(z, u) = 1 + \sum_{i \in A}P_A(i|z,u), \]
and
\begin{align*}
P_A(i| z, u) =& z^i u + \sum_{j \in A, j < i} z^i u P_A(j|z,u) + z^{2i}u^2
P_{M(A,i)}(z,u) \\
&\qquad {} + \sum_{j \in A, j > i} z^i u P_{A \setminus \{i\}}(j|z,u).
\end{align*}
\end{proposition}
\begin{proof}
Let $x$ be a composition over $A$ avoiding $\{11\hype 2, 12\hype 1\}$.
If $x$ begins with the part $i \in A$, either $x = (i)$ or $x = (i,j,\ldots)$.
In the latter case we may have $j < i, j=i$, or $j > i$, so we get
\[
P_A(i| z, u) = z^i u +
  \sum_{j \in A, j < i}P_A(ij|z,u) + P_A(ii|z,u)
  + \sum_{j \in A, j > i} P_A(ij|z,u). \]
If $j<i$ then $i$ is part of an occurrence only if $j$ is.
So if we delete $i$ the remaining composition is arbitrary.
If $j=i$ then $x = (i,i,\ldots)$.
In order to avoid $11 \hype 2$ the composition remaining after deleting
$ii$ is arbitrary as long as no parts are above $i$.
Finally if $j>i$ we may delete $i$ and have an arbitrary composition
starting with $j$ as
long as the part $i$ does not appear.
Thus we have
\begin{align*}
& z^i u +
  \sum_{j \in A, j < i}P_A(ij|z,u) + P_A(ii|z,u)
  + \sum_{j \in A, j > i} P_A(ij|z,u) \\
=& z^i u + \sum_{j \in A, j < i} z^i u P_A(j|z,u) + z^{2i}u^2 P_{M(A,i)}(z,u) \\
&\qquad {} + \sum_{j \in A, j > i} z^i u P_{A \setminus \{i\}}(j|z,u).
  \qedhere
\end{align*}
\end{proof}

%$11\hype 2$ GF not simplified by Mansour, even for words

Table \ref{tab:112121} shows initial counts of compositions avoiding
$\{11\hype 2, 12\hype 1\}$, and Figure \ref{fig:112121} shows
randomly-generated objects.

\begin{table}
\centering
\begin{tabular}[c]{|c|rrrrrrrrrrr|}
\hline
\diagbox{$n$}{$m$} & $0$ & $1$ & $2$ & $3$ & $4$ & $5$ & $6$
  & $7$ & $8$ & $9$ & $10$\\ \hline
0 &1 & 0 & 0 & 0 & 0 & 0 & 0 & 0 & 0 & 0 & 0 \\
1 &0 & 1 & 0 & 0 & 0 & 0 & 0 & 0 & 0 & 0 & 0 \\
2 &0 & 1 & 1 & 0 & 0 & 0 & 0 & 0 & 0 & 0 & 0 \\
3 &0 & 1 & 2 & 1 & 0 & 0 & 0 & 0 & 0 & 0 & 0 \\
4 &0 & 1 & 3 & 1 & 1 & 0 & 0 & 0 & 0 & 0 & 0 \\
5 &0 & 1 & 4 & 4 & 1 & 1 & 0 & 0 & 0 & 0 & 0 \\
6 &0 & 1 & 5 & 8 & 2 & 1 & 1 & 0 & 0 & 0 & 0 \\
7 &0 & 1 & 6 & 11 & 7 & 2 & 1 & 1 & 0 & 0 & 0 \\
8 &0 & 1 & 7 & 17 & 11 & 4 & 2 & 1 & 1 & 0 & 0 \\
9 &0 & 1 & 8 & 24 & 24 & 10 & 4 & 2 & 1 & 1 & 0 \\
10&0 & 1 & 9 & 30 & 42 & 16 & 6 & 4 & 2 & 1 & 1 \\\hline
\end{tabular}
\caption{Counts of the $m$-compositions of $n$ avoiding
  $\{11 \hype 2, 12 \hype 1\}$.}
\label{tab:112121}
\end{table}

\begin{figure}
\centering
\includegraphics[width=1.6in]{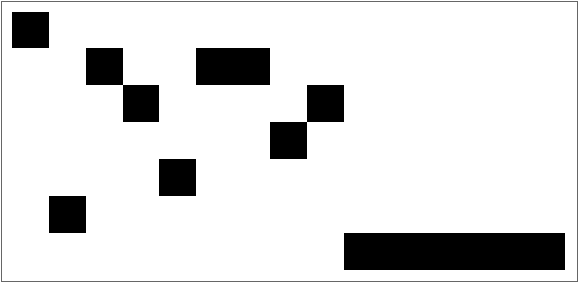}
\includegraphics[width=1.3in]{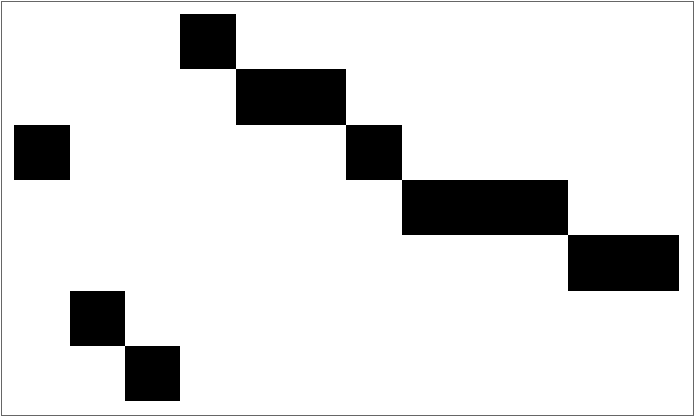}
\includegraphics[width=2.2in]{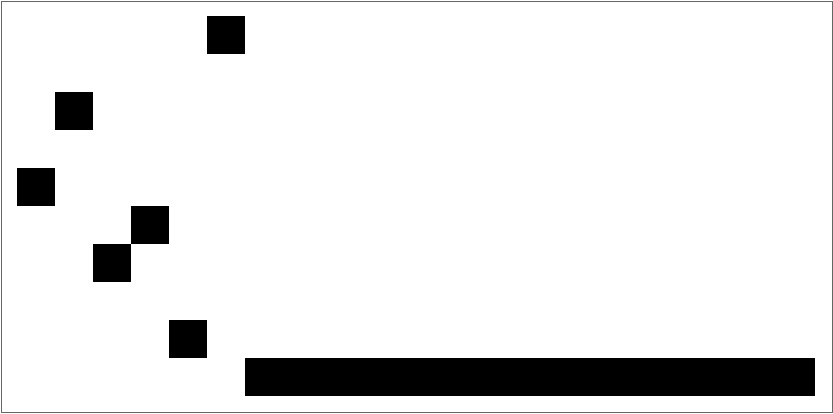}

\caption{Uniform-randomly generated compositions of $50$
  avoiding $\{11\hype 2, 12\hype 1 \}$.
  \label{fig:112121}}
\end{figure}

\paragraph*{The pair $\{12\hype 3, 3\hype 21\}$}

Let $p_k(m)$ be the number of $k$-ary $m$-words that avoid the pattern set
$\{12\hype 3, 3\hype 21\}$.

\begin{proposition}
For $k \geq 2$ we have
$p_k(m) \sim A_k \cdot \left(\sqrt{k-1}+1\right)^m$, $m \to \infty$.
\end{proposition}
\begin{proof}
Let $w$ be a $k$-ary word avoiding $\{12\hype 3, 3\hype 21\}$.
Then either $w$ contains no letters $k$, or $w$ can be written as the
concatenation
\[ w' k^{j_1} w_1 \cdots k^{j_r} w_r k^{s} w'', \]
where $r \geq 0$, $w'$ is a word over $[k-1]$ that avoids $\{12,3\hype 21\}$,
$w''$ is a word on $[k-1]$ that avoids $\{21,12\hype 3\}$, and
the $w_i$ are words on $[k-1]$ that avoid $\{12,21\}$.

Nonempty words avoiding $\{12,21\}$ clearly have one distinct letter repeated
some number of times.
Words avoiding $\{21,12\hype 3\}$ are either empty, have one distinct letter,
or have exactly one increase and no decreases.

This translates to
\[
P_k(z) = P_{k-1}(z) + G_{k-1}(z)
  \frac{1}{1-(z/(1-z))H_{k-1}(z)}
 \frac{z}{1-z} G_{k-1}(z),
\]
where $G_k(z)$ counts words avoiding $\{21,12\hype 3\}$
(or $\{12,3\hype 21\}$), so
\[
G_k(z) = 1 + \sum_{i=1}^k \frac{z}{1-z} + \sum_{i=1}^{k-1}\frac{z}{1-z}
  \sum_{j=i+1}^{k} \frac{z}{1-z}
= 1 + k \frac{z}{1-z} + \frac{k^2 -k}{2}\frac{z^2}{(1-z)^2},
\]
and
$H_k(z) = k z/(1-z)$ counts nonempty words avoiding $\{12,21\}$.

Iterating the recurrence relation, we have
\[
P_k(z) = \frac{1}{1-z} + \sum_{j=2}^{k} G_{j-1}^2(z)\frac{z}{1-z}
  \frac{1}{1-H_{j-1}(z)z/(1-z)}.
\]
We examine the factor
\[
\frac{1}{1-H_{j-1}z/(1-z)} = \frac{-z^2+2 z-1}{j z^2-2 z^2+2 z-1}
\]
The root of $j z^2-2 z^2+2 z-1$ with smallest absolute value is
$z = \frac{1}{\sqrt{j-1}+1}$.
For $j \geq 2$, this value is a simple pole less than $1$ and decreasing
(toward $0$).
%the simple pole thing is easy to see
By Theorem \ref{thm:ratgf} we conclude the statement.
%Since $\lim_{z \to \frac{1}{\sqrt{k-1}+1}}P_k(z)
%  (z - \frac{1}{\sqrt{k-1}+1}) > 0$ we conclude the statement.
\end{proof}

Table \ref{tab:123321} shows initial coefficients of $P_k(z)$.

\begin{table}
\centering
\begin{tabular}[c]{|c|rrrrrrrrrr|}
\hline
\diagbox{$k$}{$m$} & $0$ & $1$ & $2$ & $3$ & $4$ & $5$ & $6$
  & $7$ & $8$ & $9$\\ \hline
3& 1 & 3 & 9 & 25 & 65 & 162 & 394 & 946 & 2258 & 5379 \\
4& 1 & 4 & 16 & 56 & 174 & 502 & 1388 & 3755 & 10059 & 26857 \\
5& 1 & 5 & 25 & 105 & 375 & 1211 & 3689 & 10920 & 31920 & 92930 \\\hline
\end{tabular}
\caption{Counts of $k$-ary $m$-words avoiding $\{12 \hype 3, 3 \hype 21\}$.}
\label{tab:123321}
\end{table}

\subsubsection{Some partially ordered patterns with $2$ letters} \label{sec:pops}

Here we consider the family of partially ordered patterns of the form
$2^p \hype 1' \hype \cdots \hype 1^{(q)} \hype 2^r =
2\hype \cdots \hype 2\hype  1' \hype 1'' \hype
  \cdots \hype 1^{(q)} \hype 2\hype \cdots\hype 2$.
We break into cases depending on the values of $p,q,r$.

\paragraph*{Case $p,q,r\geq 1$}

%We consider the pattern $2^p \hype 1' \hype \cdots \hype 1^{(q)} \hype 2^r$.
%For general $q$, neither words nor compositions avoiding this pattern
%have been counted.
%
Let $h_k(n,m)$ be the number of integer $m$-compositions of $n$ over $[k]$ that
avoid $2^p \hype 1' \hype \cdots \hype 1^{(q)} \hype 2^r$ where
$p,q,r \geq 1$.

\begin{proposition} \label{pro:2112etc}
We have the recurrence relation
\begin{align*}
&h_k(n,m) \\
  &= \sum_{b=0}^{m}[0 \leq b < p + r \textrm{ or } b > m-q]
  \binom{m}{b}h_{k-1}(n-bk,m-b) \\
  & \qquad {} + \sum_{b=p+r}^{m-q} \sum_{t=M}^{M+q-1}
    \binom{w-2}{M-2}\binom{m-t+1}{(p-1) + (r-1) + 1} h_{k-1}(n-bk, m-b),
\end{align*}
for $m, n, \geq 0$ and $k \geq 2$.
For $k=1$, we have $h_1(n,m) = [n=m]$.
\end{proposition}
\begin{proof}
Assume $k \geq 2$ and let $b$ be the number of letters $k$ in a word $w$.
If $b \leq p+r-1$ or $b \geq m - q + 1$, these letters cannot be part of an
occurrence, so their positions do not matter, thus there are $\binom{n }{
b}h_{k-1}(n-bk,m-b)$ such words $w$.
If $b \geq p+r$, then between the $p$\textsuperscript{th} $k$ from the left and
the $r$\textsuperscript{th} $k$ from the right there must be at most $q-1$
letters that are not $k$.
Let $t$ be the number of all letters between the $p$\textsuperscript{th} $k$
from the left and the $r$\textsuperscript{th} $k$ from the right, and let
$M=b-(p-1)-(r-1)$ be the number of letters $k$ among those letters.
Then there are
\[ \sum_{t=M}^{M+q-1} \binom{t-2}{M-2}\binom{m-t+1}{(p-1) + (r-1) + 1} \]
possible ways of placing the letters $k$ in $w$.
The first binomial coefficient chooses the letters $k$ between the
$p$\textsuperscript{th} from left and $r$\textsuperscript{th} from right, and
the second chooses the position of the remaining letters $k$ as well as the
position of the $p$\textsuperscript{th} from the left.

So for $m \geq p+r$, this gives
\begin{align*}
&h_k(n,m) \\
  &= \sum_{b=0}^{m}[0 \leq b < p + r \textrm{ or } b > m-q]
  \binom{m}{b}h_{k-1}(n-bk,m-b) \\
  & \qquad {} + \sum_{b=p+r}^{m-q} \sum_{t=M}^{M+q-1}
    \binom{t-2}{M-2}\binom{m-t+1}{(p-1) + (r-1) + 1} h_{k-1}(n-bk, m-b),
\end{align*}
as desired.
It can be verified that the recurrence is valid as well for the values
$0 \leq m < p + r$.
\end{proof}

We note that $h_k(n,m)$ is a function of $p+r$ rather than $p$ and $r$
independently.
% note that we cannot combine the sums

\paragraph*{Case $p=1, q=2, r=1$}

For the special case $\tau = 2 \hype 1' \hype 1'' \hype 2$ we illustrate
an asymptotic analysis.
We further simplify by ignoring totals and counting words.
Let $H_k(z) = \sum_{m \geq 0} h_k(m) z^m$ where $h_k(m)$ is the number of
$k$-ary $m$-words that avoid $2 \hype 1' \hype 1'' \hype 2$.

\begin{proposition}
If $k \geq 2$ we have
$h_k(m) = \frac{A_k}{(3(k-1))!} m^{3(k-1)}(1 + O(m^{-1})), m \to \infty$,
where $A_k = \prod_{j=1}^{k-1}(1+3(j-1))$.
\end{proposition}
\begin{proof}
By Proposition \ref{pro:2112etc} we know
\begin{align*}
h_k(m) &= h_{k-1}(m) + m h_{k-1}(m-1) \\
  &\qquad + \sum_{b=2}^{m-2}((m-b+1) + (m-b)(b-1))h_{k-1}(m-b) \\
  &\qquad + m h_{k-1}(1) + h_{k-1}(0) \\
&=\sum_{b=0}^m((m-b+1) + (m-b)(b-1))h_{k-1}(m-b).
\end{align*}
Passing to generating functions, we have
\[ H_k(z) = \frac{1}{1-z}H_{k-1}(z) + \frac{z^2}{(1-z)^2}H'_{k-1}(z) \]
for $k \geq 2$, and $H_1(z) = 1/(1-z)$.

By induction $H_k(z)$ is rational with unique singularity at $z=1$ and
$H_k(z) \sim A_k \frac{1}{(1-z)^{1 + 3(k-1)}}, z \to 1$ so
by Theorem \ref{thm:ratgf} we conclude the statement.
\end{proof}

Table \ref{tab:2112} shows initial coefficients of $H_k(z)$.
Figure \ref{fig:2112} has uniform-randomly generated words avoiding
$2 \hype 1' \hype 1'' \hype 2$.

\begin{table}
\centering
\begin{tabular}[c]{|c|rrrrrrrrrr|}
\hline
\diagbox{$k$}{$m$} & $0$ & $1$ & $2$ & $3$ & $4$ & $5$ & $6$
  & $7$ & $8$ & $9$ \\ \hline
1& 1 & 1 & 1 & 1 & 1 & 1 & 1 & 1 & 1 & 1  \\
2& 1 & 2 & 4 & 8 & 15 & 26 & 42 & 64 & 93 & 130 \\
3& 1 & 3 & 9 & 27 & 76 & 196 & 462 & 1002 & 2019 & 3817 \\
4& 1 & 4 & 16 & 64 & 242 & 844 & 2692 & 7852 & 21043 & 52184 \\
5& 1 & 5 & 25 & 125 & 595 & 2635 & 10743 & 40163 & 137738 & 434798 \\\hline
\end{tabular}
\caption{Counts of the $k$-ary words of length $m$ avoiding
  $2 \hype 1' \hype 1'' \hype 2$.}
\label{tab:2112}
\end{table}

\begin{figure}
\centering
\includegraphics[width=5.5in]{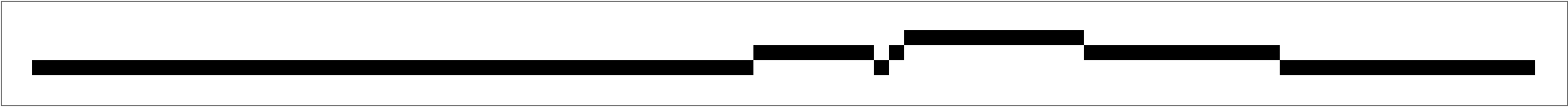}
\includegraphics[width=5.5in]{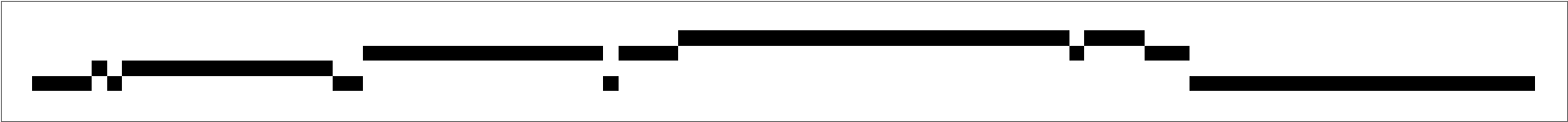}
\includegraphics[width=5.5in]{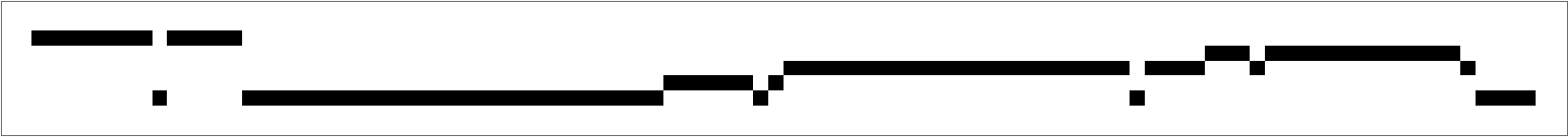}
\includegraphics[width=5.5in]{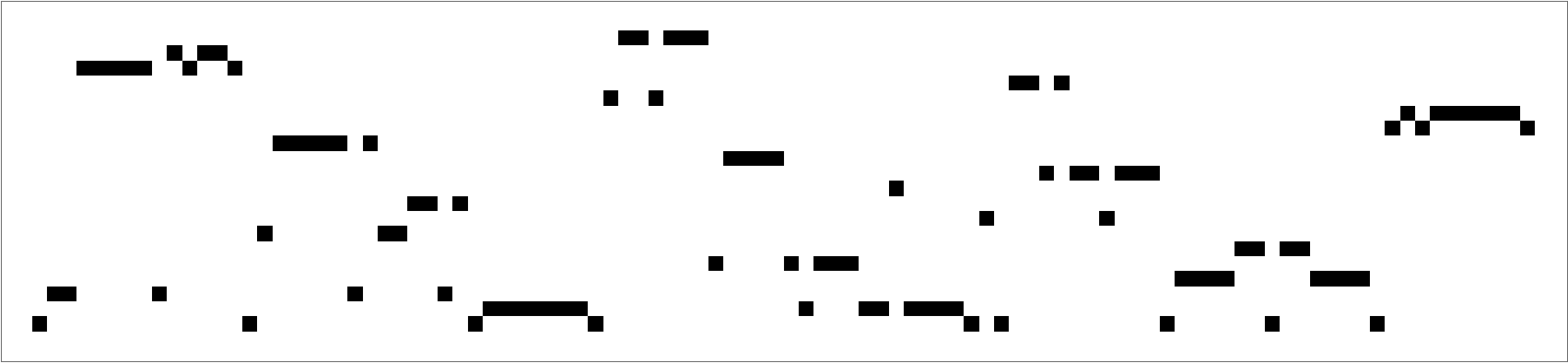}

\caption{Uniform-randomly generated $k$-ary $100$-words where $k=3,4,5,20$
  (top to bottom) avoiding $2 \hype 1' \hype 1'' \hype 2$.
  \label{fig:2112}}
\end{figure}

\paragraph*{Case $p=0$}

If we have $q, r \geq 1$ but
we allow $p=0$, we have the pattern
$\tau = 1' \hype 1'' \hype \cdots \hype 1^{(q)} \hype 2^r$.
Words avoiding $\tau$ were counted in \cite[\S 2]{gao2011counting};
integer compositions were left as an open problem.
Let $h_k(n,m)$ be the number of integer $m$-compositions of $n$ over
$[k]$ that avoid $\tau$.

\begin{proposition}
For $k \geq 2$ and $m \geq q+r$ we have
\begin{align*}
  h_k(n,m) &= \sum_{j=1}^q \binom{q}{j}h_k(n-jk, m-j) (-1)^{j+1} \\
  &\qquad {} + \sum_{b=0}^{r-1}\binom{m-q}{b}h_{k-1}(n-bk, m-b),
\end{align*}
and $h_1(n,m) = [n=m]$.
\end{proposition}
\begin{proof}
For the range $m \geq q+r$, we recursively count $m$-compositions $x$
avoiding $\tau$
by first counting $x$ such that at least one of the first $q$ letters is $k$.
By the principle of inclusion-exclusion,
the number of such $x$ is \[\sum_{j=1}^q N_j (-1)^{j+1}, \]
where $N_j$ is the sum, over all $j$-subsets of the first $q$ positions, of the
number of compositions $x$ with $k$'s in the positions given by the subset.
The quantity $N_j$ is given by \[ N_j = \binom{q}{j}h_k(n-jk, m-j), \]
since inserting $j$ copies of $k$ into any of the first $q$ positions of
an $(m-j)$-composition is reversible and does not affect the number of
occurrences of $\tau$.

Now we count the compositions $x$ that have no letters $k$ in their first $q$
positions.
Let $b$ be the number of letters $k$ in $x$.
If $b \leq r-1$, then there are
not enough letters $k$ to be part of a pattern, so there are
\[ \sum_{b=0}^{r-1}\binom{m-q}{b}h_{k-1}(n-bk, m-b), \]
compositions of this kind.

If $b \geq r$ then there is at least one occurrence of $\tau$.
Thus we have, for $m \geq q+r, k \geq 2$,
\begin{align*} \label{eq:frmk}
  h_k(n,m) &= \sum_{j=1}^q \binom{q}{j}h_k(n-jk, m-j) (-1)^{j+1} \\
  &\qquad {} + \sum_{b=0}^{r-1}\binom{m-q}{b}h_{k-1}(n-bk, m-b).
\end{align*}

For $m < q + r$, we have
$h_k(n,m) = [z^n u^m] \left(\sum_{j=1}^k z^j u\right)^m$.
\end{proof}

We expect similar techniques to those used for $p,q,r \geq 1$ and
$p=0$ apply to count avoidance of $\tau$ where
$\tau$ involves the letter $2$ and mutually incomparable symbols $1^{(j)}$.

%\paragraph*{Special case: $p=2, q=1$}
%
%We have, for $m \geq 3, k \geq 1$,
%\begin{align} \label{eq:frmk21}
%  f_k(n,m) &= 2f_k(n-k, m-1) - f_k(n-2k, m-2) + f_{k-1}(n, m)
%\end{align}

%F_k(z,u) = -uz^k + 2uz^k F_k(z,u) - u^2 z^{2k} F_k(z,u) + F_{k-1}(z,u)

Figure \ref{fig:112} shows randomly generated compositions avoiding
$1' \hype 1'' \hype 2$.

\begin{figure}
\centering
\includegraphics[width=2.15in]{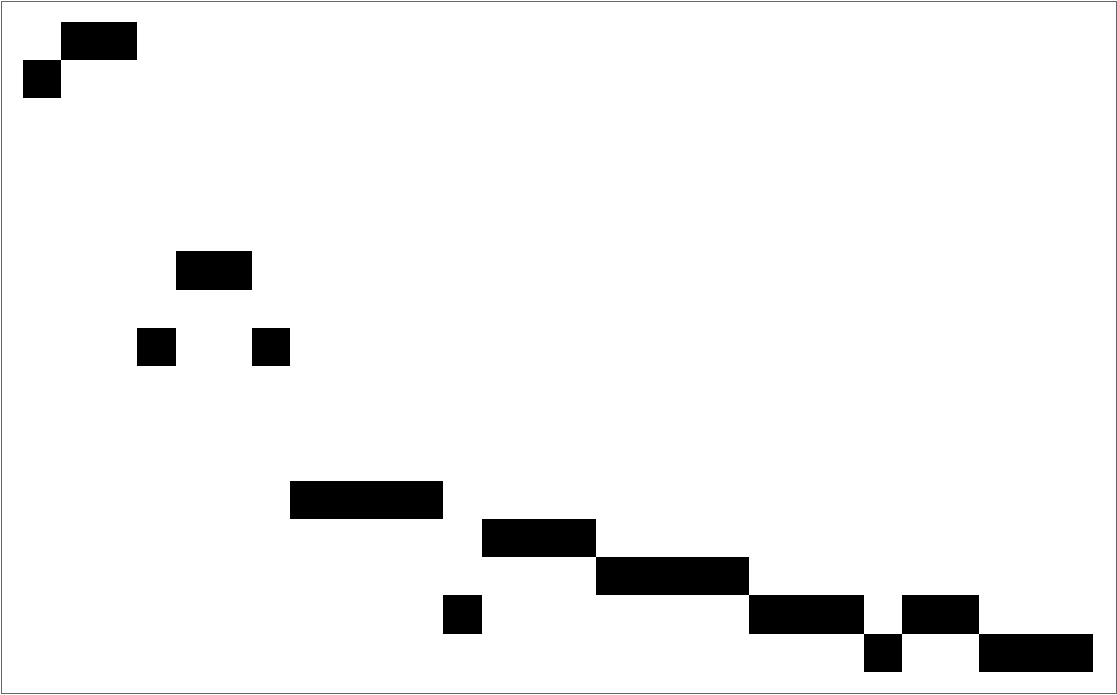}
\includegraphics[width=2in]{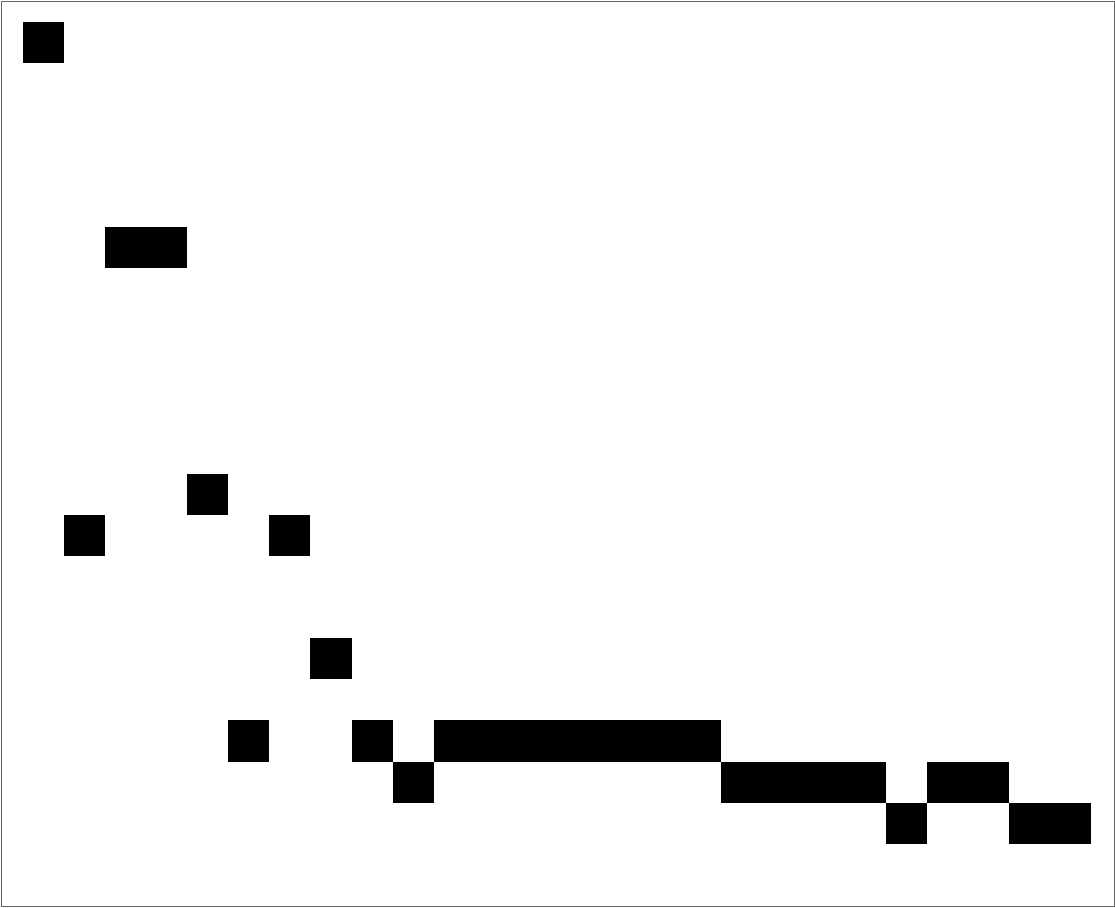}

\caption{Uniform-randomly generated compositions of $150$ avoiding $1' \hype
  1'' \hype 2$.
  \label{fig:112}}
\end{figure}

\subsubsection{Note on counting with symmetries} \label{sec:subseqsym}

% asymptotics are not apparent in general because they would require
% monotonicity which is not obvious

\paragraph*{Reversal (unlabeled undirected paths)}
As in \S \ref{sec:palcomps}
we say that an \emph{undirected word} is an unordered pair $\{w, \cev{w}\}$
where $w \neq \cev{w}$ or simply $\{w\}$ if $w = \cev{w}$.
The pair $\{w, \cev{w}\}$ avoids a pattern $\tau$ if and only if both $w$ and $\cev{w}$
avoid $\tau$.
For a subsequence pattern $\tau$, we define the set $\folds(\tau)$ to be all
possible words $\tau'$ obtained by the following procedure.
Split $\tau$ into two subwords $\tau = \tau_1 \tau_2$.
Take words $\tau'$ such that $\tau_1$ and $\cev{\tau_2}$ are subsequences of
$\tau'$.

\begin{proposition}
Let $\tilde{p}_k(m; T)$ be the number of undirected $k$-ary $m$-words that
avoid all subsequence patterns in the set $T$.
Then for even $m$ we have
\[ \tilde{p}_k(m; \{\tau\}) = \frac{1}{2}p_k(m; \{\tau, \cev{\tau}\})
  + \frac{1}{2}p_k(m/2; \folds(\tau)). \]
\end{proposition}
\begin{proof}
An undirected word is either palindromic or not.
If not, it corresponds to a pair of $2$ distinct directed words.
If it is palindromic, it takes the form $u \cev{u}$ where $u$ avoids
$\folds(\tau)$.
Also,
\[
\tilde{p}_k(m; \{\tau\}) = \frac{1}{2} \big(p_k(m; \{\tau, \cev{\tau}\})
  - p_k(m/2; \folds(\tau)) \big)
 + p_k(m/2; \folds(\tau)). \qedhere
\]
\end{proof}

The case of odd $m$ is less straightforward.
Burstein \cite{burstein1998enumeration} counts a number of examples of
words avoiding the set $\{\tau, \cev{\tau}\}$ for where $\tau$ is a short
subsequence pattern with no repeated letters.
Avoiding $\folds(\tau)$ becomes quite restrictive but is not necessarily
impossible.
Any non-decreasing word avoids $\folds(2 \hype 1 \hype 3)$.

%odd $m$.
%similar to above except we break $\tau$ into three parts
%$\tau = \tau_1 \tau_2 \tau_3$ where $\tau_2$ has length at most $1$,
%and a fold involves interspersing the reversal of $\tau_3$ with $\tau_1$
%and leaving $\tau_2$ at the end.

\paragraph*{Circular shift (unlabeled cycles)}

A circular word is an equivalence class of words where two words are equivalent
if one is a circular shift of the other.

As in Remark \ref{rem:globalcycle},
a word $w$ cyclically avoids a subsequence pattern $\tau$ if all circular
shifts of $w$ avoid the pattern.
Alternatively, $w$ cyclically avoids $\tau$ if $w$ avoids all circular
shifts of $\tau$.
We observe that we do not have property that if $u$ cyclically avoids a
pattern so does $uu$.
For example, to cyclically avoid $1 \hype 2 \hype 3$, we avoid the set
$T = \{1\hype 2\hype 3, 3\hype 1\hype 2, 2\hype 3\hype 1\}$, and if
$u=321$, then $u$ avoids the pattern but $u u=321321$
contains $2\hype 3\hype 1$.

We define the set $\merges_i(\tau)$ to contain all $\tau'$ produced by the
following procedure.
If $i$ is an integer satisfying $1 \leq i < |\tau|$, we consider any circular
shift $\tau^*$ of $\tau$ expressed as a concatenation of subwords
$\tau^* = \tau_1 \cdots \tau_i$, some of which may be empty.
For each $\tau^*$, include any word $\tau'$ such that each $\tau_j$ is a
subsequence of $\tau'$.

\begin{proposition}
Let $p_k(m; T)$ be the number of $k$-ary $m$-words avoiding subsequence
patterns in $T$, and let $\tilde{c}_k(m; T)$ be the same for circular $m$-words.
Let $t$ be the number of distinct letters in the subsequence pattern $\tau$.
We have
\begin{align*}
&\tilde{c}_k\big(m; \{\tau\}\big) \\
  &= \sum_{j | m} \frac{1}{j} \sum_{d|j} \mu(j/d) \Big(
  [m/d < |\tau|]p_k\big(d; \merges_{m/d}(\tau)\big) +
  [m/d \geq |\tau|]\alpha(d) \Big),
\end{align*}
where $\alpha(d) = \sum_{j=1}^{t-1} \binom{k}{j}
\left\{ {d\atop j} \right\}$ is the number of all $k$-ary $d$-words
with fewer than $t$ distinct letters, and $\mu$ is the Moebius function.
\end{proposition}
\begin{proof}
The period of an $m$-word $w$ is the least integer $n$ such that $w=u^{m/n}$
for some word $u$.
Fix $m$, and for $j \leq m$, define $f_j$ to be the number of $k$-ary $m$-words
cyclically avoiding $\tau$ with period dividing $j$.
Then by Moebius inversion the number of words with period exactly $j$ is
\(
\sum_{d|j} \mu(d) f_{j/d},
\)
and so the number of all circular words is
\[
\sum_{j|m} \frac{1}{j} \sum_{d|j} \mu(d) f_{j/d} =
\sum_{j|m} \frac{1}{j} \sum_{d|j} \mu(j/d) f_{d}.
\]
If an $m$-word $w$ has period dividing $d$, then it has the form $w=u^{m/d}$
for a subword $u$ of length $d$.
If $m/d \geq |\tau|$ then $w$ cyclically contains $\tau$ if and only if $u$ contains at
least as many distinct letters as there are in $\tau$.
If $m/d < |\tau|$, then $w$ cyclically avoids $\tau$ if and only if $u$ avoids
$\merges_{m/d}(\tau)$.
% every occurrence of tau in w implies an occurence of merges(tau) in u and
% vice versa
\end{proof}

\subsection{Note on compositions over $\Zk$} \label{sec:zk}

% single subsequence pattern: 123,122
% set of subsequence patterns: too complicated? or similar?
% generalized patterns: N/A
% POP: peak (others simplify to 122 or 212)

The problem of counting compositions over a group that avoid a subsequence
pattern has not been addressed in prior literature, but was suggested in
\cite{abelian}.
Here we present a general technique illustrated for the pattern
$1' \hype 2 \hype 1''$.

\begin{proposition}
Let $P_k^{\langle a \rangle}(y)$ be the generating function for compositions
of $a$ over $\mathbb{Z}_k$ avoiding
the pattern set $\{1\hype 3\hype 2, 2\hype 3\hype 1, 1\hype 2\hype 1\}$
(alternatively the partially ordered pattern $1' \hype 2 \hype 1''$), where $y$
marks length.
Then
\[ [y^m]P_k^{\langle a \rangle}(y) = \frac{1}{k} m^{2k-2} + O(m^{2k-3}),
  \qquad m \to \infty. \]
\end{proposition}
\begin{proof}
Let $P_k(x,y)$ be the generating function for integer compositions over the
part set $[k]$ avoiding $1' \hype 2 \hype 1''$, where $x$ marks total and
$y$ marks length.
Example~5.62 in \cite{cofc} provides the expression
\[ P_k(x,y) = \frac{1}{\prod_{d =1}^k (1-x^d y)^2} - \sum_{d =1}^k
  \frac{x^d y}{\prod_{b=d}^k (1-x^b y)^2}. \]

The multisection formula
\cite[Ex.\ 1.1.9]{goulden2004combinatorial}
for power series $F(z) = \sum_n f_n z^n$ is
\[ \frac{1}{k} \sum_{j = 0}^{k-1} e^{-2\pi i ja/k} F(e^{2 \pi ij/k} z)
  = \sum_{n \equiv a \pmod{k}} f_n z^n. \]

Using the multisection formula we have
\begin{align*}
P^{\langle a \rangle}_k(y)
=& \frac{1}{k}\sum_{c=1}^k e^{-2\pi i c a/k } P_k(e^{2\pi i c/k},y) \\
=& \frac{1}{k}\sum_{c=1}^k e^{-2\pi i c a/k }
  \left( \frac{1}{\prod_{d =1}^k (1-e^{2\pi i cd/k} y)^2} - \sum_{d =1}^k
  \frac{e^{2\pi i cd/k} y}{\prod_{b=d}^k (1-e^{2\pi i cb/k} y)^2}
  \right).
\end{align*}
Thus $P^{\langle a \rangle}_k(y)$ is rational.
We claim that any pole of $P^{\langle a \rangle}_k(y)$ other than $y=1$ has
order at most $2(k-1)$.
All poles other than $y=1$ would come from terms where $c \neq k$.
If $c \neq k$, then for $d=1, \ldots, k$, the factor
$e^{2\pi i cd/k}$ takes on at least $2$ different values e.g.\ at $d=1$ and
$d=k$, so we conclude the claim.
We claim the pole at $y=1$ has order $2k-1$.
Such poles can only come from the term $c=k$.
This term is
\begin{align*}
&\frac{1}{\prod_{d =1}^k (1- y)^2} - \sum_{d =1}^k
  \frac{y}{\prod_{b=d}^k (1-y)^2}\\
=& \frac{1}{\prod_{d =1}^k (1- y)^2} - \frac{y}{\prod_{d =1}^k (1- y)^2}
 - \sum_{d =2}^k \frac{y}{\prod_{b=d}^k (1-y)^2} \\
=& \frac{1}{(1-y)^{2k-1}} + O((1-y)^{-2(k-1)}).
\end{align*}
We conclude the second claim and the proposition follows
by applying Theorem \ref{thm:ratgf}.
\end{proof}

Table \ref{tab:set} shows initial coefficients of $P^{\langle a \rangle}_4(y)$.

\begin{table}
\centering
\begin{tabular}[c]{|c|rrrrrrrrrrr|}
\hline
\diagbox{$a$}{$m$} & $0$ & $1$ & $2$ & $3$ & $4$ & $5$ & $6$
  & $7$ & $8$ & $9$ & $10$ \\ \hline
0& 1 & 1 & 4 & 12 & 32 & 71 & 150 & 287 & 517 & 877 & 1436 \\
1& 0 & 1 & 4 & 13 & 34 & 76 & 154 & 294 & 526 & 893 & 1450 \\
2& 0 & 1 & 4 & 12 & 32 & 74 & 152 & 288 & 518 & 883 & 1440 \\
3& 0 & 1 & 4 & 13 & 32 & 75 & 154 & 294 & 522 & 891 & 1450 \\\hline
\end{tabular}
\caption{Counts of $m$-compositions of $a$ over $\mathbb{Z}_4$ avoiding
  $1' \hype 2 \hype 1''$.}
\label{tab:set}
\end{table}

Similar analysis can potentially be performed e.g.\ for compositions avoiding
a length-3 permutation pattern as
enumerated in \cite[Theorem~5.7]{cofc} % given to Jason
and those avoiding the pattern $1 \hype 1 \hype 2$ as enumerated in
Theorem~5.13 in \cite[p.~139]{cofc}.

\begin{remark}
Note that if we wanted to count compositions mod $k$ using a recurrence
relation that recurses on $k$, we have the following problem.
While we can create a
composition over $[k]$ by creating one over $[k-1]$ and inserting some copies of
$k$, we must know the total of the composition over $[k-1]$ as a value
mod $k$, not mod $k-1$.
So the alphabet and modulus have to be tracked separately.
%modifying recurrence relation to keep track of total $\pmod k$
%is problematic.
%This is because in recurrence relations on $k$,
%we want to keep track of the total $\pmod k$, but count words on
%$[k-1]-1, [k-2]-1$, et cetera; i.e.\ the alphabet size does not match
%the modulus as soon as we recurse.
%So really we get a system of recurrence relations on $m$, parametrized by $k$.
\end{remark}

\section{Conclusion} \label{sec:conc}

As a conclusion, this section mentions some relevant problems which are as yet
unsolved. The book \cite{cofc} contains a variety of proposed research problems
many of which are also unsolved.

% offdef
The paper \cite{automata} uses finite automata to count words avoiding a
subsequence pattern.
The only cycles in the automata are loops and so an asymptotic form for
the number of accepted words is obtained directly.
Rather than using the technique of \S \ref{sec:zk}, it may be possible to
count compositions over a finite group that avoid a subsequence pattern by
combining the techniques of \cite{automata} and \S \ref{sec:localpaths}.

%\paragraph*{Finitely generated (infinite) groups}
% code

Consider locally restricted compositions where the parts come from a finite
generating set of an infinite group.
In the framework of \S \ref{sec:localpaths},
the base digraph $\graphf{D}$ is finite but the derived digraph
$\graphf{D}_\times$ is infinite.
Take for example the infinite group $\mathbb{Z}$ with generating set
$\{-1,0,1\}$.
Unrestricted compositions over $\{-1, 0, 1\}$ with total, say, $0$ no longer
form a regular language but do form a context-free language, recognizable
by a pushdown automaton.
There may be difficulty from finitely generated non-abelian groups, however,
due to the fact that the word problem is undecidable and therefore not
even context-free.
For finitely-generated abelian groups, it is possible that the number of
locally restricted compositions of $a$ is asymptotically independent of
$a$ but it is no longer possible for each total $a$ to be asymptotically
equally likely since the group is infinite.
The recent paper \cite{dahmani2018growth} explores other problems involving
finitely-generated groups and enumeration.

%count is $\sim \frac{2^{m+\frac{3}{2}} }{\sqrt{\pi m}}$, $m \to \infty$
%note no dependence on $s$ but irrational counting sequence

%\paragraph*{Weighted (infinite) groups}

Suppose we have a group $G$ which is infinite but we also have a weight
function $W: G \to \mathbb{Z}_{>0}$.
As long as each preimage $W^{-1}(n)$ is finite, we can define the number of
locally restricted compositions of $a$ over $G$ with a given total weight $n$.
It is plausible that approaches in the above sections and \cite{infinite}
can be applied to this counting problem.

%\paragraph*{Circular integer compositions}

Circular integer compositions (unlabeled cycles weighted by positive integers)
appeared in the early paper \cite{servedio1995bijective}.
More recently, the enumerative study of locally restricted circular integer
compositions has progressed in \cite{hadjicostas2016cyclic,
hadjicostas2017cyclic} which study Carlitz compositions and restrictions on
the set of parts.
The conclusion section of \cite{hadjicostas2017cyclic} suggests expanding to
general local restrictions as in \cite{infinite}.
Other recent work \cite{wang2018meghann} has looked at part sizes in
circular integer compositions.

%\paragraph*{Colored compositions}

We can also consider a set of \emph{colored parts} $(j, c)$, where
$j \in \mathbb{Z}_{>0}$ and $c$ comes from a set of colors (with no particular
algebraic structure).
A \emph{colored integer composition} of $n$ is a sequence
$((j_1, c_1), \ldots, (j_m, c_m))$ where $\sum_i j_i = n$.
Questions about restricted integer compositions can be asked again for colored
integer compositions and for colored versions of compositions over a finite
group.
Some results and open problems mentioned in \cite{bg2018} are relevant.

%\paragraph*{\enquote{Explicit} counting for local patterns}

In the sections above on local restrictions, we focus on \enquote{implicit}
results that cover a wide range of particular restrictions.
What this does not provide is a \enquote{simple} formula for exact counts,
or explicit constants within asymptotic expressions.
So there is the possibility of finding (more) explicit, but less general,
expressions to complement our implicit ones.
For a deep discussion of the meaning of explicitness in enumeration,
see \cite{answer, pak}.

%\paragraph*{Longest runs in words avoiding subsequence patterns}

In \S\ref{sec:subseq} we count $k$-ary words of length $m$ avoiding a
subsequence pattern set $T$.
A further parameter can be tracked, namely the length $p$ of the longest
contiguous run of a single letter.
This problem for subword pattern avoidance is addressed in
\cite{bender2016locally} but there is no previous work for subsequence
patterns.
It is quickly deduced that this is roughly equivalent to counting words
that avoid $T$ and also avoid the subword pattern $1^p$.
Combining subsequence and subword pattern avoidance presents a challenge.

%\paragraph*{Restricted-preimage mappings}

Finite mappings from a set to itself correspond to functional digraphs, which
have a well-known structure \cite{flajolet1989random}.
Research such as \cite{arney1982random} has enumerated functional digraphs
with a kind of local restriction: the indegree of each vertex (a.k.a.\ the
number of its preimages) must lie within a fixed set $\Xi$.
More recent papers such as \cite{daniel} consider the distribution
of the least common multiple $T$ and product $B$ of the cycle lengths in
restricted functional digraphs, for particular $\Xi$.
These values $T, B$ are related to the sequence of iterations of the mapping.
Problems remain such as expanding to more general $\Xi$.

%Also, analyze $\bm{T}$ rather than the log (since even the expectation will
%be quite different from just an application of $\log$, so at least analyze
%that).
%does $k$ have to be a function of $n$ going to infinity?

\ifofficial
\else
  \vspace{2\baselineskip}
  \textbf{Acknowledgements.}
The author thanks Zhicheng Gao, and also Toufik Mansour, Mike Newman, Daniel
Panario, and Michiel Smid for suggestions and corrections.
\fi

\cleardoublepage
% https://tex.stackexchange.com/a/23503/8865
\addcontentsline{toc}{section}{References}

\bibliography{bib}
\bibliographystyle{plain}
\end{document}